\documentclass[12pt, draft]{article}
\usepackage[english]{babel}

\usepackage{amsmath}
\usepackage{amssymb,amsthm}

\usepackage{mathrsfs}

\usepackage{bbm}

\numberwithin{equation}{section}

\newtheorem{theorem}{Theorem}

\newtheorem{lemma}{Lemma}[section]
\newtheorem{example}[lemma]{Example}
\newtheorem{remark}[lemma]{Remark}
\newtheorem{proposition}{Proposition}

\newtheorem{notation}[lemma]{Notation}
\newtheorem{definition}[lemma]{Definition}

\renewcommand{\leq}{\leqslant}
\renewcommand{\geq}{\geqslant}

\makeatletter
\newcommand*{\IfItalic}{%
  \ifx\f@shape\my@test@it
    \expandafter\@firstoftwo
  \else
    \expandafter\@secondoftwo
  \fi
}
\newcommand*{\my@test@it}{it}
\makeatother

\newcommand{\myae}{\IfItalic{\emph{\mbox{\ae}}}{\mbox{\ae}}}

\newcommand\XLast{}
\newcommand\YLast{}

\newcommand\Vidr[4]{
\qbezier(#1,#2)(#1,#2)(#3,#4)
\renewcommand\XLast{#3}
\renewcommand\YLast{#4}}

\newcommand\VidrTo[2]{\Vidr{\XLast}{\YLast}{#1}{#2}}
\newcommand\ssemg\psi
\newcommand\fkinkg\gamma

\newcommand\sdev{s}
\newcommand\gdev{\gamma}

\usepackage[all]{xy}

\usepackage{amscd}

 \DeclareMathOperator{\Sat}{SAT}

\newcommand{\SAT}{\ensuremath{\Sat}}

\begin{document}

\begin{center}
{\Large Classification of commutative pairs of surjective maps of
interval, one of which is unimodal.}

{\large Makar Plakhotnyk\\
University of S\~ao Paulo, Brazil.\\

makar.plakhotnyk@gmail.com}\\
\end{center}

\begin{abstract}
We introduce here a classification of unimodal maps $[0,
1]\rightarrow [0, 1]$, which commute with piecewise linear
surjective maps $[0, 1]\rightarrow [0, 1]$.

Remind that if continuous piecewise linear unimodal map $g$
commutes with a non-constant piecewise linear map $\psi$, which is
not an iteration of $g$, then $g$ is topologically conjugated with
the tent map by piecewise linear conjugacy.

We use the obtained classification to illustrate the mentioned
fact.~\footnote{AMS subject classification:
%01A60, % History and biography: 20th century;
%37-03, % Dynamical systems and
       %   ergodic theory: Historical
37E05  % Maps of the interval (piecewise
       %   continuous, continuous, smooth)
}~
\end{abstract}

\section{Introduction}

We will devote this article to continuous piecewise linear
functions from an interval $[0, 1]$ to itself. For a finite
sequence of points $A_1,\ldots, A_k\in [0, 1]\times [0, 1]$ write
$$ g:\, A_1\rightarrow A_2\rightarrow \ldots \rightarrow A_k
$$ for a continuous function, whose graph passes through $A_1,\ldots,
A_k$, and does not have other kinks. For example,  $$f:\, (0,
0)\rightarrow (1/2, 1)\rightarrow (1, 0)$$ is the classical tent
map $f:\, x\mapsto 1-|1-2x|$.

We call the map $g: [0, 1]\rightarrow [0, 1]$ \textbf{unimodal},
if it can be written in the form
\begin{equation}\label{eq:1.1} g(x) = \left\{
\begin{array}{ll}g_l(x),& 0\leq x\leq
v,\\
g_r(x), & v\leq x\leq 1,
\end{array}\right.
\end{equation} where %
$v\in (0,\, 1)$ is a parameter, the function $g_l$ is increasing,
the function $g_r$ is decreasing, and $$g(0)=g(1)=1-g(v)=0.$$

We will say that unimodal map $g$ commutes with a map $\psi:\, [0,
1]\rightarrow [0, 1]$, if
$$ g\circ \psi = \psi\circ g.
$$

Remind that maps $f,g:\, [0,1]\rightarrow [0, 1]$ are called
topologically equivalent, if there is a continuous invertible map
$h: [0,1]\rightarrow [0, 1]$ such that
\begin{equation}\label{eq:1.2} h\circ f = g\circ h.\end{equation}

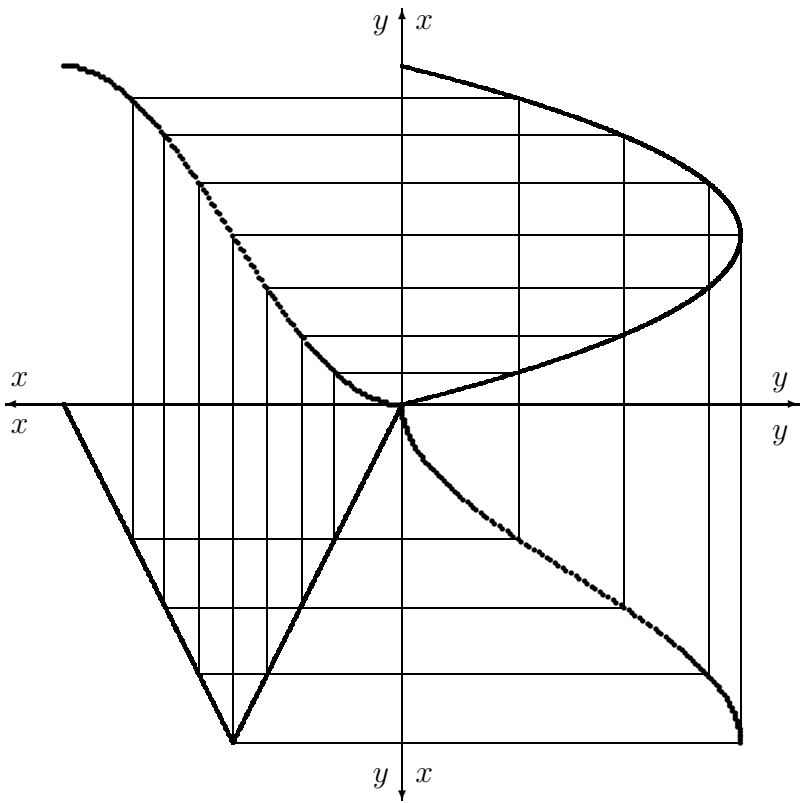
\begin{figure}[htbp]
\begin{center}
\begin{picture}(260,280)
\put(143,143){\vector(-1,0){150}} \put(143,143){\vector(1,0){150}}
\put(143,143){\vector(0,1){150}} \put(143,143){\vector(0,-1){150}}

\linethickness{0.4mm} \qbezier(143,143)(399,207)(143,271)
\qbezier(15,143)(47,79)(79,15) \qbezier(79,15)(111,79)(143,143)
\linethickness{0.1mm} \put(-5,150){$x$} \put(-5,132){$x$}

\put(132,285){$y$} \put(148,285){$x$} \put(283,150){$y$}
\put(283,130){$y$} \put(132,0){$y$} \put(148,0){$x$}

\put(15,271){\circle*{2}} \put(16,271){\circle*{2}}
\put(17,271){\circle*{2}} \put(18,271){\circle*{2}}
\put(19,271){\circle*{2}} \put(20,271){\circle*{2}}
\put(21,270){\circle*{2}} \put(22,270){\circle*{2}}
\put(23,270){\circle*{2}} \put(24,269){\circle*{2}}
\put(25,269){\circle*{2}} \put(26,269){\circle*{2}}
\put(27,268){\circle*{2}} \put(28,268){\circle*{2}}
\put(29,267){\circle*{2}} \put(30,267){\circle*{2}}
\put(31,266){\circle*{2}} \put(32,266){\circle*{2}}
\put(33,265){\circle*{2}} \put(34,264){\circle*{2}}
\put(35,263){\circle*{2}} \put(36,263){\circle*{2}}
\put(37,262){\circle*{2}} \put(38,261){\circle*{2}}
\put(39,260){\circle*{2}} \put(40,259){\circle*{2}}
\put(41,258){\circle*{2}} \put(42,257){\circle*{2}}
\put(43,256){\circle*{2}} \put(44,255){\circle*{2}}
\put(45,254){\circle*{2}} \put(46,253){\circle*{2}}
\put(47,252){\circle*{2}} \put(48,251){\circle*{2}}
\put(49,250){\circle*{2}} \put(50,249){\circle*{2}}
\put(51,248){\circle*{2}} \put(52,246){\circle*{2}}
\put(53,245){\circle*{2}} \put(54,244){\circle*{2}}
\put(55,243){\circle*{2}} \put(56,241){\circle*{2}}
\put(57,240){\circle*{2}} \put(58,239){\circle*{2}}
\put(59,237){\circle*{2}} \put(60,236){\circle*{2}}
\put(61,234){\circle*{2}} \put(62,233){\circle*{2}}
\put(63,231){\circle*{2}} \put(64,230){\circle*{2}}
\put(65,229){\circle*{2}} \put(66,227){\circle*{2}}
\put(67,226){\circle*{2}} \put(68,224){\circle*{2}}
\put(69,223){\circle*{2}} \put(70,221){\circle*{2}}
\put(71,219){\circle*{2}} \put(72,218){\circle*{2}}
\put(73,216){\circle*{2}} \put(74,215){\circle*{2}}
\put(75,213){\circle*{2}} \put(76,212){\circle*{2}}
\put(77,210){\circle*{2}} \put(78,209){\circle*{2}}
\put(79,207){\circle*{2}} \put(80,205){\circle*{2}}
\put(81,204){\circle*{2}} \put(82,202){\circle*{2}}
\put(83,201){\circle*{2}} \put(84,199){\circle*{2}}
\put(85,198){\circle*{2}} \put(86,196){\circle*{2}}
\put(87,195){\circle*{2}} \put(88,193){\circle*{2}}
\put(89,191){\circle*{2}} \put(90,190){\circle*{2}}
\put(91,188){\circle*{2}} \put(92,187){\circle*{2}}
\put(93,185){\circle*{2}} \put(94,184){\circle*{2}}
\put(95,183){\circle*{2}} \put(96,181){\circle*{2}}
\put(97,180){\circle*{2}} \put(98,178){\circle*{2}}
\put(99,177){\circle*{2}} \put(100,175){\circle*{2}}
\put(101,174){\circle*{2}} \put(102,173){\circle*{2}}
\put(103,171){\circle*{2}} \put(104,170){\circle*{2}}
\put(105,169){\circle*{2}} \put(106,168){\circle*{2}}
\put(107,166){\circle*{2}} \put(108,165){\circle*{2}}
\put(109,164){\circle*{2}} \put(110,163){\circle*{2}}
\put(111,162){\circle*{2}} \put(112,161){\circle*{2}}
\put(113,160){\circle*{2}} \put(114,159){\circle*{2}}
\put(115,158){\circle*{2}} \put(116,157){\circle*{2}}
\put(117,156){\circle*{2}} \put(118,155){\circle*{2}}
\put(119,154){\circle*{2}} \put(120,153){\circle*{2}}
\put(121,152){\circle*{2}} \put(122,151){\circle*{2}}
\put(123,151){\circle*{2}} \put(124,150){\circle*{2}}
\put(125,149){\circle*{2}} \put(126,148){\circle*{2}}
\put(127,148){\circle*{2}} \put(128,147){\circle*{2}}
\put(129,147){\circle*{2}} \put(130,146){\circle*{2}}
\put(131,146){\circle*{2}} \put(132,145){\circle*{2}}
\put(133,145){\circle*{2}} \put(134,145){\circle*{2}}
\put(135,144){\circle*{2}} \put(136,144){\circle*{2}}
\put(137,144){\circle*{2}} \put(138,143){\circle*{2}}
\put(139,143){\circle*{2}} \put(140,143){\circle*{2}}
\put(141,143){\circle*{2}} \put(142,143){\circle*{2}}
\put(143,143){\circle*{2}}

\put(143,143){\circle*{2}} \put(143,142){\circle*{2}}
\put(143,141){\circle*{2}} \put(143,140){\circle*{2}}
\put(143,139){\circle*{2}} \put(143,138){\circle*{2}}
\put(144,137){\circle*{2}} \put(144,136){\circle*{2}}
\put(144,135){\circle*{2}} \put(145,134){\circle*{2}}
\put(145,133){\circle*{2}} \put(145,132){\circle*{2}}
\put(146,131){\circle*{2}} \put(146,130){\circle*{2}}
\put(147,129){\circle*{2}} \put(147,128){\circle*{2}}
\put(148,127){\circle*{2}} \put(148,126){\circle*{2}}
\put(149,125){\circle*{2}} \put(150,124){\circle*{2}}
\put(151,123){\circle*{2}} \put(151,122){\circle*{2}}
\put(152,121){\circle*{2}} \put(153,120){\circle*{2}}
\put(154,119){\circle*{2}} \put(155,118){\circle*{2}}
\put(156,117){\circle*{2}} \put(157,116){\circle*{2}}
\put(158,115){\circle*{2}} \put(159,114){\circle*{2}}
\put(160,113){\circle*{2}} \put(161,112){\circle*{2}}
\put(162,111){\circle*{2}} \put(163,110){\circle*{2}}
\put(164,109){\circle*{2}} \put(165,108){\circle*{2}}
\put(166,107){\circle*{2}} \put(168,106){\circle*{2}}
\put(169,105){\circle*{2}} \put(170,104){\circle*{2}}
\put(171,103){\circle*{2}} \put(173,102){\circle*{2}}
\put(174,101){\circle*{2}} \put(175,100){\circle*{2}}
\put(177,99){\circle*{2}} \put(178,98){\circle*{2}}
\put(180,97){\circle*{2}} \put(181,96){\circle*{2}}
\put(183,95){\circle*{2}} \put(184,94){\circle*{2}}
\put(185,93){\circle*{2}} \put(187,92){\circle*{2}}
\put(188,91){\circle*{2}} \put(190,90){\circle*{2}}
\put(191,89){\circle*{2}} \put(193,88){\circle*{2}}
\put(195,87){\circle*{2}} \put(196,86){\circle*{2}}
\put(198,85){\circle*{2}} \put(199,84){\circle*{2}}
\put(201,83){\circle*{2}} \put(202,82){\circle*{2}}
\put(204,81){\circle*{2}} \put(205,80){\circle*{2}}
\put(207,79){\circle*{2}} \put(209,78){\circle*{2}}
\put(210,77){\circle*{2}} \put(212,76){\circle*{2}}
\put(213,75){\circle*{2}} \put(215,74){\circle*{2}}
\put(216,73){\circle*{2}} \put(218,72){\circle*{2}}
\put(219,71){\circle*{2}} \put(221,70){\circle*{2}}
\put(223,69){\circle*{2}} \put(224,68){\circle*{2}}
\put(226,67){\circle*{2}} \put(227,66){\circle*{2}}
\put(229,65){\circle*{2}} \put(230,64){\circle*{2}}
\put(231,63){\circle*{2}} \put(233,62){\circle*{2}}
\put(234,61){\circle*{2}} \put(236,60){\circle*{2}}
\put(237,59){\circle*{2}} \put(239,58){\circle*{2}}
\put(240,57){\circle*{2}} \put(241,56){\circle*{2}}
\put(243,55){\circle*{2}} \put(244,54){\circle*{2}}
\put(245,53){\circle*{2}} \put(246,52){\circle*{2}}
\put(248,51){\circle*{2}} \put(249,50){\circle*{2}}
\put(250,49){\circle*{2}} \put(251,48){\circle*{2}}
\put(252,47){\circle*{2}} \put(253,46){\circle*{2}}
\put(254,45){\circle*{2}} \put(255,44){\circle*{2}}
\put(256,43){\circle*{2}} \put(257,42){\circle*{2}}
\put(258,41){\circle*{2}} \put(259,40){\circle*{2}}
\put(260,39){\circle*{2}} \put(261,38){\circle*{2}}
\put(262,37){\circle*{2}} \put(263,36){\circle*{2}}
\put(263,35){\circle*{2}} \put(264,34){\circle*{2}}
\put(265,33){\circle*{2}} \put(266,32){\circle*{2}}
\put(266,31){\circle*{2}} \put(267,30){\circle*{2}}
\put(267,29){\circle*{2}} \put(268,28){\circle*{2}}
\put(268,27){\circle*{2}} \put(269,26){\circle*{2}}
\put(269,25){\circle*{2}} \put(269,24){\circle*{2}}
\put(270,23){\circle*{2}} \put(270,22){\circle*{2}}
\put(270,21){\circle*{2}} \put(271,20){\circle*{2}}
\put(271,19){\circle*{2}} \put(271,18){\circle*{2}}
\put(271,17){\circle*{2}} \put(271,16){\circle*{2}}
\put(271,15){\circle*{2}}

\put(105,143){\line(0,1){26}} \put(105,169){\line(1,0){122}}
\put(227,169){\line(0,-1){103}} \put(105,66){\line(0,1){77}}

\put(53,143){\line(0,1){102}} \put(53,245){\line(1,0){174}}
\put(227,245){\line(0,-1){179}} \put(227,66){\line(-1,0){174}}
\put(53,66){\line(0,1){77}}

\put(79,143){\line(0,1){64}} \put(79,207){\line(1,0){192}}
\put(271,207){\line(0,-1){192}} \put(271,15){\line(-1,0){192}}
\put(79,15){\line(0,1){128}}

\put(92,143){\line(0,1){44}} \put(92,187){\line(1,0){167}}
\put(259,187){\line(0,-1){146}} \put(92,41){\line(0,1){102}}

\put(66,143){\line(0,1){84}} \put(66,227){\line(1,0){193}}
\put(259,227){\line(0,-1){186}} \put(259,41){\line(-1,0){193}}
\put(66,41){\line(0,1){102}}

\put(117,143){\line(0,1){12}} \put(117,155){\line(1,0){70}}
\put(187,155){\line(0,-1){63}} \put(187,92){\line(-1,0){70}}
\put(117,92){\line(0,1){51}}

\put(41,143){\line(0,1){116}} \put(41,259){\line(1,0){146}}
\put(187,259){\line(0,-1){167}} \put(187,92){\line(-1,0){146}}
\put(41,92){\line(0,1){51}}

\end{picture}
\end{center}
\caption{Graphs, which illustrate topological conjugateness}
\label{fig:01}
\end{figure}

Topological conjugation of maps can be illustrated as on
Fig.~\ref{fig:01} (see~\cite{Skufca}). Suppose that~\eqref{eq:1.2}
holds for maps $f, g, h:\, [0, 1]\rightarrow [0, 1]$. Divide the
plane by 4 parts by 2 ``coordinate lines'', plot the graph of $h$
at left top and right bottom quarter of the plane, plot $f$ at
left bottom, and $g$ at right top. Thus, each of compositions
$h\circ f$ and $g\circ h$ acts from the ray, going to the left, to
the ray, going to the right. The equality~\eqref{eq:1.2} means
that the result of the action is independent on a choice of the
path (the top, or the bottom one).

\section{The main construction}

Suppose that $g$ is a unimodal map surjective map such that $g(0)
=g(1)=0$. We will say that a continuous non-constant map $\psi$ is
a non-trivial commutator of $g$, if $g$ commutes with $\psi$ and
$\psi$ is a not an iteration of $g$.

\begin{lemma}\cite[Lemmas~3.3 and~3.5]{1808.03622.2}\label{lema:2.1}
Let $\psi$ be a non-trivial commutator of $g$. Then $\psi(0)=0$
and for every maximal interval $I$ of monotonicity of $\psi$ we
have that $\psi(I) = [0, 1]$.
\end{lemma}

Let
$$f:\, (0, 0)\rightarrow (1/2, 1)\rightarrow (1, 0)$$
be the tent map and $$\xi:\, (0, 0)\rightarrow (1/3, 1)\rightarrow
(2/3, 0)\rightarrow (1, 1).$$ These maps commute, because it is
easy to see that $$\begin{array}{cc}
f\circ \xi = \xi\circ f :\\
(0, 0)\rightarrow \left(\frac{1}{6}, 1\right)\rightarrow
\left(\frac{1}{3}, 0\right)\rightarrow \left(\frac{1}{2},
1\right)\rightarrow \left(\frac{2}{3}, 0\right)\rightarrow
\left(\frac{5}{6}, 1\right)\rightarrow (1, 0).
\end{array}
$$

Analogously to Figure~\ref{fig:01}, plot $f$ and $\xi$ (see
Fig.~\ref{fig:02}a). Call axes $x$, $\xi$, $f$ and $y$, as it is
done at Fig.~\ref{fig:02}a.

\begin{figure}[htbp]
\begin{minipage}[h]{0.3\linewidth}
\begin{center}
\begin{picture}(130,130)

\put(65,65){\vector(1,0){65,65}} \put(65,65){\vector(-1,0){65,65}}
\put(65,65){\vector(0,1){65,65}} \put(65,65){\vector(0,-1){65,65}}

\Vidr{65}{65}{125}{95} \VidrTo{65}{125}

\Vidr{65}{65}{45}{125} \VidrTo{25}{65} \VidrTo{5}{125}

\Vidr{65}{65}{35}{5} \VidrTo{5}{65}

\Vidr{65}{65}{125}{45} \VidrTo{65}{25} \VidrTo{125}{5}

\put(3,70){$x$} \put(55,5){$f$} \put(55,122){$\xi$}
\put(125,73){$y$}

\end{picture}
\vskip 3mm \centerline{a) }\end{center}
\end{minipage}
\hfill
\begin{minipage}[h]{0.3\linewidth}
\begin{center}
\begin{picture}(130,130)

\put(65,65){\vector(1,0){65,65}} \put(65,65){\vector(-1,0){65,65}}
\put(65,65){\vector(0,1){65,65}} \put(65,65){\vector(0,-1){65,65}}

\put(3,70){$x$} \put(55,5){$f$} \put(55,122){$\xi$}
\put(125,73){$y$}

%\linethickness{0.4mm}

\put(20,100){$(\alpha, \beta)_{\lrcorner}$}

\put(80,100){$_{\llcorner}(\alpha, \beta)$}

\put(20,30){$(\alpha, \beta)^{\urcorner}$}

\put(80,30){$^{\ulcorner}(\alpha, \beta)$}

\linethickness{0.1mm}

\end{picture}
\vskip 3mm \centerline{b) }\end{center}
\end{minipage}
\hfill
\begin{minipage}[h]{0.3\linewidth}
\begin{center}
\begin{picture}(130,130)

\put(65,65){\vector(1,0){65,65}} \put(65,65){\vector(-1,0){65,65}}
\put(65,65){\vector(0,1){65,65}} \put(65,65){\vector(0,-1){65,65}}

\Vidr{65}{65}{125}{95} \VidrTo{65}{125}

\Vidr{65}{65}{45}{125} \VidrTo{25}{65} \VidrTo{5}{125}

\Vidr{65}{65}{35}{5} \VidrTo{5}{65}

\Vidr{65}{65}{125}{45} \VidrTo{65}{25} \VidrTo{125}{5}

\linethickness{0.4mm}

\Vidr{30}{15}{30}{80} \Vidr{95}{15}{95}{80} \Vidr{30}{80}{95}{80}
\Vidr{30}{15}{95}{15}

\linethickness{0.1mm}

%\put(3,70){$x$} \put(55,5){$g$} \put(55,122){$\psi$}
%\put(125,73){$y$}

\end{picture}
\vskip 3mm \centerline{c) }\end{center}
\end{minipage}
\hfill \caption{} \label{fig:02}
\end{figure}

We will write $(\alpha, \beta)_{\lrcorner}$ for a point in the
quadrant $x\times \xi$, meaning that $x= \alpha$ and $\xi =\beta$.

Analogously, $_{\llcorner}(\alpha, \beta)$ will mean that $\xi =
\alpha$ and $y = \beta$ in the quadrant $\xi\times y$,

$(\alpha, \beta)^{\urcorner}$ for $x =\alpha$ in $f =\beta$ in the
quadrant $x\times f$, and, finally,

$^{\ulcorner}(\alpha, \beta)$ for $f = \alpha$ and $y =\beta$ in
the quadrant $f\times y$. See Fig.~\ref{fig:02}b.

For any $x_0\in [0, 1]$ we will call a {\it single trajectory} the
set of lines $x =x_0$, $\xi = \xi(x_0)$, $f= f(x_0)$ and $y =
(f\circ \xi)(x_0)$. Remark that the definition of a single
trajectory does not demand neither the commutativity of the maps
$f$ and $\xi$, nor the equality
\begin{equation}\label{eq:2.1}(f\circ \xi)(x_0) = (\xi\circ
f)(x_0).\end{equation} We will say that the trajectory of $x_0$
does not contradict to the commutativity of $f$ and $\xi$,
if~\eqref{eq:2.1} holds.

In other words, we will \textbf{call} a quadruple of lines $x =
x_0$, $\xi =\xi_0$, $f =f_0$ and $y =y_0$ a single trajectory, if
all their intersections belong to correspond graph of the plot,
mentioned above. An example of a single trajectory is given at
Figure~\ref{fig:02}c. We will say that

Conjugate now the maps $f$ and $\xi$ by $$ h:\, (0, 0)\rightarrow
\left(\frac{1}{2},  \frac{3}{4}\right) \rightarrow (1,1).
$$ Clearly, the
resulting maps $g = h\circ f\circ h^{-1}$ and $\psi = h\circ
\xi\circ h^{-1}$ commute, because $f$ and $\xi$ commute. The
graphs of $\widetilde{f}$ and $\psi$ are given at
Fig.~\ref{fig:03}. By direct calculations,
\begin{equation}\label{eq:2.2}
g:\, (0, 0)\rightarrow (3/8, 3/4)\rightarrow (3/4, 1)\rightarrow
(7/8, 3/4)\rightarrow (1, 0)
\end{equation} and
\begin{equation}\label{eq:2.3}
\psi:\, (0, 0)\rightarrow (1/4, 3/4)\rightarrow (1/2,
1)\rightarrow (3/4, 3/4)\rightarrow (5/6, 0)\rightarrow (11/12,
3/4)\rightarrow (1, 1). \end{equation}

\begin{figure}[htbp]
\begin{center}

\begin{picture}(300,300)

\put(150,150){\vector(1,0){150}} \put(150,150){\vector(-1,0){150}}
\put(150,150){\vector(0,1){150}} \put(150,150){\vector(0,-1){150}}

\linethickness{0.4mm} \Vidr{150}{150}{120}{240}
\put(120,240){\circle*{6}}\VidrTo{90}{270}
\put(90,270){\circle*{6}}\VidrTo{60}{240}
\put(60,240){\circle*{6}}\VidrTo{50}{150}
\put(50,150){\circle*{6}}\VidrTo{40}{240}
\put(40,240){\circle*{6}}\VidrTo{30}{270}

\Vidr{150}{150}{240}{195}
\put(240,195){\circle*{6}}\VidrTo{270}{240}
\put(270,240){\circle*{6}}\Vidr{270}{240}{240}{255}
\put(240,255){\circle*{6}}\VidrTo{150}{270}

\Vidr{150}{150}{240}{120}
\put(240,120){\circle*{6}}\VidrTo{270}{90}
\put(270,90){\circle*{6}}\VidrTo{240}{60}
\put(240,60){\circle*{6}}\VidrTo{150}{50}
\put(150,50){\circle*{6}}\VidrTo{240}{40}
\put(240,40){\circle*{6}}\VidrTo{270}{30}

\Vidr{150}{150}{105}{60} \put(105,60){\circle*{6}}\VidrTo{60}{30}
\put(60,30){\circle*{6}}\Vidr{60}{30}{45}{60}
\put(45,60){\circle*{6}}\VidrTo{30}{150}\linethickness{0.1mm}
\end{picture}
\end{center}\caption{}
\label{fig:03}
\end{figure}

Is is possible to verify their commutativity of maps $g$ and
$\psi$ directly from the plot of Fig.~\ref{fig:03}. We will
describe the method of this verification below.

For every break point of each of two (four) function, from
Fig.~\ref{fig:03} plot all the single trajectories, which pass
through this point (see Fig.~\ref{fig:04}). Verify that each of
the constructed single trajectories does not contradict to the
commutativity of the maps $g$ and $\psi$. Claim, that this
verification of enough to conclude that $g$ indeed commutes with
$\psi$.

\begin{figure}
\begin{center}

\begin{picture}(300,300)

\put(150,150){\vector(1,0){150}} \put(150,150){\vector(-1,0){150}}
\put(150,150){\vector(0,1){150}} \put(150,150){\vector(0,-1){150}}

\linethickness{0.4mm}

\Vidr{150}{150}{120}{240}
\put(120,240){\circle*{6}}\VidrTo{90}{270}
\put(90,270){\circle*{6}}\VidrTo{60}{240}
\put(60,240){\circle*{6}}\VidrTo{50}{150}
\put(50,150){\circle*{6}}\VidrTo{40}{240}
\put(40,240){\circle*{6}}\VidrTo{30}{270}

\Vidr{150}{150}{240}{195}
\put(240,195){\circle*{6}}\VidrTo{270}{240}
\put(270,240){\circle*{6}}\Vidr{270}{240}{240}{255}
\put(240,255){\circle*{6}}\VidrTo{150}{270}

\Vidr{150}{150}{240}{120}
\put(240,120){\circle*{6}}\VidrTo{270}{90}
\put(270,90){\circle*{6}}\VidrTo{240}{60}
\put(240,60){\circle*{6}}\VidrTo{150}{50}
\put(150,50){\circle*{6}}\VidrTo{240}{40}
\put(240,40){\circle*{6}}\VidrTo{270}{30}

\Vidr{150}{150}{105}{60} \put(105,60){\circle*{6}}\VidrTo{60}{30}
\put(60,30){\circle*{6}}\Vidr{60}{30}{45}{60}
\put(45,60){\circle*{6}}\VidrTo{30}{150}

\linethickness{0.1mm}

\put(150,270){\line(-1,0){120}} \put(30,270){\line(0,-1){120}}

\put(270,240){\line(-1,0){230}} \put(40,240){\line(0,-1){150}}
\put(40,90){\line(1,0){230}}

\put(240,120){\line(-1,0){205}} \put(35,120){\line(0,1){135}}
\put(35,255){\line(1,0){205}}

\put(240,195){\line(-1,0){195}} \put(45,195){\line(0,-1){135}}
\put(45,60){\line(1,0){195}}

\put(50,150){\line(0,-1){100}} \put(50,50){\line(1,0){100}}

\put(60,240){\line(0,-1){210}} \put(60,30){\line(1,0){210}}

\put(105,255){\line(0,-1){195}}

\put(75,255){\line(0,-1){215}} \put(75,40){\line(1,0){165}}

\put(90,270){\line(0,-1){220}}

\put(120,240){\line(0,-1){150}}

\put(135,195){\line(0,-1){75}}

\put(240,255){\line(0,-1){215}}

\put(270,240){\line(0,-1){210}}

\put(55,195){\line(0,-1){155}} \put(55,40){\line(1,0){185}}

\put(3,160){$x$} \put(3,137){$x$} \put(137,295){$y$}
\put(158,295){$x$} \put(290,160){$y$} \put(290,140){$y$}
\put(137,10){$y$} \put(158,10){$x$}
\end{picture}
\end{center}\caption{}
\label{fig:04}
\end{figure}

With the purpose to formalize our conclusion, we will define one
more concept. Let $g$ and $\psi$ be piecewise linear maps $[0,
1]\rightarrow [0, 1]$, and let $g$ be unimodal surjective such
that $g(0) =g(1)=0$. For any $x\in (0, 1)$ call the set of all
trajectories, related to $x$, the set of all single trajectories,
generated by $\psi^{-1}(g^{-1}(g\circ \psi)(x))$. Write for
simplicity $\SAT(x)$ for the the set of all trajectories, related
to $x$. We will say that $\SAT(x)$ does not contradict to the
commutativity of $g$ and $\psi$, if any single trajectory of
$\SAT(x)$ does not contradict to the commutativity of $g$ and
$\psi$. The examples of $\SAT$ are given at Figures~\ref{fig:05}a,
\ref{fig:05}b and~\ref{fig:05}c. Precisely, Fig.~\ref{fig:05}a
contains $\SAT\left(\frac{1}{6}\right)$ for maps $f$ and $\xi$;
Fig.~\ref{fig:05}b contains $\SAT\left(\frac{1}{2}\right)$, and
Fig.~\ref{fig:05}c contains $\SAT(0)$.

\begin{figure}[htbp]
\begin{minipage}[h]{0.3\linewidth}
\begin{center}
\begin{picture}(130,130)

\put(65,65){\vector(1,0){65,65}} \put(65,65){\vector(-1,0){65,65}}
\put(65,65){\vector(0,1){65,65}} \put(65,65){\vector(0,-1){65,65}}

\Vidr{65}{65}{125}{95} \VidrTo{65}{125}

\Vidr{65}{65}{45}{125} \VidrTo{25}{65} \VidrTo{5}{125}

\Vidr{65}{65}{35}{5} \VidrTo{5}{65}

\Vidr{65}{65}{125}{45} \VidrTo{65}{25} \VidrTo{125}{5}

\linethickness{0.4mm}

\Vidr{30}{15}{30}{80} \Vidr{95}{15}{95}{110} \Vidr{20}{80}{95}{80}
\Vidr{20}{80}{20}{35} \Vidr{20}{35}{95}{35}
\Vidr{95}{110}{10}{110} \Vidr{10}{110}{10}{55}
\Vidr{10}{55}{95}{55} \Vidr{60}{55}{60}{80} \Vidr{50}{110}{50}{35}
\Vidr{40}{110}{40}{15} \Vidr{30}{15}{95}{15}

\linethickness{0.1mm}

\put(3,70){$x$} \put(55,5){$g$} \put(55,122){$\psi$}
\put(125,73){$y$}

\end{picture}
\vskip 3mm \centerline{a) }\end{center}
\end{minipage}
\hfill
\begin{minipage}[h]{0.3\linewidth}
\begin{center}
\begin{picture}(130,130)

\put(65,65){\vector(1,0){65,65}} \put(65,65){\vector(-1,0){65,65}}
\put(65,65){\vector(0,1){65,65}} \put(65,65){\vector(0,-1){65,65}}

\Vidr{65}{65}{125}{95} \VidrTo{65}{125}

\Vidr{65}{65}{45}{125} \VidrTo{25}{65} \VidrTo{5}{125}

\Vidr{65}{65}{35}{5} \VidrTo{5}{65}

\Vidr{65}{65}{125}{45} \VidrTo{65}{25} \VidrTo{125}{5}

\linethickness{0.4mm}

\Vidr{125}{95}{15}{95} \Vidr{125}{95}{15}{95}
\Vidr{15}{95}{15}{45} \Vidr{55}{95}{55}{45} \Vidr{15}{45}{125}{45}
\Vidr{35}{95}{35}{5} \Vidr{35}{5}{125}{5} \Vidr{125}{5}{125}{95}

\linethickness{0.1mm}

\put(3,70){$x$} \put(58,10){$g$} \put(55,122){$\psi$}
\put(117,73){$y$}

\end{picture}
\vskip 3mm \centerline{b) }\end{center}
\end{minipage}
\hfill
\begin{minipage}[h]{0.3\linewidth}
\begin{center}
\begin{picture}(130,130)

\put(65,65){\vector(1,0){65,65}} \put(65,65){\vector(-1,0){65,65}}
\put(65,65){\vector(0,1){65,65}} \put(65,65){\vector(0,-1){65,65}}

\Vidr{65}{65}{125}{95} \VidrTo{65}{125}

\Vidr{65}{65}{45}{125} \VidrTo{25}{65} \VidrTo{5}{125}

\Vidr{65}{65}{35}{5} \VidrTo{5}{65}

\Vidr{65}{65}{125}{45} \VidrTo{65}{25} \VidrTo{125}{5}

\linethickness{0.4mm}

\Vidr{5}{65}{65}{65} \Vidr{5}{125}{65}{125} \Vidr{5}{65}{5}{125}
\Vidr{65}{25}{65}{125} \Vidr{25}{65}{25}{25} \Vidr{25}{25}{65}{25}
\Vidr{45}{25}{45}{125}

\linethickness{0.1mm}

\put(-2,70){$x$} \put(55,5){$g$} \put(55,115){$\psi$}
\put(125,73){$y$}

\end{picture}
\vskip 3mm \centerline{c) }\end{center}
\end{minipage}
\hfill \caption{} \label{fig:05}
\end{figure}

The next lemma follows directly from the construction of \SAT s.

\begin{lemma}\label{lema:2.2}
Let $g$ and $\psi$ be piecewise linear maps $[0, 1]\rightarrow [0,
1]$, and let $g$ be unimodal surjective such that $g(0) =g(1)=0$.
Suppose that $x_1,\ldots, x_s$ is the minimal set such that
$$\mathcal{S} = \bigcup\limits_{i=1}^s\SAT(x_i)
$$ contains all the kinks of $g$ and $\psi$.

Then the set of lines $\mathcal{S}$ determines each of two graphs
of $g$ and $\psi$ in its quadrant as follows:

1. Graph starts at the origin;

2. If graph contains a point inside some rectangle, which is
formed by lines of $\mathcal{S}$, then this graph contains a
diagonal of this rectangle.

3. The function maps every its maximum segment $I$ of monotonicity
onto $[0, 1]$.
\end{lemma}

%\begin{proof}
%
%\end{proof}

We will call the {\it determinating lattice} the set of lines of
the minimal set of $\SAT$s, which contains all the kinks of the
maps $g$ and $\psi$. If $\SAT(0)$ contains an extremum points of
one of the graphs of $g$, the we will call it {\it boundary}
$\SAT$.

The goal of the next two lemmas is to provide a method to conclude
the commutativity of $g$ and $\psi$ from their commutativity on
the points of their intersection with their determinating lattice.

\begin{lemma}\label{lema:2.3}
Suppose that $a, b\in [0, 1]$ are such that:

1. $g$ commutes with $\psi$ at $a$ and $b$;

2. $\psi$ is linear on the intervals with ends $\{a; b\}$ and with
ends $\{ g(a); g(b)\}$;

3. $g$ is linear on the intervals with ends $\{a; b\}$ and with
ends $\{ \psi(a); \psi(b)\}$.

\end{lemma}
\begin{proof}
The Lemma follows from the evident geometrical constructions (see
Fig.~\ref{fig:06}).
\begin{figure}[htbp]
%\begin{minipage}[h]{0.9\linewidth}
\begin{center}
\begin{picture}(130,130)

\put(65,65){\vector(1,0){65,65}} \put(65,65){\vector(-1,0){65,65}}
\put(65,65){\vector(0,1){65,65}} \put(65,65){\vector(0,-1){65,65}}

\Vidr{50}{85}{20}{110} \Vidr{80}{85}{110}{110} %
\Vidr{50}{45}{20}{20} \Vidr{80}{45}{110}{20}

\linethickness{0.4mm}

\Vidr{50}{85}{80}{85} \VidrTo{80}{45} \VidrTo{50}{45}
\VidrTo{50}{85}

\Vidr{20}{110}{110}{110} \VidrTo{110}{20} \VidrTo{20}{20}
\VidrTo{20}{110}

\put(12,56){$a$} \put(43,67){$b$}

\put(20,65){\circle*{4}} \put(50,65){\circle*{4}}

%\Vidr{30}{15}{30}{80} \Vidr{95}{15}{95}{110} \Vidr{20}{80}{95}{80}
%\Vidr{20}{80}{20}{35} \Vidr{20}{35}{95}{35}
%\Vidr{95}{110}{10}{110} \Vidr{10}{110}{10}{55}
%\Vidr{10}{55}{95}{55} \Vidr{60}{55}{60}{80} \Vidr{50}{110}{50}{35}
%\Vidr{40}{110}{40}{15} \Vidr{30}{15}{95}{15}

\linethickness{0.1mm}

\put(3,70){$x$} \put(55,5){$g$} \put(55,122){$\psi$}
\put(125,73){$y$}

\end{picture}
%\vskip 3mm
%\centerline{a) }
\end{center}
%\end{minipage}
%\hfill
\caption{proof of Lemma~\ref{lema:2.4}} \label{fig:06}
\end{figure}
\end{proof}

Lemmas~\ref{lema:2.2} and~\ref{lema:2.3} imply the next fact.

\begin{lemma}\label{lema:2.4}
Let $g$ and $\psi$ be piecewise linear maps $[0, 1]\rightarrow [0,
1]$, and let $g$ be unimodal surjective such that $g(0) =g(1)=0$.
Suppose that $x_1,\ldots, x_s$ is the minimal set such that $$
\bigcup\limits_{i=1}^s\SAT(x_i)
$$ contains all the kinks of $g$ and $\psi$. Then $g$ commutes
with $\psi$ if and only if $\SAT(x_i)$ does not contradict to the
commutativity of $g$ and $\psi$ for every $i\in \{1,\ldots, s\}$.
\end{lemma}

The deal of out next reasonings is to reverse our reasonings about
commutating maps. In other words, we consider the problem about
the construction of maps $g$ and $\psi$ by their determinating
lattice.

\begin{remark}\label{rem:2.5}
Suppose that $\psi$ is map, which satisfies the condition of
Lemma~\ref{lema:2.1}. Let $n$ be the number of maximal intervals
of monotonicity of $\psi$. Then

(i) the equation $\psi(x)=0$ has $\left[\frac{n}{2}\right]+1$
solutions;

(ii) the equation $\psi(x)=1$ has $\left[\frac{n+1}{2}\right]$
solutions.
\end{remark}

\begin{lemma}\label{lema:2.6}
Let $\psi$ be a non-trivial commutator of $g$. Denote by $n$ the
number of maximal intervals of monotonicity of $\psi$, and $s$ the
numbers of non-boundary \SAT s in the determinating lattice of $g$
and $\psi$. Then the determinating lattice of $g$ and $\psi$
consists of:

1. $2ns +2n-1$ lines of the form $x=x_0$, where $x_0\in (0, 1)$;

2. $2s+1$ lines of the form $\psi =\psi_0$, where $\psi_0\in (0,
1)$.

3. $ns +n-1$ line of the form $g = g_0$, where $g_0\in (0, 1)$;

4. $s$ lines of the form $y=y_0$, where $y_0\in (0, 1)$;

5. coordinate lines;

6. lines $x =1$, $g =1$, $\psi =1$ and $y=1$.

\end{lemma}

\begin{proof}
Notice that every non-boundary $\SAT$ consists of:

1. $2n$ lines of the form $x=x_0$, where $x_0\in (0, 1)$;

2. $2$ lines of the form $\psi =\psi_0$, where $\psi_0\in (0, 1)$.

3. $n$ line of the form $g = g_0$, where $g_0\in (0, 1)$;

4. $1$ line of the form $y=y_0$, where $y_0\in (0, 1)$.

\noindent If $x$ is such that $g\circ \psi(x) =1$, then $\SAT(x)$
consists of:

1. $n$ lines of the form $x=x_0$, where $x_0\in (0, 1)$;

2. $0$ lines of the form $\psi =\psi_0$, where $\psi_0\in (0, 1)$;

3. $\left[\frac{n+1}{2}\right]$ line of the form $g = g_0$, where
$g_0\in (0, 1)$, by (ii) of Remark~\ref{rem:2.5};

4. $0$ lines of the form $y=y_0$, where $y_0\in (0, 1)$.

\noindent Moreover, $\SAT(0)$ contains

1. $n-1$ lines of the form $x=x_0$, where $x_0\in (0, 1)$;

2. $1$ line of the form $\psi =\psi_0$, where $\psi_0\in (0, 1)$.

3. $1+\left[\frac{n}{2}\right]$ line of the form $g = g_0$, where
$g_0\in (0, 1)$, by (ii) of Remark~\ref{rem:2.5};

4. $0$ lines of the form $y=y_0$, where $y_0\in (0, 1)$.

Thus, the lemma follows. \end{proof}

\begin{notation}\label{not:2.7}
Let $\mathcal{S}$ be a determinating lattice of the commuting maps
$g$ and $\psi$. Lemma~\ref{lema:2.2} implies that $\mathcal{S}$
completely determines both $g$ and $\psi$, moreover this Lemma
gives a constructive algorithm of their construction. Thus,
denote:

(i) $A_1, A_2,\ldots $ the consecutive positive (starting from the
origin) intersections of $\mathcal{S}$ with the graph of $g$ in
the quadrant $\psi\times y$;

(ii) $B_1, B_2,\ldots $ the consecutive positive intersections of
$\mathcal{S}$ with the graph of $\psi$ in the quadrant $x\times
\psi$;

(iii) $C_1, C_2,\ldots $ the consecutive positive intersections of
$\mathcal{S}$ with the graph of $\psi$ in the quadrant $g\times
y$;

(iv) $D_1, D_2,\ldots $ the consecutive positive intersections of
$\mathcal{S}$ with the graph of $g$ in the quadrant $x\times g$.
\end{notation}

For the functions $g$ and $\psi$ of the forms~\eqref{eq:2.2}
and~\eqref{eq:2.3} respectively, the points from
Notation~\ref{not:2.7} are shown at Figure~\ref{fig:07}. The next
fact follows from Lemma~\ref{lema:2.6}.

\begin{figure}
\begin{center}

\begin{picture}(300,300)

\put(150,150){\vector(1,0){150}} \put(150,150){\vector(-1,0){150}}
\put(150,150){\vector(0,1){150}} \put(150,150){\vector(0,-1){150}}

\linethickness{0.4mm}

\Vidr{150}{150}{120}{240}
\put(120,240){\circle*{6}}\VidrTo{90}{270}
\put(90,270){\circle*{6}}\VidrTo{60}{240}
\put(60,240){\circle*{6}}\VidrTo{50}{150}
\put(50,150){\circle*{6}}\VidrTo{40}{240}
\put(40,240){\circle*{6}}\VidrTo{30}{270}

\Vidr{150}{150}{240}{195}
\put(240,195){\circle*{6}}\VidrTo{270}{240}
\put(270,240){\circle*{6}}\Vidr{270}{240}{240}{255}
\put(240,255){\circle*{6}}\VidrTo{150}{270}

\Vidr{150}{150}{240}{120}
\put(240,120){\circle*{6}}\VidrTo{270}{90}
\put(270,90){\circle*{6}}\VidrTo{240}{60}
\put(240,60){\circle*{6}}\VidrTo{150}{50}
\put(150,50){\circle*{6}}\VidrTo{240}{40}
\put(240,40){\circle*{6}}\VidrTo{270}{30}

\Vidr{150}{150}{105}{60} \put(105,60){\circle*{6}}\VidrTo{60}{30}
\put(60,30){\circle*{6}}\Vidr{60}{30}{45}{60}
\put(45,60){\circle*{6}}\VidrTo{30}{150}

\linethickness{0.1mm}

\put(150,270){\line(-1,0){120}} \put(30,270){\line(0,-1){120}}

\put(270,240){\line(-1,0){230}} \put(40,240){\line(0,-1){150}}
\put(40,90){\line(1,0){230}}

\put(240,120){\line(-1,0){205}} \put(35,120){\line(0,1){135}}
\put(35,255){\line(1,0){205}}

\put(240,195){\line(-1,0){195}} \put(45,195){\line(0,-1){135}}
\put(45,60){\line(1,0){195}}

\put(50,150){\line(0,-1){100}} \put(50,50){\line(1,0){100}}

\put(60,240){\line(0,-1){210}} \put(60,30){\line(1,0){210}}

\put(105,255){\line(0,-1){195}}

\put(75,255){\line(0,-1){215}} \put(75,40){\line(1,0){165}}

\put(90,270){\line(0,-1){220}}

\put(120,240){\line(0,-1){150}}

\put(135,195){\line(0,-1){75}}

\put(240,255){\line(0,-1){215}}

\put(270,240){\line(0,-1){210}}

\put(55,195){\line(0,-1){155}} \put(55,40){\line(1,0){185}}

\put(245,190){$A_1$} \put(275,235){$A_2$} \put(245,260){$A_3$}

\put(126,228){$B_2$} \put(88,277){$B_4$} \put(62,228){$B_6$}
\put(55,155){$B_8$}

\put(245,122){$C_1$} \put(275,95){$C_2$} \put(226,65){$C_3$}

\put(111,63){$D_3$} \put(55,16){$D_6$} \put(25,58){$D_9$}

\put(3,160){$x$} \put(137,295){$\psi$} \put(290,160){$y$}
\put(137,10){$g$}
\end{picture}
\end{center}\caption{}
\label{fig:07}
\end{figure}

\begin{remark}\label{rem:2.8}
In terms of notations of Lemma~\ref{lema:2.6}, the points from
Notation~\ref{not:2.7} are numbered:

(i) $_{\llcorner}(0,0) = A_0, A_1,\ldots, A_{s+1} =
_{\llcorner}(g^{-1}(1), 1), \ldots, A_{2s+1}, A_{2s+2} =
_{\llcorner}(1,0)$

(ii) $(0,0)_{\lrcorner} = B_0, B_1,\ldots, B_{2ns+2n-1},
B_{2ns+2n} = (1, \psi(1))_{\lrcorner}$;

(iii) $^{\ulcorner}(0, 0) =C_0, C_1,\ldots, C_{ns+n-1}, C_{ns+n} =
^{\ulcorner}(1, \psi(1))$;

(iv) $(0,0)^{\urcorner} = D_0, D_1,\ldots, D_{2ns+2n-1},
D_{2ns+2n} = (1, 0)^{\urcorner}$.
\end{remark}

Notice, that maps $g$ and $\psi$ of the forms~\eqref{eq:2.2}
and~\eqref{eq:2.3} respectively correspond to $n=3$ and $s=1$ in
notations of Lemma~\ref{lema:2.6}. Thus, the quantities of points
in Remark~\ref{rem:2.8} corresponds to ones from
Figure~\ref{fig:07}.

\begin{notation}\label{not:2.9}
Due to Remark~\ref{rem:2.8} denote $a_i, b_i, c_i, d_i$ the first
coordinates of $A_i$, $B_i$, $C_i$ and $D_i$ for all admissible
indices.
\end{notation}

We will write explicitly the following observation.

\begin{remark}\label{rem:2.10}
Suppose that $\mathcal{S}$ is a determinating lattice of maps $g$
and $\psi$. Let points $A_0, \ldots, D_{ns+2n}$ be from
Remark~\ref{rem:2.8}. Then:

(i) For every $i\in \{1,\ldots, 2s+1\}$ either $g$ is linear at
$A_i$, or there is $j\in \{1,\ldots, 2ns+2n-1\}$ such that $A_i$
and $D_j$ have the same coordinates;

(i) For every $i\in \{1,\ldots, 2ns+2n-1\}$ either $g$ is linear
at $D_i$, or there is $j\in \{1,\ldots, 2s+1\}$ such that $D_i$
and $A_j$ have the same coordinates;

(iii) For every $i\in \{1,\ldots, 2ns+2n-1\}$ either $\psi$ is
linear at $B_i$, or there is $j\in \{1,\ldots, ns+n-1\}$ such that
$B_i$ and $C_j$ have the same coordinates;

(iv) For every $i\in \{1,\ldots, ns+n-1\}$ either $\psi$ is linear
at $C_i$, or there is $j\in \{1,\ldots, 2ns+2n-1\}$ such that
$C_i$ and $B_j$ have the same coordinates.
\end{remark}

Suppose that maps $g$ and $\psi$ have $p$ and $q$ kinks
respectively, where end-points of the interval $[0, 1]$ are
considered as kinks too. Thus, there exist non subsets
$$
\mathcal{P}\subset \{0, \ldots, 2s+2\} \times \{0, \ldots,
2ns+2n\} $$ and $$ \mathcal{Q}\subset \{0, \ldots, 2ns+2n\}\times
\{0, \ldots, ns+n\} $$ of numbers of kinks of $g$ and $\psi$, i.e.
the inclusion $(i, j)\in\mathcal{P}$ means that $A_i$ is the kink
of $g$, and $A_i$ has the same coordinates as $D_j$. Analogously,
$(i,j)\in \mathcal{Q}$ means that $B_i$ is the kink of $\psi$ and
$B_i$ has the same coordinates as $C_j$.

For example, for the maps $g$ and $\psi$ from Fig.~\ref{fig:07} we
have that $$ \mathcal{P} = \{(0,0);\, (1,3);\, (2,6);\, (3,9);\,
(4,12)\}
$$ and $$
\mathcal{Q} = \{(0, 0);\, (4,2);\, (6, 3);\ (8, 4);\, (10, 5);\,
(12, 6)\}.
$$

Thus, we cave naturally come to the classification of the
commutative piecewise linear pairs $g$ and $\psi$, where $g$ is
unimodal surjective such that $g(0) =g(1)=0$.

\begin{definition}\label{def:2.11}
Let $U, V\subset (0, 1)$ be finite sets such that $\frac{\#U
+1}{\#V+1}\in \mathbb{N}$.
 Denote $0< x_1 <\ldots
<x_u <1$ and $0 <y_1 <\ldots <y_v <1$ the ordered elements of $U$
and $V$. Denote $x_0 =y_0 = 0$ and $x_{u+1} =y_{v+1} =1$. Define
$\mu_{UV}:\, [0, 1]\rightarrow [0, 1]$ as follows:

1. $\mu_{UV}:\, \{U\cup \{0; 1\}\}\rightarrow \{V\cup \{0; 1\} \}$

2. $\mu_{UV}(0) =0$;

3. If $\mu_{UV}(x_i) =0$ for $i\leq u$, then $\mu_{UV}(x_{i+1})
=y_1$;

4. If $\mu_{UV}(x_i) =1$ for $i\leq u$, then $\mu_{UV}(x_{i+1})
=y_u$;

5. If $\mu_{UV}(x_i) =y_j\in (0, 1)$ for $i\leq u$, then
$\mu_{UV}(x_{i+1})\in \{y_{j-1}; y_{j+1}\}$.

6. $\mu_{UV}$ is piecewise linear and all its kinks belong to $U$.

7. The image of every maximal interval of monotonicity of
$\mu_{UV}$ is $[0, 1]$.
\end{definition}

Remark that in Definition~\ref{def:2.11} the condition $\frac{\#U
+1}{\#V+1}\in \mathbb{Z}$ is equivalent to $\mu_{UV}(1) \in \{0,
1\}$. Moreover, $\frac{\#U +1}{\#V+1}$ is the number of maximal
intervals of monotonicity of $\mu_{UV}$.

Let natural numbers $n,s$ and sets $X, \Psi, G, Y\subset (0, 1)$
with $\# X
 =2ns+2n-1$, $\#G =ns+n-1$, $\# \Psi = 2s+1$, and $\# Y =s$ be
given. If $\mu_{\Psi Y} =\mu_{XG}$ and $\mu_{X\Psi} = \mu_{GY}$,
then say that set $X, \Psi, G$ and $Y$ define a commutative pair
$g =\mu_{XG}$ and $\psi =\mu_{X\Psi}$. In this case denote by:

$\mathcal{A} =\{A_i,\, i\in \{1,\ldots, 2s+1\}\}$ the consecutive
intersections of the graph of $g$ with lines $x=x_0,\, x_0\in
\Psi$;

$\mathcal{B} = \{B_i,\, i\in \{1,\ldots, 2ns+2n-1\}\}$ the
consecutive intersections of the graph of $\psi$ with lines
$x=x_0,\, x_0\in X$;

$\mathcal{C} = \{C_i,\, i\in \{1,\ldots, ns+n-1\}\}$ the
consecutive intersections of the graph of $\psi$ withe lines $x =
x_0,\, x_0\in G$;

$\mathcal{D} = \{D_i,\, i\in \{1,\ldots, 2ns+2n-1\}\}$ the
consecutive intersections of the graph of $g$ with lines $x=x_0,\,
x_0\in X$.

We will say that sets $X, \Psi, G, Y$ are concordant with pairs
$$\mathcal{P}\subset \{0, \ldots, 2s+2\} \times \{0, \ldots,
2ns+2n\} $$ and $$ \mathcal{Q}\subset \{0, \ldots, 2ns+2n\}\times
\{0, \ldots, ns+n\},
$$ if
they define a commutative pair $(g, \psi)$ and, moreover, the
following holds:

1. If a point $A_i$ is a kink of $\mu_{\Psi Y}$, then there exists
$j$ such that $(i, j)\in\mathcal{P}$;

2. If a point $B_i$ is a kink of $\mu_{X\Psi}$, then there exists
$j$ such that $(i, j)\in\mathcal{Q}$;

3. If a point $C_j$ is a kink of $\mu_{GY}$, then there exists $i$
such that $(i, j)\in\mathcal{Q}$;

4. If a point $D_j$ is a kink of $\mu_{XG}$, then there exists $i$
such that $(i, j)\in\mathcal{P}$.

5. If $(i,j)\in \mathcal{P}$ then $A_i$ and $D_j$ have the same
coordinates.

6. If $(i,j)\in \mathcal{Q}$ then $B_i$ and $C_j$ have the same
coordinates.

\vskip 1cm

The construction of sets $X, \Psi, G, Y$, which are concordant
with a given set of pais $\mathcal{P}$ and $\mathcal{Q}$ is almost
converse to the construction of the sets $\mathcal{P}$ and
$\mathcal{Q}$ by given commutative maps $g$ and $\psi$. The
difference is that in the case of concordant sets we do not demand
that sets $\mathcal{P}$ and $\mathcal{Q}$ contain pairs of kinks.

We will treat the next problem. For $n,s\in \mathbb{N}$ and sets
$\mathcal{P}$ and $\mathcal{Q}$ suppose that there exist $X, \Psi,
G, Y$, which are concordant with $\mathcal{P}$ and $\mathcal{Q}$.
Describe all the functions $g$, which are constructed by $X, \Psi,
G, Y$ in the case. Notice, that sets $\mathcal{P}$ and
$\mathcal{Q}$ can be given both explicitly, or by Figure, similar
to Fig.~\ref{fig:04}. The latter way of representation of the sets
$\mathcal{P}$ and $\mathcal{Q}$ is more clear geometrically. In
this case we will say that $g$ (and $\psi$) is {\it restricted by
the lattice} $(X, \Psi, G, Y)$.

\section{Examples}

\subsection{Simple examples}

\begin{example}\label{ex:3.1}
Describe the maps $g$, which are restricted by the lattice from
Figure~\ref{fig:08}.\end{example}

It is seen from Fig.~\ref{fig:08} that $s=0$, $n=3$, $$
\mathcal{P} = \{(0, 0);\, (1,3);\, (2,6)\}
$$ and $$
\mathcal{Q} =\{(0,0);\, (2,1);\, (4,2);\, (6,3)\}.
$$

\begin{figure}[ht]
\begin{center}
\begin{picture}(400,400)
\put(190,190){\vector(0,1){190}} \put(190,190){\vector(1,0){190}}
\put(190,190){\vector(0,-1){190}}
\put(190,190){\vector(-1,0){190}}

\linethickness{0.3mm} \Vidr{190}{190}{370}{280} \VidrTo{190}{370}

%\Vidr{190}{190}{370}{370}

\Vidr{190}{190}{130}{370} \VidrTo{70}{190} \VidrTo{10}{370}

\Vidr{190}{190}{100}{10} \VidrTo{10}{190}

\Vidr{190}{190}{370}{130} \VidrTo{190}{70} \VidrTo{370}{10}

\linethickness{0.1mm}

\put(370,280){\line(-1,0){330}} \put(40,280){\line(0,-1){150}}
\put(40,130){\line(1,0){330}} \put(160,280){\line(0,-1){150}}
\put(100,280){\line(0,-1){270}} \put(130,370){\line(0,-1){300}}
\put(70,190){\line(0,-1){120}} \put(70,70){\line(1,0){120}}
\end{picture}
\end{center}
\caption{} \label{fig:08}
\end{figure}

Denote by $\sdev_1, \sdev_2, \sdev_3$ the slopes of consequent
pieces of monotonicity of $\psi$, starting from the origin, and
denote by $\gdev_1, \gdev_2$ the consequent slopes of $g$,
starting from the origin. Lemma~\ref{lema:2.3} implies the next
fact

\begin{lemma}\label{lema:3.2}
If maps $g$ and $\psi$ are restricted by the lattice from
Figure~\ref{fig:08}, then $ \sdev_1\cdot \gdev_2 =\gdev_1\cdot
\sdev_2,$ \hskip 5mm $\sdev_2\cdot \gdev_2 = \gdev_1\cdot
\sdev_3$, \hskip 5mm $\sdev_2\cdot \gdev_1 =\gdev_2\cdot \sdev_3,
$ \hskip 5mm $ \sdev_3\cdot \gdev_1 = \gdev_2\cdot \sdev_2, $
\hskip 5mm $ \sdev_3\cdot \gdev_2 = \gdev_2\cdot \sdev_1. $
\end{lemma}

Equalities from Lemma~\ref{lema:3.2} imply
\begin{equation}\label{eq:3.1}
\gdev_2 =-\gdev_1. \end{equation}

Denote $a = g^{-1}(1)$. Since both domain an range of $g$ have the
length $1$, then $$ \left\{\begin{array}{l}1 = \gdev_1 a,\\
1 = -(1-a)\cdot \gdev_2,
\end{array}\right.
$$ whence $$
1 = -\left(1 - \frac{1}{\gdev_1}\right)\cdot \gdev_2,
$$ and~\eqref{eq:3.1} implies $\gdev_1 =2$. The latter
and~\eqref{eq:3.1} means that $g$ is the tent map. Thus, we have
proved the next

\begin{proposition}\label{prop:01}
Let the lattice be as at the Figure~\ref{fig:08}. Then $g$ is the
tent map.
\end{proposition}

We will use notation from Example~\ref{ex:3.1} in our further
computations.

\begin{example}\label{ex:3.3}
Describe the maps $g$, which are restricted by the lattice from
Figure~\ref{fig:09}.
\end{example}

\begin{figure}[ht]
\begin{center}
\begin{picture}(400,400)
\put(190,190){\vector(0,1){190}} \put(190,190){\vector(1,0){190}}
\put(190,190){\vector(0,-1){190}}
\put(190,190){\vector(-1,0){190}}

\linethickness{0.3mm} \Vidr{190}{190}{370}{280} \VidrTo{190}{370}

\Vidr{190}{190}{130}{370} \VidrTo{70}{190} \VidrTo{10}{370}

\Vidr{190}{190}{100}{10} \VidrTo{10}{190}

\Vidr{190}{190}{370}{130} \VidrTo{190}{70} \VidrTo{370}{10}

\linethickness{0.1mm}

\put(370,280){\line(-1,0){330}} \put(40,280){\line(0,-1){150}}
\put(40,130){\line(1,0){330}} \put(160,280){\line(0,-1){150}}
\put(100,280){\line(0,-1){270}} \put(130,370){\line(0,-1){300}}
\put(70,190){\line(0,-1){120}} \put(70,70){\line(1,0){120}}

\put(280,325){\line(-1,0){255}} \put(280,325){\line(0,-1){285}}

\put(55,235){\line(1,0){225}} \put(55,235){\line(0,-1){135}}

\put(55,100){\line(1,0){225}}

\put(25,325){\line(0,-1){165}} \put(25,160){\line(1,0){255}}

\put(85,235){\line(0,-1){195}} \put(85,40){\line(1,0){195}}

\put(115,325){\line(0,-1){285}} \put(145,325){\line(0,-1){225}}
\put(175,235){\line(0,-1){75}}

\put(160,280){\circle*{5}}  \put(280,235){\circle*{5}}
\put(280,160){\circle*{5}}  \put(145,100){\circle*{5}}
\put(130,370){\circle*{5}}  \put(370,280){\circle*{5}}
\put(370,130){\circle*{5}}  \put(100,10){\circle*{5}}
\put(100,280){\circle*{5}}  \put(280,325){\circle*{5}}
\put(280,100){\circle*{5}}  \put(55,100){\circle*{5}}
\put(70,190){\circle*{5}}       \put(190,70){\circle*{5}}
\put(40,280){\circle*{5}}       \put(280,40){\circle*{5}}

\end{picture}
\end{center}
\caption{} \label{fig:09}
\end{figure}

We will use the notations for slopes of $\psi$ and $g$, introduced
in Example~\ref{ex:3.1}. Analogously to it was done in
Example~\ref{ex:3.1}, by Lemma~\ref{lema:2.3} we have

\begin{lemma}\label{lema:3.4}
If maps $g$ and $\psi$ are restricted by the lattice from
Figure~\ref{fig:09}, then $ \sdev_1\cdot \gdev_2 = \gdev_1\cdot
\sdev_2, $ \hskip 5mm $ \sdev_2\cdot \gdev_3 = \gdev_1\cdot
\sdev_3, $ \hskip 5mm $ \sdev_2\cdot \gdev_4 = \gdev_2\cdot
\sdev_4, $ \hskip 5mm $\sdev_3\cdot \gdev_4 = \gdev_2\cdot
\sdev_5, $ \hskip 5mm $ \sdev_3\cdot \gdev_3 = \gdev_2\cdot
\sdev_6, $ \hskip 5mm $ \sdev_4\cdot \gdev_2 = \gdev_3\cdot
\sdev_6, $ \hskip 5mm $ \sdev_4\cdot \gdev_1 = \gdev_3\cdot
\sdev_5, $ \hskip 5mm $\sdev_5\cdot \gdev_1 = \gdev_3\cdot
\sdev_4, $ \hskip 5mm $ \sdev_5\cdot \gdev_2 = \gdev_4\cdot
\sdev_3, $ \hskip 5mm $\sdev_6\cdot \gdev_3 = \gdev_4\cdot \sdev_2
$ and $\sdev_6\cdot \gdev_4 = \gdev_4\cdot \sdev_1.$
\end{lemma}

\begin{lemma}\label{lema:3.5}
Equalities of Lemma~\ref{lema:3.4} imply $\gdev_3 =-\gdev_1$ and
$\gdev_4 = \frac{-(\gdev_1)^2}{\gdev_2}$.
\end{lemma}

\begin{lemma}\label{lema:3.6}
Equalities of Lemma~\ref{lema:3.4} imply $\sdev_2 =
\frac{\sdev_1\cdot \gdev_2}{\gdev_1} $, \hskip 5mm $\sdev_3 =
 \frac{-\sdev_1\cdot \gdev_2}{\gdev_1}
$, \hskip 5mm $ \sdev_4 =
 \frac{-\sdev_1\cdot \gdev_1}{\gdev_2}
$, \hskip 5mm $ \sdev_5 =
 \frac{\sdev_1\cdot \gdev_1}{\gdev_2}
$ and $\sdev_6 = \sdev_1.$
\end{lemma}

\begin{lemma}
Equations of Lemmas~\ref{lema:3.5} and~\ref{lema:3.6} imply the
equations of Lemma~\ref{lema:3.4}
\end{lemma}

\begin{lemma}\label{lema:3.8}
If $g$ be restricted by the lattice at Fig.~\ref{fig:09}. Then
$\gdev_1 = 2$, \hskip 5mm $\gdev_3 =-2$ and $\gdev_4 =
\frac{-4}{\gdev_2}.$
\end{lemma}

Denote by $a$ is the $y$-coordinate of the first kink of $g$.
Since the length of the domain of $g$ equals $1$, then
\begin{equation}\label{eq:3.2}
\frac{a}{\gdev_1} + \frac{1-a}{\gdev_2} -\frac{1-a}{\gdev_3}
-\frac{a}{\gdev_4} =1.
\end{equation}
Using Lemma~\ref{lema:3.5} rewrite~\eqref{eq:3.2} as
$$ \frac{a}{\gdev_1} + \frac{1-a}{\gdev_2}
+\frac{1-a}{\gdev_1} +\frac{a\cdot \gdev_2}{(\gdev_1)^2} =1,
$$ which is the same as
\begin{equation}\label{eq:3.3}
a\cdot \left( \frac{\gdev_2}{(\gdev_1)^2}
-\frac{1}{\gdev_2}\right) = 1 - \frac{1}{\gdev_1}
-\frac{1}{\gdev_2} \end{equation}

Since $B_2$ and $C_1$ are the same points, then $x$-coordinate of
$A_2$ is the same as $y$-coordinate of $A_1$. In other words,
$$
\frac{a}{\gdev_1} + \frac{1-a}{\gdev_2} = a,
$$
which can be rewritten as
\begin{equation}\label{eq:3.4}
a\cdot\left( 1 -\frac{1}{\gdev_1}  +\frac{1}{\gdev_2}\right) =
\frac{1}{\gdev_2}
\end{equation}
Express the value of $a$ from~\eqref{eq:3.3} and~\eqref{eq:3.4},
whence
$$
\frac{\displaystyle{1 - \frac{1}{\gdev_1}
-\frac{1}{\gdev_2}}}{\displaystyle{\frac{\gdev_2}{(\gdev_1)^2}
-\frac{1}{\gdev_2}}}\, =
\frac{\displaystyle{\frac{1}{\gdev_2}}}{\displaystyle{1
-\frac{1}{\gdev_1} +\frac{1}{\gdev_2}}}
$$
This expression can be simplified to
$$
(\gdev_1 -1)^2 =1,
$$ whence $\gdev_1 = 2$. Now the lemma follows from
Lemma~\ref{lema:3.5} and equality~\eqref{eq:3.4}.

\begin{lemma}\label{lema:3.9}
Let $a$ be the $y$-coordinate of the first kink of a piecewise
linear unimodal map $g$, and let $\gdev_1,\ldots, \gdev_4$ be the
slopes of $g$. If equalities of Lemma~\ref{lema:3.8} hold, then
$$ g:\, (0, 0)\rightarrow A_1\left(\frac{a}{2}, a\right) \rightarrow
A_2\left(a, 1\right) \rightarrow A_3\left(\frac{a+1}{2}, a\right)
\rightarrow (1, 0).
$$
\end{lemma}

\begin{proof}
Since $a$ is denoted the $y$-coordinate of the first king of $g$,
we can write the coordinates of $A_1$, $A_2$ and $A_3$ as follows.
The coordinates of $A_1$ are $$A_1\left(\frac{a}{2}, a\right),$$
since $g'(0)=2$. Since $g$ has tangent $\gamma_2$ on the interval
$A_1A_2$ with the length of $y$-projection $1-a$, then the
$x$-projection of $A_1A_2$ is $\frac{1-a}{\gamma_2}$. Thus, $$
A_2\left(\frac{a}{2} +\frac{1-a}{\gamma_2}, 1\right).
$$

Since $\gamma_2 = \frac{2}{a} -2 = \frac{2(1-a)}{a}$, then $A_2$
has coordinates
$$ A_2\left(a, 1\right).
$$

We can compute the $x$-coordinate of $A_3$ as $$a -
\frac{1-a}{\gamma_3} =a +\frac{1-a}{2} =\frac{a+1}{2}.
$$ Since $A_1$ and $A_3$ has the same $y$-coordinate, then $$
A_3\left(\frac{3a-1}{2}, a\right).
$$
\end{proof}

\begin{lemma}\label{lema:3.10}
Let $a$ be the $y$-coordinate of first kink of a piecewise linear
unimodal map $g$, and $\gdev_1,\ldots, \gdev_4$ be the slopes of
$g$. Then $g$ is restricted by the lattice at Figure~\ref{fig:09}
if and only if equalities of Lemma~\ref{lema:3.8} hold.
\end{lemma}

\begin{proof}
Since formulas of Lemma~\ref{lema:3.8} are obtained during its
proof constructively from the description of Figure~\ref{fig:09},
then the unique fact, which is necessary to prove is that pairs of
points, determined by the sets $\mathcal{P}$ and $\mathcal{Q}$,
consist of the pairs points with equal coordinates.

Thus, we have to compute the explicit coordinates of points $B_2,
B_4, B_6, B_8$ and $B_{10}$, and compare them with $C_1, C_2, C_3,
C_4,$ and $C_5$. Also we have to compare the coordinates of points
$A_1, A_2$ and $A_3$ with $D_3, D_6$ and $D_9$ respectively.

$B_4\left(\frac{\frac{a}{\gdev_1} + \frac{1-a}{\gdev_2}}{s_1} +
\frac{1 -\frac{a}{\gdev_1} -\frac{1-a}{\gdev_2}}{s_2}, 1\right)$

$$
C_2\left(\frac{a}{s_1} +\frac{1-a}{s_2}, 1\right).
$$

Since $B_4$ and $C_2$ is the same point on the graph of $\psi$,
then $$ \frac{\frac{a}{\gdev_1} + \frac{1-a}{\gdev_2}}{s_1} +
\frac{1 -\frac{a}{\gdev_1} -\frac{1-a}{\gdev_2}}{s_2} =
\frac{a}{s_1} +\frac{1-a}{s_2}
$$
Using Lemmas~\ref{lema:3.6} and~\ref{lema:3.8} rewrite the latter
equality:
$$ \frac{\frac{\frac{2}{\gdev_2 +2}}{2}
+\frac{1-\frac{2}{\gdev_2 +2}}{\gdev_2}}{s_1} + \frac{1
-\frac{\frac{2}{\gdev_2 +2}}{2} -\frac{1-\frac{2}{\gdev_2
+2}}{\gdev_2}}{\frac{\sdev_1\cdot \gdev_2}{2}} = \frac{2}{\gdev_2
+2}\cdot \frac{1}{ \sdev_1} +\frac{1-\frac{2}{\gdev_2
+2}}{\frac{\sdev_1\cdot \gdev_2}{2}}.
$$
This can be simplified to
$$ \frac{1}{\gdev_2 +2} +\frac{1}{\gdev_2 +2}
+\left(1 -\frac{1}{\gdev_2 +2} -\frac{1}{\gdev_2 +2}\right)\cdot
\frac{2}{\gdev_2}
 = \frac{2}{\gdev_2 +2}
+\left(1-\frac{2}{\gdev_2 +2}\right)\cdot \frac{2}{\gdev_2},
$$ which is the correct equality.

Calculate $$B_7\left(\frac{a}{\gdev_1} +\frac{1-a}{\gdev_2},
\frac{a}{\gdev_1} +\frac{1-a}{\gdev_2}\right)$$ and
$$C_3\left(
\frac{a}{\sdev_1} + \frac{1-a}{\sdev_2} -\frac{1-a}{\sdev_3},
\frac{a}{\gdev_1} +\frac{1-a}{\gdev_2}\right).$$ Since $B_6$ and
$C_3$ are the same point, then $$ \frac{a}{\gdev_1}
+\frac{1-a}{\gdev_2} = \frac{a}{\sdev_1} + \frac{1-a}{\sdev_2}
-\frac{1-a}{\sdev_3}.
$$
Using Lemmas~\ref{lema:3.6} and~\ref{lema:3.8},  obtain
$$ \frac{2}{\gdev_2 +2}\cdot \frac{1}{2} +\frac{1-\frac{2}{\gdev_2
+2}}{\gdev_2} = \frac{2}{\gdev_2 +2}\cdot\frac{1}{\sdev_1} +
\frac{1-\frac{2}{\gdev_2 +2}}{\frac{\sdev_1\cdot \gdev_2}{2}}
-\frac{1-\frac{2}{\gdev_2 +2}}{\frac{-\sdev_1\cdot \gdev_2}{2}}.
$$$$ \frac{1}{\gdev_2 +2}
+\frac{1}{\gdev_2 +2} = \frac{2}{\gdev_2 +2}\cdot\frac{1}{\sdev_1}
+ \frac{2}{\sdev_1\cdot (\gdev_2 +2)} + \frac{2}{\sdev_1\cdot
(\gdev_2 +2)}
$$ $$ 2 = \frac{6}{\sdev_1}
+ \frac{4}{\sdev_1}
$$ $$
\sdev_1 =3.
$$

Now, $\overrightarrow{OB_6} =\overrightarrow{OC_3}$ implies
$\overrightarrow{OB_8} = \overrightarrow{OC_4}.$

Analogously, $\overrightarrow{OB_8} =\overrightarrow{OC_4}$
implies $\overrightarrow{OB_{10}} =\overrightarrow{OC_4}$

We will need for this purpose the slopes $\sdev_1,\ldots, \sdev_6$
of the linear parts of the map $\psi$.

Since the length of the domain of $\psi$ equals to $1$, then
$$
\left(\frac{a}{\gdev_1} +
\frac{1-a}{\gdev_2}\right)\cdot\left(\frac{1}{\sdev_1}
-\frac{1}{\sdev_4} +\frac{1}{\sdev_5}\right)
-\left(\frac{1-a}{\gdev_3} +\frac{a}{\gdev_4}\right)\cdot
\left(\frac{1}{\sdev_2} -\frac{1}{\sdev_3}
+\frac{1}{\sdev_6}\right) =1
$$
Using Lemma~\ref{lema:3.6} and~\ref{lema:3.8}, continue:
$$ \left(\frac{2}{\gdev_2 +2}\cdot
\frac{1}{2} +
\frac{1-\frac{2}{\gdev_2+2}}{\gdev_2}%
\right)\cdot\left(\frac{1}{3}
+\frac{\gdev_2}{6} +%
\frac{\gdev_2}{6}%
\right) -$$$$-\left(\frac{1-\frac{2}{\gdev_2 +2}}{-2}
+\frac{\frac{2}{\gdev_2 +2}\cdot \gdev_2}{-4}\right)\cdot
\left(\frac{2}{3\cdot \gdev_2} -\frac{2}{-3\cdot \gdev_2}
+\frac{1}{3}\right) =
$$$$ =\left(\frac{1}{\gdev_2 +2}+%
\frac{1}{\gdev_2+2}%
\right)\cdot\left(\frac{1}{3}
+\frac{\gdev_2}{3}%
\right) +\left(\frac{\gdev_2}{2\cdot (\gdev_2+2)}
+\frac{\gdev_2}{2\cdot (\gdev_2+2)}\right)\cdot
\left(\frac{4}{3\cdot \gdev_2} +\frac{1}{3}\right) =
$$$$
=\frac{2}{\gdev_2 +2}\cdot \frac{\gdev_2+1}{3}
+\frac{\gdev_2}{\gdev_2+2}\cdot \frac{4 +\gdev_2}{3\cdot \gdev_2}
=\frac{2\cdot (\gdev_2+1) +(4+\gdev_2)}{3\cdot (\gdev_2 +2)}=1
$$

Points $D_3$ and $D_9$ have the same coordinates as $A_1$ and
$A_3$, because the $y$-coordinates of $D_3$ and $D_9$ is the
$x$-coordinate of $C_3$, coordinates of $C_3$ are equal, and
$y$-coordinates of $C_3$ are the same as $y$-coordinates of $A_1$
and $A_3$.

Now the proposition follows from Lemma~\ref{lema:2.4}.
\end{proof}

\begin{example}\label{ex:3.11}
Describe all the maps $g$, which are restricted by the lattice
from Figure~\ref{fig:10}.
\end{example}

\begin{figure}[ht]
\begin{center}
\begin{picture}(400,400)
\put(190,190){\vector(0,1){190}} \put(190,190){\vector(1,0){190}}
\put(190,190){\vector(0,-1){190}}
\put(190,190){\vector(-1,0){190}}

\linethickness{0.3mm} \Vidr{190}{190}{370}{280} \VidrTo{190}{370}

\Vidr{190}{190}{130}{370} \VidrTo{70}{190} \VidrTo{10}{370}

\Vidr{190}{190}{100}{10} \VidrTo{10}{190}

\Vidr{190}{190}{370}{130} \VidrTo{190}{70} \VidrTo{370}{10}

\linethickness{0.1mm}

\put(370,280){\line(-1,0){330}} \put(40,280){\line(0,-1){150}}
\put(40,130){\line(1,0){330}} \put(160,280){\line(0,-1){150}}
\put(100,280){\line(0,-1){270}} \put(130,370){\line(0,-1){300}}
\put(70,190){\line(0,-1){120}} \put(70,70){\line(1,0){120}}

\put(280,325){\line(-1,0){255}} \put(280,325){\line(0,-1){285}}

\put(55,235){\line(1,0){225}} \put(55,235){\line(0,-1){135}}

\put(55,100){\line(1,0){225}}

\put(25,325){\line(0,-1){165}} \put(25,160){\line(1,0){255}}

\put(85,235){\line(0,-1){195}} \put(85,40){\line(1,0){195}}

\put(115,325){\line(0,-1){285}} \put(145,325){\line(0,-1){225}}
\put(175,235){\line(0,-1){75}}

\put(160,280){\circle*{5}}  \put(135,368){$B_4$}
\put(280,235){\circle*{5}} \put(284,228){$A_1$}
\put(280,160){\circle*{5}} \put(284,160){$C_1$}

\put(160,130){\circle*{5}} \put(160,120){$D_2$}
\put(40,130){\circle*{5}}

\put(130,370){\circle*{5}}  \put(370,280){\circle*{5}}
\put(370,130){\circle*{5}}  \put(100,10){\circle*{5}}
\put(100,280){\circle*{5}}  \put(280,325){\circle*{5}}
\put(280,100){\circle*{5}}  %\put(55,100){\circle*{5}}
\put(70,190){\circle*{5}}       \put(190,70){\circle*{5}}
\put(40,280){\circle*{5}} \put(280,40){\circle*{5}}

\end{picture}
\end{center}
\caption{} \label{fig:10}
\end{figure}

\begin{lemma}\label{lema:3.12}
If the lattice is as on the Figure~\ref{fig:10}, then $
\sdev_1\cdot \gdev_2 = \gdev_1\cdot \sdev_2$, \hskip 5mm $
\sdev_2\cdot \gdev_3 = \gdev_2 \cdot \sdev_3$, \hskip 5mm $
\sdev_2\cdot \gdev_4 = \gdev_2\cdot \sdev_4$, \hskip 5mm $
\sdev_3\cdot \gdev_4 = \gdev_2\cdot \sdev_5$, \hskip 5mm $
\sdev_3\cdot \gdev_3 = \gdev_2\cdot \sdev_6$, \hskip 5mm $
\sdev_4\cdot \gdev_2 = \gdev_3\cdot \sdev_6$, \hskip 5mm $
\sdev_4\cdot \gdev_1 = \gdev_3\cdot \sdev_5$, \hskip 5mm $
\sdev_5\cdot \gdev_1 = \gdev_3\cdot \sdev_4$, \hskip 5mm $
\sdev_5\cdot \gdev_2 = \gdev_3 \cdot \sdev_3$, \hskip 5mm $
\sdev_6\cdot \gdev_3 = \gdev_4\cdot \sdev_2$ and $ \sdev_6\cdot
\gdev_4 = \gdev_4\cdot \sdev_1$.
\end{lemma}

Equalities of Lemma~\ref{lema:3.12} imply $ \gdev_1 = \gdev_2$ and
$ \gdev_3 =\gdev_4 = -\gdev_1$, whence $g$ is the tent map. Since
$B_2$ has the same coordinate as $C_1$, then $A_1$ is a pre-image
of the maximum point of $g$, whence $A_1 = (1/4, 1/2)$. Now, $B_2
=(1/6,1/2)$, whence $D_2 =(1/3, 1/4)$. Since $D_2\neq A_1$, then
we have the next proposition.

\begin{proposition}\label{prop:03}
There is no commutative piecewise linear functions, which are
restricted by the lattice from Figure~\ref{fig:10}.
\end{proposition}

\begin{example}\label{ex:3.13}
Describe all the maps $g$, which are restricted by the lattice
from Figure~\ref{fig:11}.
\end{example}

\begin{figure}[ht]
\begin{center}
\begin{picture}(400,400)
\put(190,190){\vector(0,1){190}} \put(190,190){\vector(1,0){190}}
\put(190,190){\vector(0,-1){190}}
\put(190,190){\vector(-1,0){190}}

\linethickness{0.3mm} \Vidr{190}{190}{370}{280} \VidrTo{190}{370}

\Vidr{190}{190}{130}{370} \VidrTo{70}{190} \VidrTo{10}{370}

\Vidr{190}{190}{100}{10} \VidrTo{10}{190}

\Vidr{190}{190}{370}{130} \VidrTo{190}{70} \VidrTo{370}{10}

\Vidr{190}{190}{370}{370} %\Vidr{190}{190}{10}{10}

\linethickness{0.1mm}

\put(334,334){\circle{3}} \put(334,334){\circle{5}}
\put(334,334){\circle{7}} \put(262,262){\circle{3}}
\put(262,262){\circle{5}} \put(262,262){\circle{7}}
%\put(46,46){\circle*{5}}
%\put(118,118){\circle*{5}}

\put(22,334){\circle*{5}}
%\put(22,166){\circle*{5}}
%\put(34,298){\circle*{5}}
%\put(34,142){\circle*{5}}
%\put(40,280){\circle*{5}}
%\put(40,130){\circle*{5}}
\put(46,262){\circle*{5}} \put(46,118){\circle*{5}}
%\put(58,226){\circle*{5}}
%\put(58,94){\circle*{5}}
%\put(70,70){\circle*{5}}
%\put(70,190){\circle*{5}}
%\put(82,226){\circle*{5}}
\put(82,46){\circle*{5}} \put(94,262){\circle*{5}}
%\put(94,22){\circle*{5}}
\put(94,262){\circle*{5}}
%\put(100,280){\circle*{5}}
%\put(106,298){\circle*{5}}
%\put(106,22){\circle*{5}}
\put(118,334){\circle*{5}} \put(118,46){\circle*{5}}
%\put(130,70){\circle*{5}}
\put(142,334){\circle*{5}}
%\put(154,298){\circle*{5}}
\put(154,118){\circle*{5}}
%\put(160,280){\circle*{5}}
%\put(160,130){\circle*{5}}
%\put(166,142){\circle*{5}}
\put(166,262){\circle*{5}}
%\put(178,226){\circle*{5}}
\put(190,70){\circle*{5}} \put(262,226){\circle*{5}}
\put(262,94){\circle*{5}} \put(262,334){\circle*{5}}
\put(262,166){\circle*{5}} \put(262,46){\circle*{5}}
\put(334,298){\circle*{5}} \put(334,142){\circle*{5}}
\put(334,262){\circle*{5}} \put(334,118){\circle*{5}}
\put(334,142){\circle*{5}} \put(334,22){\circle*{5}}
%\put(310,190){\circle*{5}}

\put(22,334){\line(1,0){240}} \put(22,334){\line(0,-1){168}}

\put(34,298){\line(1,0){300}} \put(34,298){\line(0,-1){156}}

\put(40,280){\line(1,0){330}} \put(40,280){\line(0,-1){150}}

\put(46,262){\line(1,0){288}} \put(46,262){\line(0,-1){144}}

\put(70,190){\line(0,-1){120}}

\put(58,226){\line(1,0){204}} \put(58,226){\line(0,-1){132}}

\put(22,166){\line(1,0){240}}

\put(34,142){\line(1,0){300}}

\put(40,130){\line(1,0){330}}

\put(46,118){\line(1,0){288}}

\put(58,94){\line(1,0){204}}

\put(70,70){\line(1,0){120}}

\put(82,46){\line(1,0){180}} \put(82,46){\line(0,1){180}}

\put(94,22){\line(1,0){240}} \put(94,22){\line(0,1){240}}

\put(100,10){\line(0,1){270}}

\put(106,22){\line(0,1){276}}

\put(118,46){\line(0,1){288}}

\put(130,70){\line(0,1){300}}

\put(142,94){\line(0,1){240}}

\put(154,118){\line(0,1){180}}

\put(160,130){\line(0,1){150}} \put(166,142){\line(0,1){120}}

\put(178,166){\line(0,1){60}} \put(262,46){\line(0,1){288}}
\put(334,22){\line(0,1){276}}

\end{picture}
\end{center}\caption{} \label{fig:11}
\end{figure}

\begin{lemma}\label{lema:3.14}
If maps $g$ and $\psi$ are restricted by the lattice from
Figure~\ref{fig:11}, then $ \sdev_1\cdot \gdev_2 = \gdev_1\cdot
\sdev_2$, \hskip 5mm $\sdev_2\cdot \gdev_3 = \gdev_1\cdot
\sdev_3$, \hskip 5mm $ \sdev_2\cdot \gdev_4 = \gdev_1\cdot
\sdev_4$, \hskip 5mm $ \sdev_2\cdot \gdev_5 = \gdev_2\cdot
\sdev_5$, \hskip 5mm $ \sdev_3\cdot \gdev_6 = \gdev_2\cdot
\sdev_6$, \hskip 5mm $ \sdev_4\cdot \gdev_6 = \gdev_2\cdot
\sdev_7$, \hskip 5mm $ \sdev_5\cdot \gdev_5 = \gdev_3\cdot
\sdev_8$, \hskip 5mm $ \sdev_5\cdot \gdev_4 = \gdev_3\cdot
\sdev_9$, \hskip 5mm $ \sdev_5\cdot \gdev_3 = \gdev_4\cdot
\sdev_9$, \hskip 5mm $ \sdev_6\cdot \gdev_2 = \gdev_4\cdot
\sdev_8$, \hskip 5mm $ \sdev_6\cdot \gdev_1 = \gdev_5\cdot
\sdev_7$, \hskip 5mm $ \sdev_7\cdot \gdev_1 = \gdev_5\cdot
\sdev_6$, \hskip 5mm $ \sdev_7\cdot \gdev_2 = \gdev_5\cdot
\sdev_5$, \hskip 5mm $ \sdev_8\cdot \gdev_3 = \gdev_6\cdot
\sdev_4$, \hskip 5mm $ \sdev_8\cdot \gdev_4 = \gdev_6\cdot
\sdev_3$, \hskip 5mm $ \sdev_8\cdot \gdev_5 = \gdev_6\cdot
\sdev_2$, \hskip 5mm $ \sdev_9\cdot \gdev_6 = \gdev_6\cdot
\sdev_1$
\end{lemma}

\begin{lemma}\label{lema:3.15}
Equalities from Lemma~\ref{lema:3.14} imply $ \gdev_4 = -\gdev_3$,
\hskip 5mm $\gdev_5 =-\gdev_1$ and $\gdev_6 =
\frac{-(\gdev_1)^3}{\gdev_2\cdot \gdev_3}$.
\end{lemma}

\begin{lemma}\label{lema:3.16}
Equalities from Lemma~\ref{lema:3.14} imply $\sdev_2 =
\frac{\sdev_1\cdot \gdev_2}{\gdev_1} $, \hskip 5mm $ \sdev_3 =
\frac{\sdev_1\cdot \gdev_2\cdot \gdev_3}{(\gdev_1)^2} $, \hskip
5mm $\sdev_4 = \frac{\sdev_1\cdot \gdev_2\cdot
\gdev_4}{(\gdev_1)^2}$, \hskip 5mm $ \sdev_5 = \frac{\sdev_1\cdot
\gdev_5}{\gdev_1} $, \hskip 5mm $\sdev_6 = \frac{\sdev_1\cdot
\gdev_3\cdot \gdev_6}{(\gdev_1)^2}$, \hskip 5mm $ \sdev_7 =
\frac{\sdev_1\cdot \gdev_4\cdot \gdev_6}{(\gdev_1)^2} $, \hskip
5mm $\sdev_8 = \frac{\sdev_1\cdot (\gdev_5)^2}{\gdev_1\cdot
\gdev_3}$ and $ \sdev_9 = \sdev_1$.
\end{lemma}

\begin{lemma}\label{lema:3.17}
Suppose that $g$ is restricted by the lattice be as at the
Figure~\ref{fig:11}. Denote by $a$ and $b$ the $y$-lengthes of the
first two parts of linearity of $g$, and $\gdev_1,\ldots, \gdev_6$
the slopes of~$g$. Then $a = \frac{8}{ 4\cdot \gdev_2 +
\gdev_2\cdot \gdev_3 +8
 }$, \hskip 5mm $
b =\frac{4\cdot \gdev_2}{ 4\cdot \gdev_2 + \gdev_2\cdot \gdev_3 +8
}$, \hskip 5mm $\gdev_1 =2$, \hskip 5mm $\gdev_4 =-\gdev_3$,
\hskip 5mm $\gdev_5 =-2$, \hskip 5mm $\gdev_6 =
\frac{-8}{\gdev_2\cdot \gdev_3}$.
\end{lemma}

\begin{proof}
Clearly, $_{\llcorner}\overrightarrow{OP_1}
=\left(\frac{a}{\gdev_1}, a\right)$ and
$_{\llcorner}\overrightarrow{OP_2} =\left(\frac{a}{\gdev_1} +
\frac{b}{\gdev_2}, a+b\right)$. Since $(Q_1)_\lrcorner{}
_{\llcorner}P_2$ is a horizontal line, then
$(\overrightarrow{OQ_1})_{\lrcorner}\cdot (0,1) =
\frac{a}{\gdev_1} + \frac{b}{\gdev_2}$. From another hand
$_{\llcorner}P_1(^{\ulcorner}Q_1)$ is vertical line, whence
$^{\ulcorner}(\overrightarrow{OQ_1})\cdot (0, 1) =a.$ Comparing
$^{\ulcorner}(\overrightarrow{OQ_1})\cdot (0, 1)$ with
$(\overrightarrow{OQ_1})_{\lrcorner}\cdot (0,1)$ obtain
\begin{equation}\label{eq:3.5}
\frac{a}{\gdev_1} + \frac{b}{\gdev_2} = a.
\end{equation}
Since $(Q_2)_{\lrcorner}{}(_\llcorner P_5)$ is a horizontal line
and $_{\llcorner}P_2(^{\ulcorner}Q_2)$ is a vertical line, then
$\overrightarrow{OP_5}\cdot (1,0) = \overrightarrow{OP_2}\cdot
(0,1)$. Thus, $$ a+b = \frac{a}{\gdev_1} +\frac{b}{\gdev_2}
+\frac{1-a-b}{\gdev_3} -\frac{1-a-b}{\gdev_4} -\frac{b}{\gdev_5}.
$$ By~\eqref{eq:3.5} and Lemma~\ref{lema:3.15}, rewrite the latter
equality as $$ \frac{2(1-a-b)}{\gdev_3} -\frac{b}{\gdev_5} = b,
$$
which, by Lemma~\ref{lema:3.15}, can be simplified as
\begin{equation}\label{eq:3.6}
\frac{2(1-a-b)}{\gdev_3} +\frac{b}{\gdev_1} = b.\end{equation}

Since the length of the domain of $g$, regarded in the quadrant
$\psi\times y$ equals $1$, then
\begin{equation}\label{eq:3.7}
a + b- \frac{a}{\gdev_6} =1.
\end{equation}

Lemma~\ref{lema:3.15} and~\eqref{eq:3.7} imply
\begin{equation}\label{eq:3.8}
b =1 -\frac{a\cdot \gdev_2\cdot \gdev_3}{(\gdev_1)^3} -a,
\end{equation}
By~\eqref{eq:3.8}, rewrite~\eqref{eq:3.5} as \[ \frac{a}{\gdev_1}
+ \frac{1}{\gdev_2} -\frac{a\cdot \gdev_3}{(\gdev_1)^3}
-\frac{a}{\gdev_2},
 = a\]
whence express $a$ as \begin{equation}\label{eq:3.9} a =
\frac{(\gdev_1)^3}{ (\gdev_1)^3\cdot \gdev_2 -(\gdev_1)^2\cdot
\gdev_2 + \gdev_3\cdot \gdev_2 +(\gdev_1)^3 }
\end{equation}

Plug~\eqref{eq:3.9} into~\eqref{eq:3.8} and, after simplification,
get
\begin{equation}\label{eq:3.10}
b =\frac{(\gdev_1)^3\cdot \gdev_2 -(\gdev_1)^2\cdot \gdev_2}{
(\gdev_1)^3\cdot \gdev_2 -(\gdev_1)^2\cdot \gdev_2 + \gdev_3\cdot
\gdev_2 +(\gdev_1)^3 }
\end{equation}

Substitute~\eqref{eq:3.9} and~\eqref{eq:3.10} into~\eqref{eq:3.6}:
\[\frac{2\cdot \gdev_2\cdot \gdev_3}{\gdev_3}
+((\gdev_1)^2\cdot \gdev_2 -\gdev_1\cdot \gdev_2) =
(\gdev_1)^3\cdot \gdev_2 -(\gdev_1)^2\cdot \gdev_2.\] This
expression can be simplified to
\[(\gdev_1 -2)((\gdev_1)^2 +1) =0,
\] whence $\gdev_1 = 2$
and the proposition follows from Lemma~\ref{lema:3.15} and
equalities~\eqref{eq:3.9} and~\eqref{eq:3.10}.
\end{proof}

\begin{lemma}
Suppose that $g$ is restricted by the lattice be as at the
Figure~\ref{fig:11}. Denote by $a$ and $b$ the $y$-lengthes of the
first two parts of linearity of $g$, and $\gdev_1,\ldots, \gdev_6$
the slopes of~$g$. Then $ \sdev_2 = \frac{3\cdot \gdev_2}{2}$,
\hskip 5mm $\sdev_3 = \frac{3\cdot \gdev_2\cdot \gdev_3}{4}$,
\hskip 5mm $\sdev_4 = \frac{-3\cdot \gdev_2\cdot \gdev_3}{4} $,
\hskip 5mm $\sdev_5 = -3$, \hskip 5mm $\sdev_6 =
\frac{-6}{\gdev_2} $, \hskip 5mm $\sdev_7 = \frac{6}{\gdev_2} $,
\hskip 5mm $\sdev_8 = \frac{6}{\gdev_3} $, \hskip 5mm $\sdev_9 =
3$.
\end{lemma}

\begin{proof}
Since the length of the domain of $\psi$ of the quadrant $x\times
\psi$ equals to $1$, then
$$
a\left( \frac{1}{\sdev_1} -\frac{1}{\sdev_6}
+\frac{1}{\sdev_7}\right) +b\left(\frac{1}{\sdev_2}
-\frac{1}{\sdev_5} +\frac{1}{\sdev_8}\right)
+(1-a-b)\left(\frac{1}{\sdev_3} -\frac{1}{\sdev_4}
+\frac{1}{\sdev_9}\right) =1
$$
By Lemmas~\ref{lema:3.15} and~\ref{lema:3.16} this expression can
be simplified to
$$
a\cdot \frac{2 \cdot \gdev_2}{\gdev_1} +b\cdot
\left(\frac{\gdev_1}{\gdev_2}
 +\frac{\gdev_3}{\gdev_1}
 \right)%+$$ $$
+(1-a-b)\cdot \frac{2(\gdev_1)^2}{\gdev_2\cdot \gdev_3} + 1
 =\sdev_1.
$$ By Lemma~\ref{lema:3.17} rewrite
$$
\frac{8}{ 4\cdot \gdev_2 + \gdev_2\cdot \gdev_3 +8}\cdot \gdev_2
+\frac{4\cdot \gdev_2}{ 4\cdot \gdev_2 + \gdev_2\cdot \gdev_3 +8
}\cdot \left(\frac{2}{\gdev_2}
 +\frac{\gdev_3}{2}
 \right)+$$ $$
+\left(1-\frac{8}{ 4\cdot \gdev_2 + \gdev_2\cdot \gdev_3
+8}-\frac{4\cdot \gdev_2}{ 4\cdot \gdev_2 + \gdev_2\cdot \gdev_3
+8 }\right)\cdot \frac{8}{\gdev_2\cdot \gdev_3} + 1
 =\sdev_1,
$$ which can be simplified to
$\sdev_1 = 3.$ Now the lemma follows from Lemma~\ref{lema:3.16}.
\end{proof}

\begin{lemma}\label{lema:3.19}
Formulas from Lemma~\ref{lema:3.17} imply $\gdev_2 = \frac{2\cdot
b}{a},$ and $\gdev_3 = \frac{4\cdot (1-a-b)}{b}.$
\end{lemma}

\begin{proof}
By the hypothesis, express $\gdev_2$ as
\begin{equation}\label{eq:3.11}
\gdev_2 = \frac{8\cdot (1-a)}{a\cdot (4+ \gdev_3)}
\end{equation}

Plug~\eqref{eq:3.11} into the expression for $b$ from
Lemma~\ref{lema:3.17}, whence express $\gdev_3$ as
$$
\gdev_3 = \frac{4\cdot (1-a)}{b} -4 =\frac{4\cdot (1-a-b)}{b}
$$

Now, by~\eqref{eq:3.11},
$$\gdev_2 =\frac{2\cdot b}{a}.$$
\end{proof}

\begin{lemma}\label{lema:3.20}
The map $g$, described in Lemma~\ref{lema:3.17}, can be
represented as
$$ g:\, (0, 0)\rightarrow A_1\left(\frac{a}{2}, a\right)
\rightarrow A_2(a, a+b) \rightarrow A_3\left(\frac{4a +b}{4},
1\right) \rightarrow$$ $$ \rightarrow A_4\left(\frac{2a +b}{2},
a+b\right) \rightarrow A_5\left(a+b, a\right) \rightarrow (1,0),
$$ where $a, b,\, 0<a<b<1$ are arbitrary parameters.
\end{lemma}

\begin{proof}
Using $\psi\times y$-part of Figure~\ref{fig:11}, write the
coordinates of $A_i$th as follows: $A_1\left(\frac{a}{2},
a\right)$, $A_2\left(\frac{a}{2} +\frac{b}{\gdev_2}, a+b\right)$,
$A_3\left(\frac{a}{2} +\frac{b}{\gdev_2} +\frac{1-(a+b)}{\gdev_3},
1\right)$, $A_4\left(\frac{a}{2} +\frac{b}{\gdev_2}
+\frac{1-(a+b)}{\gdev_3} -\frac{1-(a+b)}{\gdev_4}, a+b\right)$ and
$$A_5\left(\frac{a}{2} +\frac{b}{\gdev_2} +\frac{1-(a+b)}{\gdev_3}
-\frac{1-(a+b)}{\gdev_4} -\frac{b}{\gdev_5}, a\right).$$

We can rewrite these formulas, using Lemmas~\ref{lema:3.17}
and~\ref{lema:3.19}.

The $x$-coordinate of $A_2$ is \begin{equation}\label{eq:3.12}
\frac{a}{2} +\frac{b}{\gdev_2} = \frac{a}{2} +\frac{1\cdot
b}{2\cdot b} =a,
\end{equation} whence $$
A_2(a, a+b).
$$
By Lemma~\ref{lema:3.19} and using~\eqref{eq:3.12} rewrite the
$x$-coordinate of $A_3$ as $$ \frac{a}{2} +\frac{b}{\gdev_2}
+\frac{1-(a+b)}{\gdev_3} \overset{\text{by~\eqref{eq:3.12}}}{=} a
+\frac{1-(a+b)}{\gdev_3} \overset{\text{by
Lem.~\ref{lema:3.19}}}{=}$$
$$=
a +\frac{b\cdot (1-(a+b))}{4\cdot (1-a-b)} = \frac{4a +b}{4}.
$$ Thus,
\begin{equation}\label{eq:3.13}
\frac{a}{2} +\frac{b}{\gdev_2} +\frac{1-(a+b)}{\gdev_3}=\frac{4a
+b}{4}
\end{equation}
and $$ A_3\left(\frac{4a +b}{4}, 1\right).
$$
By Lemma~\ref{lema:3.19} and using~\eqref{eq:3.13} rewrite the
$x$-coordinate of $A_4$ as $$ \frac{a}{2} +\frac{b}{\gdev_2}
+\frac{1-(a+b)}{\gdev_3} -\frac{1-(a+b)}{\gdev_4}
\overset{\text{by~\eqref{eq:3.13}}}{=} \frac{4a +b}{4}
-\frac{1-(a+b)}{\gdev_4} \overset{\text{by Lem.~\ref{lema:3.17},
\ref{lema:3.19}}}{=}
$$$$
=\frac{4a +b}{4} +\frac{b\cdot (1-(a+b))}{4\cdot (1-a-b)}
=\frac{2a +b}{2},
$$ whence

\begin{equation}\label{eq:3.14}
\frac{a}{2} +\frac{b}{\gdev_2} +\frac{1-(a+b)}{\gdev_3}
-\frac{1-(a+b)}{\gdev_4} = \frac{2a +b}{2}
\end{equation} and $$
A_4\left(\frac{2a +b}{2}, a+b\right).
$$
By Lemmas~\ref{lema:3.17} and~\ref{lema:3.19} and
using~\eqref{eq:3.14} write the $x$-coordinate of $A_5$ as $$
\frac{a}{2} +\frac{b-a}{\gdev_2} +\frac{1-(a+b)}{\gdev_3}
-\frac{1-(a+b)}{\gdev_4} -\frac{b}{\gdev_5}
\overset{\text{by~\eqref{eq:3.14}}}{=}
$$$$
=\frac{2a +b}{2} -\frac{b}{\gdev_5} \overset{\text{by
Lem.~\ref{lema:3.19}}}{=} \frac{2a +b}{2} +\frac{b}{2} =a+b,
$$ whence $$
A_5\left(a+b, a\right).
$$
\end{proof}

\subsection{Whether the maps $\psi$ from the examples are
topologically conjugated to the tent map?}\label{sec:3.2}

We will determine in this section, wether the maps $g$ from
examples Examples~\ref{ex:3.3} and \ref{ex:3.13} are topologically
conjugated to the tent map by a piecewise linear map. This
equation is motivated by the following facts.

\begin{lemma}\label{lema:3.21}
If a piecewise linear map $g$ is topologically conjugated to the
tent map via piecewise linear conjugacy, then $g'(0)=2$.
\end{lemma}

\begin{theorem}\cite[Theorem~2]{UMZh-2016}\label{th:1}
For every $v \in (0,\, 1)$ and arbitrary increasing piecewise
linear $g_l:\, [0,\, v]\to [0,\, 1]$, which does not have positive
fixed points, such that $g_l(v)=1$, $g_l'(0)=2$ there exists a
piecewise linear unimodal map $g$, which is conjugated with the
tent map via piecewise linear homeomorphism, and coincides with
$g_l$ on $[0, v]$. Moreover, such $g$ is uniquely defined by
$g_l$.
\end{theorem}

\begin{theorem}\cite[Theorem~3]{UMZh-2016}\label{th:2}
Let $v \in (0,\, 1)$ and $g_r:\, [v,\, 1]\to [0,\, 1]$ be a
decreasing piecewise linear surjective map such that
$(g_r^2)'(x_0) = 4$, where $x_0$ is a fixed point of $g_r$. Then
there exists a piecewise linear unimodal map $g$, which is
conjugated with the tent map via piecewise linear homeomorphism,
and coincides with $g_r$ on $[v, 1]$. Moreover, such $g$ is
uniquely defined by $g_r$.
\end{theorem}

In fact, Theorems~\ref{th:1} and~\ref{th:2} are contain the
necessary conditions for a map to be topologically conjugated with
the tent map vis piecewise linear conjugacy. Notice that these
conditions are satisfied from the maps $g$ from
examples~\ref{ex:3.3} and~\ref{ex:3.13}. Remind that these maps
are described in Lemmas~\ref{lema:3.10} and~\ref{lema:3.17}
respectively.

\begin{lemma}
Each map $g$ from Examples~\ref{ex:3.3} and~\ref{ex:3.13} satisfy
the condition of Theorems~\ref{th:1} and~\ref{th:2}.
\end{lemma}

\begin{proof}
Suppose that a map $g$ is restricted by the lattice from
Example~\ref{ex:3.3}. Then, $g$ is described in
Lemma~\ref{lema:3.9}, whence the derivative of $g$ at origin is
$2$. Since $a<1$, then $\frac{a+1}{2}>a$, whence the positive
fixed point of $g$ belongs to $A_2A_3$. This proves that the
derivative of $g$ at its positive fixed point equals $\gdev_3 =
-2$.

The case of Example~\ref{ex:3.13} follows analogously from
Lemma~\ref{lema:3.20}.
\end{proof}

%In order to prove that maps $g$ from Examples~\ref{ex:3.3}
%and~\ref{ex:3.13} are topologically conjugated with the tent map by
%a piecewise linear conjugacy, we will consider a geometrical
%interpretation of the topological conjugation.

Lemmas~\ref{lema:2.2}, \ref{lema:2.3} and~\ref{lema:2.4} provide a
techniques
 to prove that maps $g$ and $\psi$ commute. Indeed,
these lemmas provide a technique to verify the equality $$
g_1\circ \psi_1 = \psi_2\circ g_2
$$ for arbitrary maps $g_1, g_2, \psi_1, \psi_2$, not necessarily
with the conditions $g_1 =g_2$ and $\psi_1 = \psi_2$. In
particular, for given map $g$ and a conjugacy $h$ we can verify
the conjugation of the tent map $f$ and the map $g$ by the
conjugacy $h$. On another hand, there is no given map $h$ to
verify the conjugacy from $f$ to $g$ in Examples~\ref{ex:3.3}
and~\ref{ex:3.13}.

We will introduced some methods to find the conjugacy from the
tent map to a given unimodal map $g$.
 For any continuous unimodal map $g$, which
is topologically conjugated to the tent map, we have defined
in~\cite{1810.05953} a $g$-decomposition $(\beta_i)_{i\geq 0}$ of
any point $x\in [0, 1]$. This $g$-decomposition coincides with
$v^{-1}$-coordinates expansion of a point for the skew tent map
$f_v$ in our~\cite{Chaos}. The next fact follows
from~\cite[Th.~2]{Yong-Guo-Wang}.

\begin{lemma}\label{lema:3.23}
Let $g_1$ be a piecewise linear unimodal map, which is
topologically conjugated to the tent map. Let $h$ be an increasing
conjugacy from $g_1$ to a map $g_2$. Then for every point $x\in
[0, 1]$ the $g_1$-coordinates of $x$ coincide with
$g_2$-coordinates of $h(x)$.
\end{lemma}

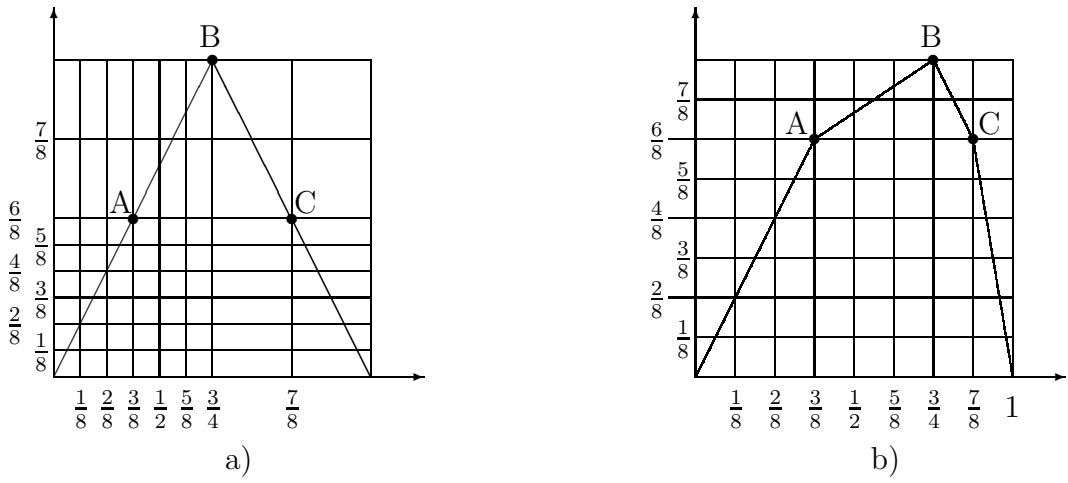
\begin{figure}[htbp]
\begin{minipage}[h]{0.49\linewidth}
\begin{center}
\begin{picture}(140,140)
\put(120,0){\line(0,1){120}} \put(0,120){\line(1,0){120}}
\put(0,0){\vector(1,0){140}} \put(0,0){\vector(0,1){140}}
\put(0,0){\line(1,2){60}} \put(60,120){\line(1,-2){60}}
\put(60,120){\line(1,-2){60}}

\put(10,0){\line(0,1){120}} \put(20,0){\line(0,1){120}}
\put(30,0){\line(0,1){120}} \put(40,0){\line(0,1){120}}
\put(50,0){\line(0,1){120}} \put(60,0){\line(0,1){120}}
\put(90,0){\line(0,1){120}}

\put(7,-15){$\frac{1}{8}$} \put(17,-15){$\frac{2}{8}$}
\put(27,-15){$\frac{3}{8}$} \put(37,-15){$\frac{1}{2}$}
\put(47,-15){$\frac{5}{8}$} \put(57,-15){$\frac{3}{4}$}
\put(87,-15){$\frac{7}{8}$}

\put(-8,6.5){$\frac{1}{8}$} \put(-18,16.5){$\frac{2}{8}$}
\put(-8,26.5){$\frac{3}{8}$} \put(-18,36.5){$\frac{4}{8}$}
\put(-8,46.5){$\frac{5}{8}$} \put(-18,56.5){$\frac{6}{8}$}
\put(-8,86.5){$\frac{7}{8}$}

\put(0,10){\line(1,0){120}} \put(0,20){\line(1,0){120}}
\put(0,30){\line(1,0){120}} \put(0,40){\line(1,0){120}}
\put(0,50){\line(1,0){120}} \put(0,60){\line(1,0){120}}
\put(0,90){\line(1,0){120}}

\put(30,60){\circle*{4}} \put(60,120){\circle*{4}} \put(21,62){A}
\put(55,125){B}

\put(90,60){\circle*{4}} \put(91,62){C}
\end{picture}
\end{center}
\centerline{a)}
\end{minipage}
\hfill
\begin{minipage}[h]{0.49\linewidth}
\begin{center}
\begin{picture}(140,140)
\put(120,0){\line(0,1){120}} \put(0,120){\line(1,0){120}}
\put(0,0){\vector(1,0){140}} \put(0,0){\vector(0,1){140}}

\put(15,0){\line(0,1){120}} \put(30,0){\line(0,1){120}}
\put(45,0){\line(0,1){120}} \put(60,0){\line(0,1){120}}
\put(75,0){\line(0,1){120}} \put(90,0){\line(0,1){120}}
\put(105,0){\line(0,1){120}}

\put(12,-15){$\frac{1}{8}$} \put(27,-15){$\frac{2}{8}$}
\put(42,-15){$\frac{3}{8}$} \put(57,-15){$\frac{1}{2}$}
\put(72,-15){$\frac{5}{8}$} \put(87,-15){$\frac{3}{4}$}
\put(102,-15){$\frac{7}{8}$} \put(117,-15){$1$}

\put(-8,11.5){$\frac{1}{8}$} \put(-18,26.5){$\frac{2}{8}$}
\put(-8,41.5){$\frac{3}{8}$} \put(-18,56.5){$\frac{4}{8}$}
\put(-8,71.5){$\frac{5}{8}$} \put(-18,86.5){$\frac{6}{8}$}
\put(-8,101.5){$\frac{7}{8}$}

\put(0,15){\line(1,0){120}} \put(-10,30){\line(1,0){130}}
\put(0,45){\line(1,0){120}} \put(-10,60){\line(1,0){130}}
\put(0,75){\line(1,0){120}} \put(-10,90){\line(1,0){130}}
\put(0,105){\line(1,0){120}}

\put(45,90){\circle*{4}} \put(34,92){A} %
\put(90,120){\circle*{4}} \put(85,125){B} %
\put(105,90){\circle*{4}} \put(107,92){C}

\qbezier(0,0)(22.5,45)(45,90) \qbezier(45,90)(67.5,105)(90,120)
\qbezier(90,120)(97.5,105)(105,90)
\qbezier(105,90)(112.5,45)(120,0)

\end{picture}
\end{center}
\centerline{ b)}
\end{minipage}
\hfill \caption{Construction of the conjugated map}\label{fig:12}
\end{figure}

Thus, Lemma~\ref{lema:3.23} lets to define a correspondence
between points of the graph of $g_1$ and the graph of $g_2$ as $$
\Gamma_{g_1}:\, (x, g_1(x))\mapsto (h(x), (h\circ g_1)(x))\in
\Gamma_{g_2}.
$$
Indeed, since $g_1 = h^{-1}\circ g_2\circ h$, then $h\circ g_1 =
g_2\circ h$, whence $(h(x), (h\circ g_1)(x))$ belongs to
$\Gamma_{g_2}$ for every $x\in [0, 1]$. We can see these
reasonings geometrically. Let, as above, $f$ be the tent map and
$h$, $$ h:\, (0, 0)\rightarrow \left(\frac{1}{2},
\frac{3}{4}\right) \rightarrow (1,1)$$ be the conjugacy.

The conjugation by $h$ can be understood as the change of scale
$x\mapsto h(x)$. Thus, we can look at the Fig.~\ref{fig:12}a as
the graph of the map $g = h\circ f\circ h^{-1}$ in non-uniform
scale. We can first to plot the graph of $f$ on the
Fig.~\ref{fig:12}a, and then the add the non-uniform scale for the
interval $[0, 1]$. More precisely, in this case we denote each
point $(x,y)$ on the plot of $f$ by $(h(x), h(y))$, obtaining the
graph of new function. If one comes back to the ``common'' uniform
scale, Figure~\ref{fig:12}a transforms to Fig.~\ref{fig:12}b,
where $g$ is given in a common way.

Our deal is to see the tent map $x\mapsto 1 -|1-2x|$ on
Fig.~\ref{fig:12}b. Suppose that the graph $g$ of the
form~\eqref{eq:2.2} is given at Figure~\ref{fig:13}a, where it is
seen that $g(A_1) =g(A_3) =A_2$. Thus, $g$-coordinates of $A_1$,
$A_2$ and $A_3$ are $\left(\frac{1}{4}, \frac{1}{2}\right)$,
$\left(\frac{1}{2}, 1\right)$ and $\left(\frac{3}{4},
\frac{1}{2}\right)$ respectively.

Since $f$-coordinate of $x$ is $x$ for any $x\in [0, 1]$ then it
follows from Lemma~\ref{lema:3.23} that $h\left(\frac{1}{4}\right)
=\frac{3}{8}$, $h\left(\frac{1}{2}\right) =\frac{3}{4}$, and
$h\left(\frac{3}{4}\right) =\frac{7}{8}$.

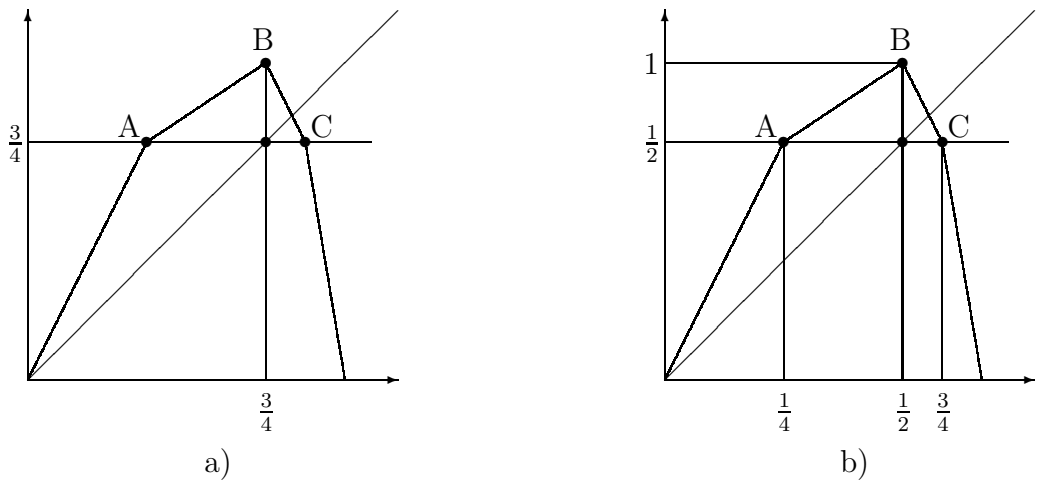
\begin{figure}[htbp]
\begin{minipage}[h]{0.49\linewidth}
\begin{center}
\begin{picture}(140,140)
%\put(120,0){\line(0,1){120}} \put(0,120){\line(1,0){120}}
\put(0,0){\vector(1,0){140}} \put(0,0){\vector(0,1){140}}

\put(0,0){\line(1,1){140}}

%\put(15,0){\line(0,1){120}} \put(30,0){\line(0,1){120}}
%\put(45,0){\line(0,1){120}} \put(60,0){\line(0,1){120}}
%\put(75,0){\line(0,1){120}}
\put(90,0){\line(0,1){120}}
%\put(105,0){\line(0,1){120}}

%\put(12,-15){$\frac{1}{8}$} \put(27,-15){$\frac{2}{8}$}
%\put(42,-15){$\frac{3}{8}$} \put(57,-15){$\frac{1}{2}$}
%\put(72,-15){$\frac{5}{8}$}
\put(87,-15){$\frac{3}{4}$}
%\put(102,-15){$\frac{7}{8}$} \put(117,-15){$1$}

%\put(-8,11.5){$\frac{1}{8}$} \put(-18,26.5){$\frac{2}{8}$}
%\put(-8,41.5){$\frac{3}{8}$} \put(-18,56.5){$\frac{4}{8}$}
%\put(-8,71.5){$\frac{5}{8}$}
\put(-8,86.5){$\frac{3}{4}$} %\put(-8,101.5){$\frac{7}{8}$}

%\put(0,15){\line(1,0){120}} \put(-10,30){\line(1,0){130}}
%\put(0,45){\line(1,0){120}} \put(-10,60){\line(1,0){130}}
%\put(0,75){\line(1,0){120}}
\put(0,90){\line(1,0){130}}
%\put(0,105){\line(1,0){120}}

\put(45,90){\circle*{4}} \put(34,92){A} %
\put(90,120){\circle*{4}} \put(85,125){B} %
\put(105,90){\circle*{4}} \put(107,92){C}

\qbezier(0,0)(22.5,45)(45,90) \qbezier(45,90)(67.5,105)(90,120)
\qbezier(90,120)(97.5,105)(105,90)
\qbezier(105,90)(112.5,45)(120,0)

\put(90,90){\circle*{4}}

\end{picture}
\end{center}
\centerline{ a)}
\end{minipage}
\begin{minipage}[h]{0.49\linewidth}
\begin{center}
\begin{picture}(140,140)
%\put(120,0){\line(0,1){120}} \put(0,120){\line(1,0){120}}
\put(0,0){\vector(1,0){140}} \put(0,0){\vector(0,1){140}}

\put(0,0){\line(1,1){140}}

%\put(15,0){\line(0,1){120}} \put(30,0){\line(0,1){120}}
%\put(45,0){\line(0,1){120}} \put(60,0){\line(0,1){120}}
%\put(75,0){\line(0,1){120}}
\put(90,0){\line(0,1){120}}
%\put(105,0){\line(0,1){120}}

\put(87,-15){$\frac{1}{2}$}

\put(-8,86.5){$\frac{1}{2}$}

\put(0,90){\line(1,0){130}}

\put(45,90){\circle*{4}} \put(34,92){A} %
\put(90,120){\circle*{4}} \put(85,125){B} %
\put(105,90){\circle*{4}} \put(107,92){C}

\qbezier(0,0)(22.5,45)(45,90) \qbezier(45,90)(67.5,105)(90,120)
\qbezier(90,120)(97.5,105)(105,90)
\qbezier(105,90)(112.5,45)(120,0)

\put(90,90){\circle*{4}}

\qbezier(45,90)(45,45)(45,0) \qbezier(105,90)(105,45)(105,0)
\qbezier(90,120)(45,120)(0,120)

\put(42,-15){$\frac{1}{4}$} \put(102,-15){$\frac{3}{4}$}

\put(-8,116){$1$}

\end{picture}
\end{center}
\centerline{ b)}
\end{minipage}
\caption{The graph of the conjugated map}\label{fig:13}
\end{figure}

Thus, the conjugacy $h$ should pass throw points
$\left(\frac{1}{4}, \frac{3}{8}\right)$, $\left(\frac{1}{2},
\frac{3}{4}\right)$, and $\left(\frac{3}{4}, \frac{7}{8}\right)$.
Thus, $$ h:\, (0, 0)\rightarrow \left(\frac{1}{4},
\frac{3}{8}\right)\rightarrow \left(\frac{1}{2},
\frac{3}{4}\right)\rightarrow \left(\frac{3}{4},
\frac{7}{8}\right)\rightarrow (1, 1)
$$ seems to be a conjugacy from $f$ to $g$.
From another hand, the latter map coincides with
\begin{equation}\label{eq:3.15}
h:\, (0, 0)\rightarrow \left(\frac{1}{2},
\frac{3}{4}\right)\rightarrow (1, 1), \end{equation} since
$\left(\frac{1}{4}, \frac{3}{8}\right)$ and $\left(\frac{3}{4},
\frac{7}{8}\right)$ are not kinks. The uniqueness of the kink of
$h$ corresponds to Figure~\ref{fig:13}.

\begin{figure}
\begin{center}

\begin{picture}(300,300)

\put(150,150){\vector(1,0){150}} \put(150,150){\vector(-1,0){150}}
\put(150,150){\vector(0,1){150}} \put(150,150){\vector(0,-1){150}}

\linethickness{0.4mm}

%\Vidr{150}{150}{120}{240}
%\put(120,240){\circle*{6}}\VidrTo{90}{270}
%\put(90,270){\circle*{6}}\VidrTo{60}{240}
%\put(60,240){\circle*{6}}\VidrTo{50}{150}
%\put(50,150){\circle*{6}}\VidrTo{40}{240}
%\put(40,240){\circle*{6}}\VidrTo{30}{270}

\Vidr{150}{150}{240}{195}
\put(240,195){\circle*{6}}\VidrTo{270}{240}
\put(270,240){\circle*{6}}\Vidr{270}{240}{240}{255}
\put(240,255){\circle*{6}}\VidrTo{150}{270}

%\Vidr{150}{150}{240}{120}
%\put(240,120){\circle*{6}}\VidrTo{270}{90}
%\put(270,90){\circle*{6}}\VidrTo{240}{60}
%\put(240,60){\circle*{6}}\VidrTo{150}{50}
%\put(150,50){\circle*{6}}\VidrTo{240}{40}
%\put(240,40){\circle*{6}}\VidrTo{270}{30}

\linethickness{0.1mm}

\Vidr{150}{150}{90}{30} \VidrTo{30}{150}

\Vidr{120}{195}{240}{195} \Vidr{90}{240}{270}{240}
\Vidr{60}{255}{240}{255}

\Vidr{120}{195}{120}{90}

\Vidr{90}{240}{90}{30} \Vidr{240}{255}{240}{30}

\put(90,240){\circle*{6}} \put(240,90){\circle*{6}}
\put(240,240){\circle*{6}}

\Vidr{150}{150}{90}{240} \VidrTo{30}{270}

\Vidr{150}{150}{240}{90} \VidrTo{270}{30}

\qbezier(60,255)(60,172.5)(60,90) \qbezier(60,90)(150,90)(240,90)

%\Vidr{150}{150}{105}{60} %\put(105,60){\circle*{6}}
%\VidrTo{60}{30}
%\put(60,30){\circle*{6}}
%\Vidr{60}{30}{45}{60}
%\put(45,60){\circle*{6}}
%\VidrTo{30}{150}

\put(245,190){$A_1$} \put(275,235){$A_2$} \put(245,260){$A_3$}

%\put(126,228){$B_2$} \put(88,277){$B_4$} \put(62,228){$B_6$}
%\put(55,155){$B_8$}

%\put(245,122){$C_1$} \put(275,95){$C_2$} \put(226,65){$C_3$}

%\put(111,63){$D_3$} \put(55,16){$D_6$} \put(25,58){$D_9$}

\put(3,160){$x$} \put(137,295){$\psi$} \put(290,160){$y$}
\put(137,10){$f$}
\end{picture}
\end{center}\caption{}
\label{fig:14}
\end{figure}

When the assumption that the conjugacy $h$ is given
by~\eqref{eq:3.15} is obtained, the verification is necessary. We
can do this verification by too ways:

(i) check that the map on Figure~\ref{fig:13}b
$$
(0,0)\rightarrow \left(\frac{1}{4},\frac{1}{2}\right) \rightarrow
\left(\frac{1}{2},1\right)\rightarrow
\left(\frac{3}{4},\frac{1}{2}\right)\rightarrow \left(1,0\right)
$$ is indeed the tent map;

(ii) To check by Fig.~\ref{fig:14} that the trajectories, which
pass throw $A_1$ and $A_3$ do not contradict to the equality
$g\circ h = h\circ f$.

\begin{lemma}\label{lema:3.24}
Suppose that piecewise linear unimodal $g$ is such that
$g^{-\infty}(0)$ is dense in $[0, 1]$. Suppose that $a_i,\, 1\leq
i\leq k$ are the kinks of increasing (decreasing) part of $g$ and
$\alpha_{i},\, 1\leq i\leq k$ are the $g$-coordinates of $a_i$s.
If the conjugacy from $f$ to $g$ is piecewise linear, the the
complete set of its kinks is $\{(\alpha_i, a_i)\}\cup
\{(f(\alpha_i), g(a_i))\},\, 1\leq i\leq k$.
\end{lemma}

\begin{proof}
Lemma follows from the geometrical interpretation of the
topological conjugacy as a change of the scale.
\end{proof}

\begin{lemma}\label{lema:3.25}
Suppose that $g$ is a piecewise linear map of the
form~\eqref{eq:1.1}, and $h$ is a conjugacy. Denote $\mathcal{A}
=\{a_i,\, 1\leq i\leq k\}$ the set with the following properties:

(i) all kinks of $g$ belong to $\mathcal{A}$;

(ii) all kinks of $h$ belong to $\mathcal{A}$;

(iii) all kinks of $h$ belong to $g_l(\mathcal{A}\cap [0, v])\cap
g_r(\mathcal{A}\cap [v, 1])$;

(iv) for every $i$ the interval, bounded by $h(\alpha_i)$ and
$h(\alpha_{i+1})$ does not contain any kink of $h$;

(v) for every $i$ the point $(h(\alpha_i), (h\circ g)(\alpha_i))$
belongs to the graph of the tent map $f$.

\noindent Then $h$ is the conjugacy from $g$ to $f$.
\end{lemma}

\begin{proof}
Lemma follows from the geometrical interpretation of the
topological conjugacy as a change of the scale.
\end{proof}

\begin{lemma}\cite[Lema~10]{UMZh-2016}\label{lema:3.26}
Let $g$ be piecewise linear unimodal map, and $h$ is the conjugacy
from the tent map $f$ to $g$. Denote by $\varepsilon$ the first
kink of $g$, and $k$ the slope of $h$ at zero. Then
$\frac{2\varepsilon}{k}$ is the first kink of $h$~\footnote{There
is a typing error in~\cite{UMZh-2016}. The first kink of $h$ is
$\frac{2\varepsilon}{k}$, but not $\frac{\varepsilon}{k}$ as it is
written in~\cite{UMZh-2016}.}.
\end{lemma}

\begin{proof}
By Lemma~\ref{lema:3.21} the map $g$ is given by $g:\, x\mapsto
2x$ for sufficiently small $x$. Thus, there exists $\delta>0$ such
that $h:\, x\mapsto kx$ for all $x\leq 2\delta$ and $g:\, x\mapsto
2x$ for all $x\leq \frac{\delta}{k}$ (see Fig.~\ref{fig:15}a).

Suppose that $\gamma<\frac{1}{2}$ is such that $h:\, x\mapsto kx$
for $x\leq \gamma$ and $g:\, x\mapsto 2x$ for all $x\leq
\frac{\gamma}{k}$, i.e. $h$ has no kinks on $[0, \gamma]$ and $g$
has no kinks on $\left[ 0, \frac{\gamma}{k}\right]$. Then
functional equation \begin{equation}\label{eq:3.16} h\circ f =
g\circ h
\end{equation} determines $h$ on $[0, 2\gamma]$, as $$
h = g\circ h\circ \left\{x\mapsto \frac{x}{2}\right\}.
$$ see Fig.~\ref{fig:15}b. Moreover, in this case $$
h = \left\{x\mapsto 2x\right\}\circ \left\{x\mapsto
kx\right\}\circ \left\{x\mapsto \frac{x}{2}\right\}
=\left\{x\mapsto kx\right\},
$$ which is defined in $f\circ [0,\gamma] = [0,2\gamma]$.

\begin{figure}[htbp]
\begin{minipage}[h]{0.3\linewidth}
\begin{center}
$ \xymatrix{ [0, \delta] \ar^{x\mapsto 2x}[rr] \ar_{x\mapsto
kx}[d] && [0, 2\delta] \ar_{x\mapsto
kx}[d]\\
[0, k\delta]  \ar^{x\mapsto 2x}[rr] && [0, 2k\delta]} $
\centerline{a.}\end{center}
\end{minipage}
\hfill
\begin{minipage}[h]{0.33\linewidth}
\begin{center}
$ \xymatrix{ [0, \gamma] \ar^{x\mapsto 2x}[rr] \ar_{x\mapsto
kx}[d] &&
[0, 2\gamma] \ar@{:>}^{h}[d]\\
[0, k\gamma]  \ar^{x\mapsto 2x}[rr] && [0, 2k\gamma]} $
\centerline{b.}\end{center}
\end{minipage}
\hfill
\begin{minipage}[h]{0.3\linewidth}
\begin{center}
$ \xymatrix{ \left[0, \frac{\varepsilon}{k}\right] \ar^{x\mapsto
2x}[rr] \ar_{x\mapsto kx}[d] && \left[0,
\frac{2\varepsilon}{k}\right] \ar_{x\mapsto
kx}[d]\\
[0, \varepsilon]  \ar^{x\mapsto 2x}[rr] && [0, 2\varepsilon]}
$\\
\centerline{c.}\end{center}
\end{minipage}
\hfill \caption{} \label{fig:15}
\end{figure}

The value $\gamma =\frac{\varepsilon}{k}$ is a critical value,
when~\eqref{eq:3.16} can be expressed as a commutative diagram
from Fig.~\ref{fig:15}b, because for $x>\varepsilon$ the map $g$
in the bottom part of the diagram gas another slope. Thus critical
situation is illustrated on Fig.~\ref{fig:15}c.
\end{proof}

\begin{remark}\label{rem:3.27}
Using the notations of the current work, we can rewrite
Lemma~\ref{lema:3.26} as follows. Let a piecewise linear map $g$,
be given as
$$
g:\, (0, 0)\rightarrow (\varepsilon, 2\varepsilon)\rightarrow
\ldots .
$$ If $h$ is a piecewise linear conjugacy from the tent map to
$g$, then
$$
h:\, (0, 0)\rightarrow (h^{-1}(2\varepsilon),
2\varepsilon)\rightarrow \ldots .
$$ Moreover, denote $\alpha = h^{-1}(\varepsilon)$, whence
$$
g:\, (0, 0)\rightarrow (h(\alpha), h(2\alpha))\rightarrow \ldots
$$ and $$
h:\, (0, 0)\rightarrow (2\alpha, h(2\alpha))\rightarrow \ldots .
$$
\end{remark}

\begin{lemma}
For any $a\in (0, 1)$ the map $g$, defined in
Lemma~\ref{lema:3.9}, is topologically conjugated with the tent
map by piecewise linear conjugacy.
\end{lemma}

\begin{proof}
Denote by $a$ the maximum point of $g$. By Lemma~\ref{lema:3.24}
if the conjugacy $h$ from $g$ to $f$ is piecewise linear, then
$$h:\, (0, 0)\rightarrow \left(a,
\frac{1}{2}\right)\rightarrow (1,1).$$ Denote $\widetilde{g} =
h\circ g\circ h^{-1}$.

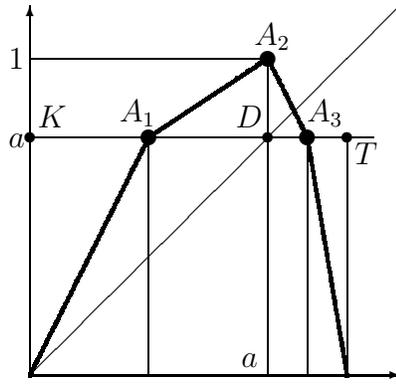
\begin{figure}[htbp]
\begin{minipage}[h]{\linewidth}
\begin{center}
\begin{picture}(140,140)
%\put(120,0){\line(0,1){120}} \put(0,120){\line(1,0){120}}
\put(0,0){\vector(1,0){140}} \put(0,0){\vector(0,1){140}}

\put(0,0){\line(1,1){140}}

%\put(15,0){\line(0,1){120}} \put(30,0){\line(0,1){120}}
%\put(45,0){\line(0,1){120}} \put(60,0){\line(0,1){120}}
%\put(75,0){\line(0,1){120}}
\put(90,0){\line(0,1){120}}
%\put(105,0){\line(0,1){120}}

\put(80,3){$a$}

\put(-8,86.5){$a$}

\put(0,90){\line(1,0){130}}

\put(45,90){\circle*{6}} \put(34,96){$A_1$} %
\put(90,120){\circle*{6}} \put(85,125){$A_2$} %
\put(105,90){\circle*{6}} \put(105,96){$A_3$} %
\put(78,94){$D$} \put(3,94){$K$} \put(123,80){$T$}

\linethickness{0.4mm} \qbezier(0,0)(22.5,45)(45,90)
\qbezier(45,90)(67.5,105)(90,120)
\qbezier(90,120)(97.5,105)(105,90)
\qbezier(105,90)(112.5,45)(120,0) \linethickness{0.1mm}

\put(90,90){\circle*{4}} \put(0,90){\circle*{4}}
\put(120,90){\circle*{4}}

\qbezier(45,90)(45,45)(45,0) \qbezier(105,90)(105,45)(105,0)
\qbezier(90,120)(45,120)(0,120)

\qbezier(120,90)(120,45)(120,0)

\put(42,-15){$\frac{a}{2}$} \put(98,-15){$\frac{a+1}{2}$}

\put(-8,116){$1$}

\end{picture}
\end{center}
%\centerline{ b)}
\end{minipage}
\caption{The graph of the conjugated map}\label{fig:16}
\end{figure}

In new coordinates point $D$ will be
$\widetilde{D}\left(\frac{1}{2}, \frac{1}{2}\right)$. Since $h$ is
linear on $(0, a)$ and $A_1$ is a middle point of $KD$, then
$\widetilde{A}_1\left(\frac{1}{4}, \frac{1}{2}\right)$.
Analogously, new coordinates of $A_3$ are
$\widetilde{A}_3\left(\frac{3}{4}, \frac{1}{2}\right)$, because
$A_3$ is the middle point of $DT$. Since $$ \widetilde{g}:\, (0,
0)\rightarrow \widetilde{A}_1\rightarrow
\widetilde{A}_2\rightarrow \widetilde{A}_3\rightarrow (1,0)
$$ is the tent map, we are done.
\end{proof}

\begin{lemma}
For any $a, b\in (0, 1)$ such that $a<b$ and $4a+b<4$ the map $g$,
defined in Lemma~\ref{lema:3.20}, is topologically conjugated with
the tent map by piecewise linear conjugacy.
\end{lemma}

$$ g:\, (0, 0)\rightarrow A_1\left(\frac{a}{2},
a\right) \rightarrow A_2(a, a+b) \rightarrow A_3\left(\frac{4a
+b}{4}, 1\right) \rightarrow$$ $$ \rightarrow A_4\left(\frac{2a
+b}{2}, a+b\right) \rightarrow A_5\left(a+b, a\right) \rightarrow
(1,0),
$$

\begin{proof}
It is seen from the form of $g$ that $g(A_1) =A_2$, $g(A_2) =A_5$,
$g(A_4) =A_2$ and $g(A_5) =A_2$. Thus, $\{A_2, A_5\}$ is a
periodical trajectory of period $2$ under the action of $g$.

Since the tent map $f$ has the unique periodical trajectory
$\left\{ \frac{2}{5}, \frac{4}{5}\right\}$ of period $2$, then for
the conjugacy from $g$ to $f$ the equalities $h\left(a\right)
=\frac{2}{5}$ and $h(a+b) =\frac{4}{5}$. Since $g(A_1)=A_2$, then
$h\left(\frac{a}{2}\right)=\frac{1}{5}$.

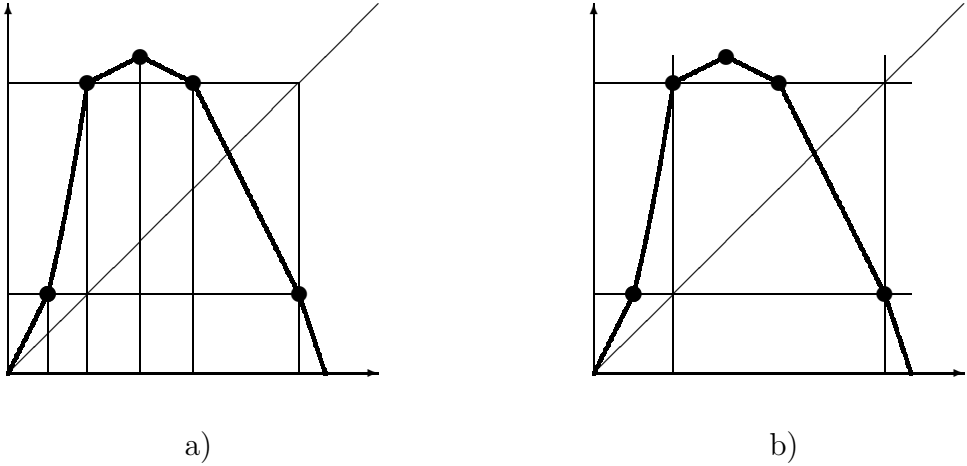
\begin{figure}[htbp]
\begin{minipage}[h]{0,45\linewidth}
\begin{center}
\begin{picture}(140,140)
%\put(120,0){\line(0,1){120}} \put(0,120){\line(1,0){120}}
\put(0,0){\vector(1,0){140}} \put(0,0){\vector(0,1){140}}

\put(0,0){\line(1,1){140}}

%\put(90,0){\line(0,1){120}}

\linethickness{0.4mm} \qbezier(0,0)(7.5,15)(15,30)
\qbezier(15,30)(22.5,60)(30,110) \qbezier(30,110)(40,115)(50,120)
\qbezier(50,120)(60,115)(70,110) \qbezier(70,110)(90,70)(110,30)
\qbezier(110,30)(115,15)(120,0) \linethickness{0.1mm}

\put(15,30){\circle*{6}} \put(30,110){\circle*{6}}
\put(50,120){\circle*{6}} \put(70,110){\circle*{6}}
\put(110,30){\circle*{6}} \put(15,30){\circle*{6}}
\put(15,30){\circle*{6}}

\qbezier(0,30)(55,30)(110,30) \qbezier(0,110)(55,110)(110,110)
\qbezier(110,110)(110,55)(110,0) \qbezier(15,30)(15,15)(15,0)

\qbezier(30,110)(30,55)(30,0) \qbezier(70,110)(70,55)(70,0)
\qbezier(50,120)(50,60)(50,0)

\end{picture}
\end{center}
\centerline{ a)}
\end{minipage}
\begin{minipage}[h]{0,45\linewidth}
\begin{center}
\begin{picture}(140,140)
%\put(120,0){\line(0,1){120}} \put(0,120){\line(1,0){120}}
\put(0,0){\vector(1,0){140}} \put(0,0){\vector(0,1){140}}

\put(0,0){\line(1,1){140}}

%\put(90,0){\line(0,1){120}}

\linethickness{0.4mm} \qbezier(0,0)(7.5,15)(15,30)
\qbezier(15,30)(22.5,60)(30,110) \qbezier(30,110)(40,115)(50,120)
\qbezier(50,120)(60,115)(70,110) \qbezier(70,110)(90,70)(110,30)
\qbezier(110,30)(115,15)(120,0) \linethickness{0.1mm}

\put(15,30){\circle*{6}} \put(30,110){\circle*{6}}
\put(50,120){\circle*{6}} \put(70,110){\circle*{6}}
\put(110,30){\circle*{6}} \put(15,30){\circle*{6}}
\put(15,30){\circle*{6}}

%\qbezier(50,120)(50,60)(50,0) \qbezier(120,50)(60,50)(0,50)

\qbezier(0,30)(60,30)(120,30) \qbezier(0,110)(60,110)(120,110)
\qbezier(110,120)(110,60)(110,0) %\qbezier(15,30)(15,15)(15,0)

\qbezier(30,120)(30,60)(30,0) %\qbezier(70,110)(70,55)(70,0)
%\qbezier(50,120)(50,60)(50,0)

\end{picture}
\end{center}
\centerline{ b)}
\end{minipage}
\caption{The graph of the conjugated map}\label{fig:17}
\end{figure}

Clearly, new coordinates of $A_3$ are $\left(\frac{1}{2},
1\right)$, whence $h\left(\frac{4a+b}{4}\right) =\frac{1}{2}$.
Finally, $g(A_4)=A_2$ implies $h\left(\frac{2a+b}{2}\right) =\max
f^{-1}\left(\frac{4}{5}\right) =\frac{2-\frac{4}{5}}{2}
=\frac{3}{5}$. Thus, $$ h:\, (0,0)\rightarrow
\widetilde{A}_1\left(\frac{a}{2}, \frac{1}{5}\right)\rightarrow
\widetilde{A}_2\left(a, \frac{2}{5}\right)\rightarrow
\widetilde{A}_3\left(\frac{4a+b}{4}, \frac{1}{2}\right)\rightarrow
$$$$ \rightarrow
\widetilde{A}_4\left(\frac{2a+b}{2}, \frac{3}{5}\right)\rightarrow
\widetilde{A}_5\left(a+b, \frac{4}{5}\right)\rightarrow (1,1).
$$
Clearly, $\widetilde{A}_1$ belongs to $O\widetilde{A}_2$.
Moreover, $$
\frac{\frac{2a+b}{2}-\frac{4a+b}{4}}{\frac{3}{5}-\frac{1}{2}}
=\frac{(a+b)-\frac{2a+b}{2}}{\frac{4}{5}-\frac{3}{5}}
$$
implies that $\widetilde{A}_4$ belongs to the interval
$(\widetilde{A}_3\widetilde{A}_5)$. Thus, $$ h:\, (0,0)\rightarrow
\widetilde{A}_2\left(a, \frac{2}{5}\right)\rightarrow
\widetilde{A}_3\left(\frac{4a+b}{4}, \frac{1}{2}\right)\rightarrow
\widetilde{A}_5\left(a+b, \frac{4}{5}\right)\rightarrow (1,1).
$$ Moreover, $$
\frac{\frac{4a+b}{4}-a}{\frac{1}{2}-\frac{2}{5}} =
\frac{(a+b)-\frac{4a+b}{4}}{\frac{4}{5} -\frac{1}{2}},
$$ which is equivalent to $$
\frac{b/4}{1/10} = \frac{3b/4}{3/10},
$$
shows that $\widetilde{A}_3$ belongs to $(\widetilde{A}_2,
\widetilde{A}_5)$, whence $h$ can be written as $$ h:\,
(0,0)\rightarrow \widetilde{A}_2\left(a,
\frac{2}{5}\right)\rightarrow \widetilde{A}_5\left(a+b,
\frac{4}{5}\right)\rightarrow (1,1).
$$ The last finishes the proof, due to diagram at
Fig.~\ref{fig:17}b.
\end{proof}

\subsection{One more complicated example}

We have shows in Section~\ref{sec:3.2} that maps $g$ in
Examples~\ref{ex:3.3} and~\ref{ex:3.13} are topologically
conjugated to the tent map $f$ by a piecewise linear conjugacy. We
have easily constructed this conjugacy, because it happened easy
to find $g$-coordinates of kinks of $g$ in the assumption that $g$
is topologically conjugated to the tent map.

We will construct below an example, where kinks of $g$ will be
$g$-irrational in the assumption that $g$ is topologically
conjugated to the tent map. Let the conjugacy $h$ be given as $$
h:\, (0, 0)\rightarrow \left(\frac{\sqrt{2}}{2},
\frac{\sqrt{3}}{2}\right)\rightarrow (1, 1).
$$
Our direct computations show that maps $g = h\circ f\circ h^{-1}$
and $\psi =h\circ \xi_3\circ h^{-1}$ are restricted by the lattice
from Figure~\ref{fig:18}.

\begin{example}\label{ex:4.9}
Describe all the maps $g$, which are restricted by the lattice
from Figure~\ref{fig:18}.\end{example}

\begin{figure}
\begin{center}
\begin{picture}(460,460)

\put(230,230){\vector(1,0){230}} \put(230,230){\vector(-1,0){230}}
\put(230,230){\vector(0,1){230}} \put(230,230){\vector(0,-1){230}}

\linethickness{0.4mm}

\put(403,403){\circle{3}}
\put(403,403){\circle{5}}\put(403,403){\circle{7}}

\put(230,230){\circle*{6}}\Vidr{230}{230}{172}{403}
\put(172,403){\circle*{6}}\Vidr{172}{403}{148}{430}
\put(148,430){\circle*{6}}\Vidr{148}{430}{124}{403}
\put(124,403){\circle*{6}}\Vidr{124}{403}{67}{230}
\put(67,230){\circle*{6}}\Vidr{67}{230}{57}{260}
\put(57,260){\circle*{6}}\Vidr{57}{260}{39}{403}
\put(39,403){\circle*{6}}\Vidr{39}{403}{30}{430}

\put(230,230){\circle*{6}}\Vidr{230}{230}{403}{317}
\put(403,317){\circle*{6}}\Vidr{403}{317}{430}{352}
\put(430,352){\circle*{6}}\Vidr{430}{352}{403}{388}
\put(403,388){\circle*{6}}\Vidr{403}{388}{373}{403}
\put(373,403){\circle*{6}}\Vidr{373}{403}{230}{430}

\put(230,230){\circle*{6}}\Vidr{230}{230}{403}{172}
\put(403,172){\circle*{6}}\Vidr{403}{172}{430}{148}
\put(430,148){\circle*{6}}\Vidr{430}{148}{403}{124}
\put(403,124){\circle*{6}}\Vidr{403}{124}{230}{67}
\put(230,67){\circle*{6}}\Vidr{230}{67}{260}{57}
\put(260,57){\circle*{6}}\Vidr{260}{57}{403}{39}
\put(403,39){\circle*{6}}\Vidr{403}{39}{430}{30}

\put(230,230){\circle*{6}}\Vidr{230}{230}{143}{57}
\put(143,57){\circle*{6}}\Vidr{143}{57}{108}{30}
\put(108,30){\circle*{6}}\Vidr{108}{30}{72}{57}
\put(72,57){\circle*{6}}\Vidr{72}{57}{57}{87}
\put(57,87){\circle*{6}}\Vidr{57}{87}{30}{230}

\linethickness{0.1mm}

\put(30,352){\line(1,0){400}} \put(30,148){\line(1,0){400}}
\put(30,317){\line(1,0){400}} \put(30,388){\line(1,0){400}}
\put(30,172){\line(1,0){400}} \put(30,39){\line(1,0){400}}
\put(30,124){\line(1,0){400}} \put(30,403){\line(1,0){400}}
\put(30,302){\line(1,0){400}} \put(30,115){\line(1,0){400}}
\put(30,43){\line(1,0){400}} \put(30,182){\line(1,0){400}}
\put(30,260){\line(1,0){400}} \put(30,419){\line(1,0){400}}
\put(30,87){\line(1,0){400}} \put(30,53){\line(1,0){400}}
\put(30,210){\line(1,0){400}} \put(30,245){\line(1,0){400}}
\put(30,57){\line(1,0){400}} \put(30,424){\line(1,0){400}}
\put(30,77){\line(1,0){400}} \put(30,57){\line(1,0){400}}
\put(30,220){\line(1,0){400}} \put(30,67){\line(1,0){400}}

\put(189,30){\line(0,1){400}} \put(108,30){\line(0,1){400}}
\put(45,30){\line(0,1){400}} \put(201,30){\line(0,1){400}}
\put(96,30){\line(0,1){400}} \put(50,30){\line(0,1){400}}
\put(177,30){\line(0,1){400}} \put(119,30){\line(0,1){400}}
\put(41,30){\line(0,1){400}} \put(403,30){\line(0,1){400}}
\put(373,30){\line(0,1){400}} \put(172,30){\line(0,1){400}}
\put(124,30){\line(0,1){400}} \put(39,30){\line(0,1){400}}
\put(206,30){\line(0,1){400}} \put(91,30){\line(0,1){400}}
\put(52,30){\line(0,1){400}} \put(289,30){\line(0,1){400}}
\put(220,30){\line(0,1){400}} \put(77,30){\line(0,1){400}}
\put(57,30){\line(0,1){400}} \put(158,30){\line(0,1){400}}
\put(138,30){\line(0,1){400}} \put(34,30){\line(0,1){400}}
\put(72,30){\line(0,1){400}} \put(260,30){\line(0,1){400}}
\put(225,30){\line(0,1){400}} \put(72,30){\line(0,1){400}}
\put(62,30){\line(0,1){400}} \put(153,30){\line(0,1){400}}
\put(143,30){\line(0,1){400}} \put(32,30){\line(0,1){400}}
\put(67,30){\line(0,1){400}} \put(148,30){\line(0,1){400}}
\end{picture}
\end{center}\caption{}
\label{fig:18}
\end{figure}

\begin{lemma}\label{lema:3.31}
If maps $g$ and $\psi$ are restricted by the lattice from
Figure~\ref{fig:18}, then %
$\sdev_1\cdot \gdev_2 = \gdev_1\cdot \sdev_2 $, \hskip 5mm
$\sdev_1\cdot \gdev_3 = \gdev_1\cdot \sdev_3, $ \hskip 5mm
$\sdev_1\cdot \gdev_4 = \gdev_1\cdot \sdev_4 $, \hskip 5mm
$\sdev_2\cdot \gdev_5 = \gdev_1\cdot \sdev_4 $, \hskip 5mm
$\sdev_3\cdot \gdev_5 = \gdev_1\cdot \sdev_5 $, \hskip 5mm
$\sdev_3\cdot \gdev_5 = \gdev_2\cdot \sdev_6$, \hskip 5mm
$\sdev_4\cdot \gdev_4 = \gdev_2\cdot \sdev_6$, \hskip 5mm
$\sdev_4\cdot \gdev_3 = \gdev_2\cdot \sdev_7$, \hskip 5mm
$\sdev_4\cdot \gdev_2 = \gdev_3\cdot \sdev_7$, \hskip 5mm
$\sdev_4\cdot \gdev_1 = \gdev_3\cdot \sdev_6$, \hskip 5mm
$\sdev_4\cdot \gdev_1 = \gdev_4\cdot \sdev_5$, \hskip 5mm
$\sdev_5\cdot \gdev_1 = \gdev_4\cdot \sdev_4$, \hskip 5mm
$\sdev_6\cdot \gdev_1 = \gdev_5\cdot \sdev_4$, \hskip 5mm
$\sdev_6\cdot \gdev_2 = \gdev_5\cdot \sdev_3$, \hskip 5mm
$\sdev_6\cdot \gdev_3 = \gdev_5\cdot \sdev_2$, \hskip 5mm
$\sdev_6\cdot \gdev_4 = \gdev_5\cdot \sdev_1$, \hskip 5mm
$\sdev_7\cdot \gdev_5 = \gdev_5\cdot \sdev_1$.
\end{lemma}

\begin{lemma}\label{lema:3.32}
The equalities from Lemma~\ref{lema:3.31} imply $\gdev_3 =
-\gdev_2, $\hskip 5mm $\gdev_4 = -\gdev_1$, \hskip 5mm $\gdev_5 =
\frac{-(\gdev_1)^2}{\gdev_2}$.
\end{lemma}

\begin{lemma}\label{lema:3.33}
The equalities from Lemma~\ref{lema:3.31} imply $\sdev_2
=\frac{\sdev_1\cdot \gdev_2}{\gdev_1}$, \hskip 5mm $\sdev_3 =
\frac{-\sdev_1\cdot \gdev_2}{\gdev_1}$, \hskip 5mm $\sdev_4 =
-\sdev_1$, \hskip 5mm $\sdev_5 = \sdev_1$, \hskip 5mm $\sdev_6 =
\frac{\sdev_1\cdot \gdev_1}{\gdev_2} $ and $\sdev_7 = \sdev_1$.
\end{lemma}

\begin{lemma}\label{lema:3.34}
Equations of Lemmas~\ref{lema:3.32} and~\ref{lema:3.33} imply
equations of Lemma~\ref{lema:3.31}.
\end{lemma}

\begin{lemma}\label{lema:3.35}
Let $g$ be restricted by the lattice from Figure~\ref{fig:18}.
Then $_\llcorner\overrightarrow{OP_1} =\left(\frac{a+b}{\gdev_1},
a+b\right)$, $_\llcorner\overrightarrow{P_1P_2}
=\left(\frac{1-a-b}{\gdev_2}, 1-a-b\right)$,
$_\llcorner\overrightarrow{P_2P_3} =\left(-\frac{1-a-b}{\gdev_3},
-1+a+b\right)$, $_\llcorner\overrightarrow{P_3P_4}
=\left(-\frac{b}{\gdev_4}, -b\right)$
 and
$_\llcorner\overrightarrow{P_4P_5} =\left(-\frac{a}{\gdev_5},
-a\right)$.
\end{lemma}

\begin{proof}
These values appear, if we consider the kinks of $g$ from the
quadrant $\psi\times y$.
\end{proof}

\begin{lemma}\label{lema:3.36}
Let the lattice be as at the Figure~\ref{fig:18}. Denote $a$ and
$b$ the $y$-lengthes of least and pre-last parts of linearity of
$g$ respectively, and $\gdev_1,\ldots, \gdev_5$ the slopes of~$g$.
Then $b = 1- \frac{a\cdot \gdev_2}{4} -a$, \hskip 5mm $ \gdev_1 =
2$, \hskip 5mm $ \gdev_3 = -\gdev_2$, \hskip 5mm $ \gdev_4 =-2$,
and $ \gdev_5 = \frac{-4}{\gdev_2}.$
\end{lemma}

\begin{proof}
Since $_\llcorner\overrightarrow{OP_1} +\ldots
+_\llcorner\overrightarrow{OP_5} = (1,0)$, then
Lemma~\ref{lema:3.35} implies

\begin{equation}\label{eq:3.17} \frac{a+b}{\gdev_1} +
\frac{1-a-b}{\gdev_2} - \frac{1-a-b}{\gdev_3} -\frac{b}{\gdev_4} -
\frac{a}{\gdev_5} = 1.
\end{equation} By Lemmas~\ref{lema:3.32} and~\ref{lema:3.33},
rewrite~\eqref{eq:3.17} as
\begin{equation}\label{eq:3.18} \frac{a+2b}{\gdev_1} +
\frac{2(1-a-b)}{\gdev_2} + \frac{a\cdot \gdev_2}{(\gdev_1)^2} = 1.
\end{equation}

Since
$$
(1,0)\cdot \overrightarrow{OQ_5}_{\lrcorner} =(1,0)\cdot
\overrightarrow{OP_4}^{\urcorner}
$$ and $$
(1,0)\cdot ^\ulcorner\overrightarrow{OQ_5} = (0,1)\cdot
\overrightarrow{OP_1}^{\urcorner},$$ then $$ (1,0)\cdot
\overrightarrow{OP_4} = (0,1)\cdot \overrightarrow{OP_1},
$$ whence
\begin{equation}\label{eq:3.19} 1+ \frac{a}{\gdev_5} = a+b.
\end{equation} By Lemma~\ref{lema:3.32}   rewrite~\eqref{eq:3.19} as
$\displaystyle{1- \frac{a\cdot \gdev_2}{(\gdev_1)^2} = a+b},$
whence
\begin{equation}\label{eq:3.20}
b = 1- \frac{a\cdot \gdev_2}{(\gdev_1)^2} -a.
\end{equation} Plug~\eqref{eq:3.20} into~\eqref{eq:3.18} and
$$ \frac{a+2\left(1- \frac{a\cdot \gdev_2}{(\gdev_1)^2}
-a\right)}{\gdev_1} + \frac{2\left(1-a-\left(1- \frac{a\cdot
\gdev_2}{(\gdev_1)^2} -a\right)\right)}{\gdev_2} + \frac{a\cdot
\gdev_2}{(\gdev_1)^2} = 1.
$$
This equality can be simplified to
$$
a\cdot \frac{(\gdev_2-\gdev_1)\cdot (\gdev_1 -2)}{(\gdev_1)^3 } =
\frac{\gdev_1 -2}{\gdev_1}.
$$ Thus, either \begin{equation}\label{eq:3.21} \gdev_1 =2,
\end{equation}
 or \begin{equation}\label{eq:3.22}
a = \frac{(\gdev_1)^2}{\gdev_2-\gdev_1}
\end{equation}
Plug~\eqref{eq:3.22} into~\eqref{eq:3.20} and obtain
\begin{equation}\label{eq:3.23}
b = 1- \frac{\gdev_2}{\gdev_2-\gdev_1}
-\frac{(\gdev_1)^2}{\gdev_2-\gdev_1} = \frac{-\gdev_1
-(\gdev_1)^2}{\gdev_2-\gdev_1}.
\end{equation}
Since~\eqref{eq:3.22} implies $\gdev_2-\gdev_1>0$, then $b<0$
in~\eqref{eq:3.23}. This contradiction proves~\eqref{eq:3.21}. Now
the proposition follows from Lemma~\ref{lema:3.31} and
equality~\eqref{eq:3.20}.
\end{proof}

\begin{lemma}\label{lema:3.37}
A map $g$ is defined by Lemma~\ref{lema:3.36}, if and only if
there is $a,\, b\in (0, 1)$ such that
$$
g:\, (0,0)\rightarrow P_1\left( a,
2a\right)\rightarrow%
P_2\left(
b, 1\right)\rightarrow%
P_3\left(2b -a,
2a\right)\rightarrow%
P_4\left( 2a, 4\cdot(b -a)\right) \rightarrow (1,0).
$$
\end{lemma}

\begin{proof}
Using Lemma~\ref{lema:3.35}, due to formulas from
Lemma~\ref{lema:3.36}, write $\overrightarrow{OP_1} =\left(
\frac{4-a\cdot \gdev_2}{8}, \frac{4-a\cdot \gdev_2}{4}\right)$,
$\overrightarrow{P_1P_2} =\left(\frac{a}{4}, \frac{a\cdot
\gdev_2}{4}\right)$, $\overrightarrow{P_2P_3} =\left(\frac{a}{4},
-\frac{a\cdot \gdev_2}{4}\right)$, $\overrightarrow{P_3P_4}
=\left(\frac{b}{2}, -b\right)$ and $\overrightarrow{P_4P_5}
=\left(-\frac{a\cdot \gdev_2}{4}, -a\right)$.

Thus, $\overrightarrow{OP_2} =\left( \frac{4-a\cdot
\gdev_2+2a}{8}, 1\right)$, $\overrightarrow{OP_3} =\left(
\frac{4-a\cdot \gdev_2+4a}{8}, 1-\frac{a\cdot \gdev_2}{4}\right)$,
and  $$\overrightarrow{OP_4} = \left( \frac{4-a\cdot \gdev_2 +4a
+4b}{8}, 1-\frac{a\cdot \gdev_2}{4} -b\right) =\left(
\frac{4-a\cdot \gdev_2}{4}, a\right).$$ Denote $\widetilde{a}
=\frac{4-a\cdot \gdev_2}{8}$ and $\widetilde{b} =\frac{4-a\cdot
\gdev_2}{8} +\frac{a}{4}$. In this case $a = 4\cdot(\widetilde{b}
-\widetilde{a})$ and

Thus, $\overrightarrow{OP_1} =\left( \widetilde{a},
2\widetilde{a}\right)$, $\overrightarrow{OP_2} =\left(
\widetilde{b}, 1\right)$, $\overrightarrow{OP_3}
=\left(2\widetilde{b} -\widetilde{a}, 2\widetilde{a}\right)$, and
$\overrightarrow{OP_4} =\left( 2\widetilde{a},
4\cdot(\widetilde{b} -\widetilde{a})\right).$
\end{proof}

\begin{lemma}\label{lema:3.38}
Let $\psi$ be restricted by the lattice from Figure~\ref{fig:18}.
Then $\sdev_1 =3$.
\end{lemma}

\begin{proof}
\begin{equation}\label{eq:3.24}
\left(\overrightarrow{OQ_1}\right)_{\lrcorner} =
\left(\frac{2a}{\sdev_1},2a\right).\end{equation}

$\left(\overrightarrow{Q_1Q_2}\right)_{\lrcorner} =
\left(\frac{1-2a}{\sdev_2},1-2a\right)$. Using~\eqref{eq:3.24}
obtain $\left(\overrightarrow{OQ_2}\right)_{\lrcorner} =
\left(\frac{2a}{\sdev_1} +\frac{1-2a}{\sdev_2},1\right)$, and by
Lemma~\ref{lema:3.33} simplify
\begin{equation}\label{eq:3.25}
\left(\overrightarrow{OQ_2}\right)_{\lrcorner} =
\left(\frac{2a}{\sdev_1} +2\cdot \frac{1-2a}{\sdev_1\cdot
\gdev_2},1\right).
\end{equation}
Since $\left(\overrightarrow{Q_2Q_3}\right)_{\lrcorner} =
\left(-\frac{1-2a}{\sdev_3}, -(1-2a)\right)
\overset{\text{Lema.~\ref{lema:3.33}}}{=} \left(2\cdot
\frac{1-2a}{\sdev_1\cdot \gdev_2}, -(1-2a)\right)$,
 then,
by~\eqref{eq:3.25}, \begin{equation}\label{eq:3.26}
\left(\overrightarrow{OQ_3}\right)_{\lrcorner}
=\left(\frac{2a}{\sdev_1} +4\cdot \frac{1-2a}{\sdev_1\cdot
\gdev_2}, 2a\right).
\end{equation}

Since $\left(\overrightarrow{Q_3Q_4}\right)_{\lrcorner} =
\left(-\frac{2a}{\sdev_4},-2a \right)
\overset{\text{Lema.~\ref{lema:3.33}}}{=}
\left(\frac{2a}{\sdev_1},-2a\right)$, then it follows
from~\eqref{eq:3.26} that $$
\left(\overrightarrow{OQ_4}\right)_{\lrcorner}
=\left(\frac{4a}{\sdev_1} +4\cdot \frac{1-2a}{\sdev_1\cdot
\gdev_2}, 0\right).
$$

Since $Q_5(P_4)^{\urcorner}$ is a vertical line, then $$ (1,
0)\cdot \left(\overrightarrow{OQ_5}\right)_{\lrcorner} =(1,
0)\cdot \left(\overrightarrow{OP_4}\right) =2a.
$$ Thus, $$
\left(\overrightarrow{OQ_5}\right)_{\lrcorner} =\left(2a,
\sdev_5\cdot \left(2a-\frac{4a}{\sdev_1} -4\cdot
\frac{1-2a}{\sdev_1\cdot \gdev_2}\right)\right)
\overset{\text{Lema.~\ref{lema:3.33}}}{=}$$ $$ =\left(2a,
\sdev_1\cdot \left(2a-\frac{4a}{\sdev_1} -4\cdot
\frac{1-2a}{\sdev_1\cdot \gdev_2}\right)\right)=
$$$$
=\left(2a, 2a\cdot \sdev_1-4a -4\cdot
\frac{1-2a}{\gdev_2}\right)$$

Clearly, $$ \left(\overrightarrow{OQ_6}\right)_{\lrcorner} =
\left(1-\frac{1-2a}{\sdev_7}, 2a\right)
\overset{\text{Lema.~\ref{lema:3.33}}}{=}
\left(1-\frac{1-2a}{\sdev_1}, 2a\right).$$ Since the slope of
$Q_5Q_6$ is $\sdev_6$, then $$ \sdev_6 = \frac{2a-\left(2a\cdot
\sdev_1-4a -4\cdot
\frac{1-2a}{\gdev_2}\right)}{1-\frac{1-2a}{\sdev_1}-2a}
$$

By Lema~\ref{lema:3.33}, $\sdev_6 = \frac{\sdev_1\cdot
\gdev_1}{\gdev_2} $, whence $$ \frac{\sdev_1}{\gdev_2} =
\frac{3a-a\cdot \sdev_1 +2\cdot
\frac{1-2a}{\gdev_2}}{1-\frac{1-2a}{\sdev_1}-2a}.
$$
This expression can be simplified to
$$
0  = (a\cdot \gdev_2 + 1-2a) \cdot (3-\sdev_1)
$$

Thus, either $\sdev_1 =3$, or $$ \gdev_2 = -\frac{1-2a}{a},
$$ but the latter equality is impossible, because $\gdev_2>0$.

\end{proof}

It is seen from Fig.~\ref{fig:18} that if a piecewise linear
conjugacy from $f$ to $g$ exists, then it has the unique kink
$h(t) =2a$ for some $t>1/2$. In this case $$
h\left(\frac{1}{2}\right) =h(t)\cdot \frac{1/2}{t} = \frac{a}{t}.
$$
Since $P_2$ is the maximum point of $g$, then
$h\left(\frac{1}{2}\right) =b$, whence $$ t = \frac{a}{b}.
$$
Thus, if $g$ is topologically conjugated with the tent map via
piecewise linear conjugacy, the following lemma should hold.

\begin{lemma}
Let $g$ be as in Lemma~\ref{lema:3.37}. Then the map $$ h:\,
(0,0)\rightarrow \left(\frac{a}{b}, 2a\right)\rightarrow (1,1)
$$ is a conjugacy from $f$ to $g$.
\end{lemma}

\begin{proof}
The conjugacy from $g$ to $f$ is $$h^{-1}:\, (0,0)\rightarrow
\left(2a, \frac{a}{b}\right)\rightarrow (1,1),
$$ whence $$
h^{-1}(x) =\left\{ \begin{array}{ll} \frac{x}{2b},& \text{if
} x\leq 2a,\\
\frac{2a}{b} +\frac{x-2a}{1-2a}\cdot \left(1 -\frac{2a}{b}\right)&
\text{otherwise.}\end{array}\right.
$$ Thus, $$
h(P_1) =\left(\frac{a}{2b}, \frac{a}{b}\right)\in \Gamma_f,
$$$$
h(P_2) =\left(\frac{1}{2}, 1\right)\in \Gamma_f,
$$$$
h(P_3)= \left(\frac{2b-a}{2b}, \frac{a}{b}\right)=\left(1
-\frac{a}{2b}, \frac{a}{b}\right)\in \Gamma_f,
$$$$
h(P_4) =\left(\frac{a}{b}, \frac{4(b-a)}{2b}\right)
=\left(\frac{a}{b}, 2 -\frac{2a}{b}\right)\in \Gamma_f.
$$
Now the lemma follows from Lemma~\ref{lema:3.25}.
\end{proof}

\section{The case, when the unimodal map is topologically
conjugated with the tent}

We will use Theorem~\ref{th:3}, together with some preliminaries
in this section to simplify the reasonings from
Examples~\ref{ex:3.1}, \ref{ex:3.3}, \ref{ex:3.11}, \ref{ex:3.13}
and~\ref{ex:4.9}.

\subsection{Preliminaries}

We have proved in~\cite[Theorem~1]{Odesa-2016} (see
also~\cite[Theorem~1]{Plakh-Arx-Odesa}) that for every surjective
continuous map $\xi$, which commutes with the tent map, there
exists and integer positive $t$ such that
\begin{equation}\label{eq:4.1}
\xi = \xi_t:\, x \mapsto \displaystyle{\frac{1 - (-1)^{[tx]}}{2}
+(-1)^{[tx]}\{tx\}},\end{equation} where $\{\cdot \}$ denotes the
function of the fractional part of a number and $[\cdot ]$ is the
integer part. Notice that~\eqref{eq:4.1} describes a piecewise
linear function $\xi_t: [0, 1]\rightarrow [0, 1]$, whose tangents
are $\pm t$, which passes through origin, and all whose kinks
belong to lines $y =0$ and $y =1$. The graphs of $\xi_5$ and
$\xi_6$ are given at Fig.~\ref{fig:19}.

\begin{figure}[ht]
\begin{center}
\begin{picture}(140,135)
\put(0,0){\vector(1,0){140}} \put(0,0){\vector(0,1){135}}
%\put(0,0){\line(1,1){135}}

\linethickness{0.4mm} \Vidr{0}{0}{24}{120} \VidrTo{48}{0}
\VidrTo{72}{120} \VidrTo{96}{0} \VidrTo{120}{120}
\linethickness{0.1mm}
\end{picture}\hskip 3cm
\begin{picture}(140,135)
\put(0,0){\vector(1,0){140}} \put(0,0){\vector(0,1){135}}
%\put(0,0){\line(1,1){135}}

\linethickness{0.4mm} \Vidr{0}{0}{20}{120} \VidrTo{40}{0}
\VidrTo{60}{120} \VidrTo{80}{0} \VidrTo{100}{120} \VidrTo{120}{0}
\linethickness{0.1mm}
\end{picture}
\end{center}
\caption{Graphs of $\xi_5$ and $\xi_6$} \label{fig:19}\end{figure}
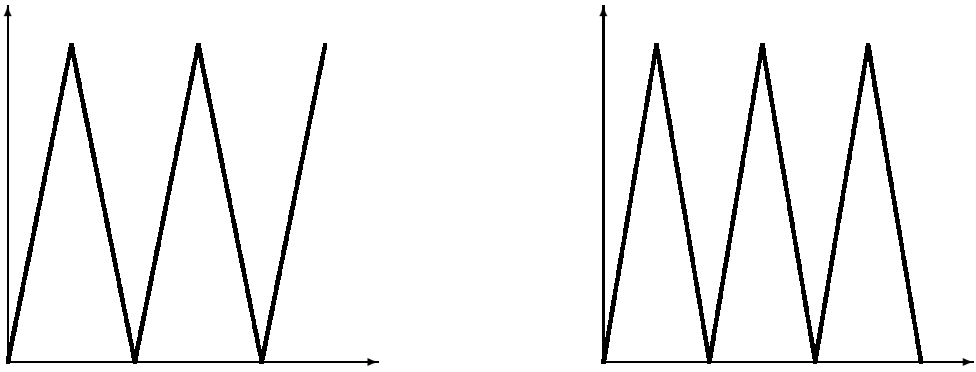

\setlength{\unitlength}{1pt}

\begin{lemma}\label{lema:4.1}
Let $h$ be the conjugacy from the tent map $f$ to a piecewise
linear unimodal map~$g$. Suppose the map $g$ commutes with a
non-constant $\psi$, which has $t,\, t\geq 1$ maximal parts of
monotonicity. Then
$$
\psi = h\circ \xi_t\circ h^{-1},$$ where $\xi_t$ is self semi
conjugacy of the tent map, determined by~\eqref{eq:4.1}. Moreover,
if $\psi$ is not a constant map, then $t$ is the number of maximal
parts of monotonicity of $\psi$.
\end{lemma}

\begin{proof}
It is easily seen from the commutative diagram at
Fig.~\ref{fig:20} that $$ \xi = h^{-1}\circ \psi\circ h$$ will be
a commutator of $f$. Thus, by~\cite[Theorem~1]{Odesa-2016} (or
also~\cite[Theorem~1]{Plakh-Arx-Odesa}), there is $p\geq 1$ such
that $\xi_p = h^{-1}\circ \psi\circ h$, whence
$$ \psi = h\circ \xi_p\circ h^{-1}.
$$ Since $p$ is the number of maximal
parts of monotonicity of $\xi$, then $p =t$ and we are done.

\begin{figure}[ht]
$$\xymatrix{ [0, 1] \ar^{f}[rr] \ar_{h}[d] && [0, 1]
\ar^{h}[d]\\
[0, 1] \ar^{g}[rr] \ar_{\psi}[d] && [0, 1]
\ar^{\psi}[d] \\
[0, 1] \ar^{g}[rr] \ar_{h^{-1}}[d] && [0, 1] \ar^{h^{-1}}[d]\\
[0, 1] \ar^{f}[rr] && [0, 1] }$$ %
\caption{Construction of a commutator of $f$ by given commutator
of $g$}\label{fig:20}
\end{figure}
\end{proof}

We are ready now to formulate an important computation, which we
will use later.

\begin{lemma}\label{lema:4.2}
Let $h$ be the conjugacy from the tent map $f$ to a unimodal
map~$g$. Suppose that $g$ commutes with a map $\psi$, and $\xi_t$,
is determined by Lemma~\ref{lema:4.1}. Then

\noindent (i) $$ \psi(h(x)) = h( \xi_t(x))
$$ for all $x\in [0, 1]$;

\noindent (ii) $$g(h(x)) = h(f(x))$$ for all $x\in [0, 1]$.
\end{lemma}

Notice that Lemma~\ref{lema:4.2} means that $g$ and $\psi$ act on
the set $\{h(x),\, x\in [0, 1]\}$ ``similarly'' to $f$ and $\xi_t$
respectively.

\subsection{Examples}

The idea of the consideration of Examples~\ref{ex:3.1},
\ref{ex:3.3}, \ref{ex:3.11}, \ref{ex:3.13} and~\ref{ex:4.9} will
be as follows.

Denote by $h$ the conjugacy from the tent map to $g$. For any kink
of $g$ denote by $a$ such that $h(a)$ is the $x$-coordinate of
this kink. Thus, we denote by $h(a_1),\ldots, h(a_m)$ the kinks of
$g$. Using Lemma~\ref{lema:4.2}, we can determined all the kinks
of the diagram (which corresponds to the example, which we will
consider) as application of $h$ to some linear combination of
$a_1,\ldots, a_m$. After this, equality of points, which is
determined by sets $\mathcal{P}$ and $\mathcal{Q}$ will give us
some expressions on $a_1,\ldots, a_m$. The latter with lead to a
description of the graph of $g$.

Since lattice from all considered examples correspond to a map
$\psi$ with three maximal parts of monotonicity, it will be
convenient for us to have explicit formulas for $\xi_3$ and
$\xi_3^{-1}$. Thus,
$$ \xi_3(x) = \left\{
\begin{array}{ll} 3x& \text{if
}x\leq 1/3,\\
-3x+2&\text{if }1/3< x\leq 2/3,\\
3x-2&\text{if }2/3< x
\end{array}\right.
$$ and $$
\xi_3^{-1}(x) = \left\{ \frac{x}{3},\, \frac{2-x}{3},\,
\frac{2+x}{3}\right\}.
$$

\begin{example}
Describe the maps $g$, which are topologically conjugated to the
tent map, and are restricted by the lattice from
Figure~\ref{fig:08}.
\end{example}

It follows from linearity of $g$ on $\left(0,
h\left(\frac{1}{2}\right)\right)$ and on $\left(
h\left(\frac{1}{2}\right),1\right)$ that
\begin{equation}\label{eq:4.2}
\gdev_1 = \frac{1}{h\left(\frac{1}{2}\right)} \end{equation} and
\begin{equation}\label{eq:4.3}
-\gdev_2 = \frac{1}{1- h\left(\frac{1}{2}\right)}
\end{equation}

By~(ii) of Lemma~\ref{lema:4.2},
\begin{equation}\label{eq:4.4}
h\left(\frac{2}{3}\right) = \gdev_1\cdot
h\left(\frac{1}{3}\right).
\end{equation}

Since $\psi$ is linear on $\left(h\left(\frac{1}{3}\right),
h\left(\frac{2}{3}\right)\right)$, then
\begin{equation}\label{eq:4.5}
\frac{1}{h\left(\frac{2}{3}\right)-h\left(\frac{1}{3}\right)}
=\frac{1-h\left(\frac{1}{2}\right)}{h\left(\frac{1}{2}\right)
-h\left(\frac{1}{3}\right)}.
\end{equation}

Plug~\eqref{eq:4.2} and~\eqref{eq:4.4} into~\eqref{eq:4.5} and
obtain
$$
\frac{1}{(\gdev_1-1)\cdot h\left(\frac{1}{3}\right)} = \frac{1
-\frac{1}{\gdev_1}}{\frac{1}{\gdev_1} -h\left(\frac{1}{3}\right)},
$$ which can be simplified to
$$
h\left(\frac{1}{3}\right)= \frac{1}{\gdev_1^2 -\gdev_1 +1}.
$$
Now,~\eqref{eq:4.4} implies $$ h\left(\frac{2}{3}\right)=
\frac{\gdev_1}{\gdev_1^2 -\gdev_1 +1}.
$$
Since, by~(ii) of Lemma~\ref{lema:4.2}, the point
 $h\left(\frac{2}{3}\right)$ is a fixed point of $g$, then
$$ -\gdev_2
=\frac{h\left(\frac{2}{3}\right)}{1-h\left(\frac{2}{3}\right)}.
$$
Express $\gdev_2$ from~\eqref{eq:4.2} and~\eqref{eq:4.3} and plug
into the latter equality, whence
$$
\frac{1}{1-\frac{1}{\gdev_1}} =\frac{\gdev_1}{\gdev_1^2 -\gdev_1
+1}\cdot \frac{1}{1-\frac{\gdev_1}{\gdev_1^2 -\gdev_1 +1}}.
$$ This can be simplified to $$
\frac{\gdev_1}{1-\gdev_1} =\frac{\gdev_1}{(\gdev_1-1)^2},
$$ which gives $$
\gdev_1 = 2.
$$ The obtained equality means that $g$ is the tent map.

\begin{example}
Describe the maps $g$, which are topologically conjugated to the
tent map, and are restricted by the lattice from
Figure~\ref{fig:09}.
\end{example}

Let $a\in (0, 1)$ be such that the first kink of $g$ is $h(a)$.
Thus, the kinks of $g$ are $A_1(h(a), h(2a))$,
$A_2\left(h\left(\frac{1}{2}\right), 1\right)$ and $A_3 =(h(1-a),
h(2a))$.

Kinks of $\psi$ in the quadrant $x\times \psi$ are %
$B_2\left(h\left(\frac{1}{6}\right),
h\left(\frac{1}{2}\right)\right)$, %
$B_6\left(h\left(\frac{1}{2}\right),h\left(\frac{1}{2}\right)\right)$
and %
$B_{10}
=\left(h\left(\frac{5}{6}\right),h\left(\frac{1}{2}\right)\right)$.
Next, $D_3\left(h\left(\frac{1-a}{3}\right),
h\left(\frac{2-2a}{3}\right)\right)$,
$D_9\left(h\left(\frac{2+a}{3}\right),
h\left(2-\frac{4+2a}{3}\right)\right)
=\left(h\left(\frac{2+a}{3}\right),
h\left(\frac{2-2a}{3}\right)\right)$,
$C_1\left(h\left(\frac{2a}{3}\right), h\left(2a\right)\right)$,
$C_3\left(h\left(\frac{2-2a}{3}\right), h\left(2a\right)\right)$
and $C_5\left(h\left(\frac{2+2a}{3}\right), h(2a)\right)$.

Since $\overrightarrow{OA_1} = \overrightarrow{OD_3}$,
$\overrightarrow{OA_3} = \overrightarrow{OD_9}$,
$\overrightarrow{OB_2} = \overrightarrow{OC_1}$,
$\overrightarrow{OB_6} = \overrightarrow{OC_3}$ and
$\overrightarrow{OB_{10}} = \overrightarrow{OC_5}$, then
$$\overrightarrow{OA_1} = \overrightarrow{OD_3}
\Longrightarrow (h(a), h(2a)) = \left(h\left(\frac{1-a}{3}\right),
h\left(\frac{2-2a}{3}\right)\right),$$
$$\overrightarrow{OA_3} =
\overrightarrow{OD_9} \Longrightarrow (h(1-a), h(2a)) =
\left(h\left(\frac{2+a}{3}\right),
h\left(\frac{2-2a}{3}\right)\right),$$
$$\overrightarrow{OB_2} =
\overrightarrow{OC_1} \Longrightarrow
\left(h\left(\frac{1}{6}\right), h\left(\frac{1}{2}\right)\right)
= \left(h\left(\frac{2a}{3}\right), h\left(2a\right)\right),$$
$$\overrightarrow{OB_6} =
\overrightarrow{OC_3} \Longrightarrow
\left(h\left(\frac{1}{2}\right),h\left(\frac{1}{2}\right)\right) =
\left(h\left(\frac{2-2a}{3}\right), h\left(2a\right)\right)$$ and
$$\overrightarrow{OB_{10}} =
\overrightarrow{OC_5} \Longrightarrow
\left(h\left(\frac{5}{6}\right),h\left(\frac{1}{2}\right)\right) =
\left(h\left(\frac{2+2a}{3}\right), h(2a)\right).$$ As $h$
increase, then these equalities are equivalent to $a=
\frac{1}{4}$. Moreover, since the all equalities $A_i=D_j$, where
$(i,j)\in \mathcal{P}$ and $B_i=C_j$, where $(i,j)\in\mathcal{Q}$
are satisfied, then $a=1/4$ is equivalent to the commutativity of
$g$ and $\psi$.

Denote $b=h\left(\frac{1}{4}\right)$, whence
$h\left(\frac{1}{2}\right) = (g\circ h)\left(\frac{1}{4}\right) =
g(b)$, and $$ g:\, (0,0)\rightarrow
\left(h\left(\frac{1}{4}\right),
h\left(\frac{1}{2}\right)\right)\rightarrow
\left(h\left(\frac{1}{2}\right), 1\right)\rightarrow
\left(h\left(\frac{3}{4}\right),
h\left(\frac{1}{4}\right)\right)\rightarrow (1,0).
$$

By Lemma~\ref{lema:3.21}, we have $g'(0)=2$, whence $$
h\left(\frac{1}{2}\right) = 2\cdot h\left(\frac{1}{4}\right).
$$
Thus, denote $a = h\left(\frac{1}{2}\right)$, whence the
increasing part of $g$ is $$ g:\, (0,0)\rightarrow
\left(\frac{a}{2}, a\right)\rightarrow \left(a, 1\right).
$$ This increasing part defines the conjugacy $$
h:\, (0,0)\rightarrow
\left(\frac{1}{4},\frac{a}{2}\right)\rightarrow
\left(\frac{1}{2},a\right)\rightarrow (1,1),
$$ which is the same as $$
h:\, (0,0)\rightarrow \left(\frac{1}{2},a\right)\rightarrow
(1,1).$$ Thus, $h\left(\frac{3}{4}\right) =\frac{1+a}{2}$, whence
$$
g:\, (0,0)\rightarrow \left(\frac{a}{2}, a\right)\rightarrow
\left(a, 1\right)\rightarrow \left(\frac{1+a}{2}, \frac{a}{2}
\right)\rightarrow (1,0)
$$ is the general form of the map, which
is topologically conjugated to the tent map, and is restricted by
the lattice from Figure~\ref{fig:09}.

\begin{proposition}
Let $g$ be a piecewise linear map. It is restricted by the lattice
from Figure~\ref{fig:09} if and only if there exists $a$ such that
$$g:\, (0,0)\rightarrow \left(\frac{a}{2}, a\right)\rightarrow
\left(a, 1\right)\rightarrow \left(\frac{1+a}{2}, \frac{a}{2}
\right)\rightarrow (1,0).$$ Moreover, in this case the map
$$
h:\, (0,0)\rightarrow \left(\frac{1}{2},a\right)\rightarrow (1,1).
$$ provides the conjugation from the tent map to $g$.
\end{proposition}

\begin{example}
Describe the maps $g$, which are topologically conjugated to the
tent map, and are restricted by the lattice from
Figure~\ref{fig:10}.
\end{example}

Denote $a\in (0, 1/2)$ such that $h(a)$ is the first kink of $g$.
Thus, by Lemma~\ref{lema:3.21}, this kink will be $(h(a), 2h(a))$,
whence $$ g:\, (0,0)\rightarrow (h(a), 2h(a))\rightarrow
\left(h\left(\frac{1}{2}\right), 1\right)\rightarrow (h(1-a),
2h(a))\rightarrow (1, 0).
$$

Thus, the radius vectors to the consequent kinks of $\psi$ are
$\overline{OB_2} =\left(h\left(\frac{1}{6}\right),
h\left(\frac{1}{2}\right)\right)$,
$\overrightarrow{OB_4}=\left(h\left(\frac{1}{3}\right), 1\right)$,
$\overrightarrow{OB_6} =\left(h\left(\frac{1}{2}\right),
h\left(\frac{1}{2}\right)\right)$, $\overrightarrow{OB_8}
=(h(2/3),0)$ and $\overrightarrow{OB_{10}} =(h(5/6), h(1/2))$.

By geometrical construction of Diagram~\ref{fig:10} obtain that
$OD_2 = (h(1/6), h(1/3))$, $OD_{10} = (h(5/6), h(1/3))$,
$\overrightarrow{OC_1} =(h(2a/3), 2h(a))$, $\overrightarrow{OC_3}
= (h(2(1-a)/3), 2h(a)$ and $\overrightarrow{OC_5} =
(h(\frac{2-2(2-a)}{3}), 2h(a)) =(h(\frac{2 +2a}{3}), 2h(a))$.

Since $OD_2 = OA_1$ and $OB_2 = OC_1$, then \begin{equation}
\label{eq:4.6} (h(1/6), h(1/3)) = (h(a), 2h(a)),\end{equation} and
\begin{equation} \label{eq:4.7}
\left(h\left(\frac{1}{6}\right), h\left(\frac{1}{2}\right)\right)
= \left(h\left(\frac{2a}{3}\right), 2h(a)\right).
\end{equation}
Equations~\eqref{eq:4.6} and~\eqref{eq:4.7} imply that $a = 1/6$
and, in the same time, $a = 1/4$, which gives a contradiction.

\begin{example}
Describe the maps $g$, which are topologically conjugated to the
tent map, and are restricted by the lattice from
Figure~\ref{fig:11}.
\end{example}

Denote $a,\, b\in (0, 1/2)$ such that $a<b$, and $h(a)$ and $h(b)$
are the first two kinks of $g$. Thus,
$$
g:\, (0,0)\rightarrow A_1(h(a), 2h(a)) \rightarrow A_2(h(b),
h(2b))\rightarrow A_3\left(h\left(\frac{1}{2}\right),
1\right)\rightarrow $$$$ \rightarrow A_4(h(1-b), h(2b))\rightarrow
A_5(h(1-a), 2h(a))\rightarrow (1, 0).
$$
The consequent radius vectors to the consequent kinks of $\psi$
are, according to the left top part of Figure~\ref{fig:11}, are
$\overrightarrow{OB_2} =\left(h\left(\frac{b}{3}\right),
h(b)\right)$, $\overrightarrow{OB_5}
=\left(h\left(\frac{1-a}{3}\right), h(1-a)\right)$,
$\overrightarrow{OB_7} =\left(h\left(\frac{2-(1-a)}{3}\right),
h(1-a)\right) =\left(h\left(\frac{1+a}{3}\right), h(1-a)\right)$,
$\overrightarrow{OB_{10}} =\left(h\left(\frac{2-b}{3}\right),
h(b)\right)$, $\overrightarrow{OB_{14}}
=\left(h\left(\frac{b+2}{3}\right), h(b)\right)$ and
$\overrightarrow{OB_{17}} =\left(h\left(\frac{(1-a)+2}{3}\right),
h(1-a)\right) =\left(h\left( \frac{3-a}{3}\right), h(1-a)\right)$.

Using the diagram from Figure~\ref{fig:11}, we can calculate the
coordinates of radius vectors $\overrightarrow{OD_4}
=\left(h\left(\frac{1-b}{3}\right),
h\left(\frac{2-2b}{3}\right)\right)$, $\overrightarrow{OD_7}
=\left(h\left(\frac{2-(1-a)}{3}\right),
h\left(\frac{2(2-(1-a))}{3}\right)\right)
=\left(h\left(\frac{a+1}{3}\right),
h\left(\frac{2a+2}{3}\right)\right)$, $\overrightarrow{OD_{11}}
=\left( h\left(\frac{2-a}{3}\right),
h\left(\frac{2a+2}{3}\right)\right)$; $\overrightarrow{OD_{14}}
=\left(h\left(\frac{b+2}{3}\right),
h\left(2-2\cdot\frac{b+2}{3}\right)\right)$,
$\overrightarrow{OC_1} =\left( h\left(\frac{2a}{3}\right),
h(2a)\right)$, $\overrightarrow{OC_2} =\left(
h\left(\frac{2b}{3}\right), h(2b)\right)$, $\overrightarrow{OC_4}
=\left( h\left(\frac{2-2b}{3}\right), h(2b)\right)$,
$\overrightarrow{OC_5} =\left( h\left(\frac{2-2a}{3}\right),
h(2a)\right)$, $\overrightarrow{OC_7} =\left(
h\left(\frac{2+2a}{3}\right), h(2a)\right)$ and
$\overrightarrow{OC_8} =\left( h\left(\frac{2+2b}{3}\right),
h(2b)\right)$.

Comparing the kinks of two graphs of $g$ and $\psi$ at the
Figure~\ref{fig:11}, conclude that $\overrightarrow{OB_2}
=\overrightarrow{OC_1}$, %
$\overrightarrow{OB_5} =\overrightarrow{OC_2}$,
$\overrightarrow{OB_7} =\overrightarrow{OC_4}$,
$\overrightarrow{OB_{10}} =\overrightarrow{OC_5}$,
$\overrightarrow{OB_{14}} =\overrightarrow{OC_7}$,
$\overrightarrow{OB_{17}} =\overrightarrow{OC_8}$,
$\overrightarrow{OA_1} =\overrightarrow{OD_4}$,
$\overrightarrow{OA_2} =\overrightarrow{OD_7}$,
$\overrightarrow{OA_4} =\overrightarrow{OD_{11}}$ and
$\overrightarrow{OA_5} =\overrightarrow{OD_{14}}$. Thus,
\begin{equation}\label{eq:4.8}
\left( h\left(\frac{2b}{3}\right), h(2b)\right)=
\left(h\left(\frac{1-a}{3}\right), h(1-a)\right), \end{equation}$$
\left(h\left(\frac{1-a}{3}\right), h(1-a)\right) = \left(
h\left(\frac{2b}{3}\right), h(2b)\right),
$$$$
\left(h\left(\frac{1+a}{3}\right), h(1-a)\right) = \left(
h\left(\frac{2-2b}{3}\right), h(2b)\right),
$$$$
\left(h\left(\frac{2-b}{3}\right), h(b)\right) = \left(
h\left(\frac{2-2a}{3}\right), h(2a)\right),$$
$$
\left(h\left(\frac{b+2}{3}\right), h(b)\right) = \left(
h\left(\frac{2+2a}{3}\right), h(2a)\right),
$$$$
\left(h\left( \frac{3-a}{3}\right), h(1-a)\right) = \left(
h\left(\frac{2+2b}{3}\right), h(2b)\right),
$$$$
(h(a), 2h(a)) = \left(h\left(\frac{1-b}{3}\right),
h\left(\frac{2-2b}{3}\right)\right),
$$$$
(h(b), h(2b)) = \left(h\left(\frac{a+1}{3}\right),
h\left(\frac{2a+2}{3}\right)\right),
$$$$
(h(1-b), h(2b)) = \left( h\left(\frac{2-a}{3}\right),
h\left(\frac{2a+2}{3}\right)\right),
$$\begin{equation}\label{eq:4.9}
(h(1-a), 2h(a)) = \left(h\left(\frac{b+2}{3}\right),
h\left(2-2\cdot\frac{b+2}{3}\right)\right).
\end{equation}
Since $h$ increase, then equations from~\eqref{eq:4.8}
to~\eqref{eq:4.9} imply $$
\left\{\begin{array}{l} a = 1/5,\\
b = 2/5.
\end{array}\right.
$$
Rewrite the expression for $g$ as $$ g:\, (0,0)\rightarrow
A_1\left(h\left(\frac{1}{5}\right),
h\left(\frac{2}{5}\right)\right) \rightarrow
A_2\left(h\left(\frac{2}{5}\right),
h\left(\frac{4}{5}\right)\right)\rightarrow
A_3\left(h\left(\frac{1}{2}\right), 1\right)\rightarrow $$$$
\rightarrow A_4\left(h\left(\frac{3}{5}\right),
h\left(\frac{4}{5}\right)\right)\rightarrow
A_5\left(h\left(\frac{4}{5}\right),
h\left(\frac{2}{5}\right)\right)\rightarrow (1, 0).
$$

Denote $\widetilde{a} =h\left(\frac{2}{5}\right)$, $\widetilde{b}
=h\left(\frac{1}{2}\right)$, $\widetilde{c} =
h\left(\frac{3}{5}\right)$ and $\widetilde{d}
=h\left(\frac{4}{5}\right)$. Then $$ g:\, (0,0)\rightarrow
A_1\left(\frac{\widetilde{a}}{2}, \widetilde{a}\right) \rightarrow
A_2\left(\widetilde{a}, \widetilde{d}\right)\rightarrow
A_3\left(\widetilde{b}, 1\right)\rightarrow $$$$ \rightarrow
A_4\left(\widetilde{c}, \widetilde{d}\right)\rightarrow
A_5\left(\widetilde{d}, \widetilde{a}\right)\rightarrow (1, 0).
$$

By Lemma~\ref{lema:3.24} the conjugacy from $f$ to $g$ is $$ h:\,
(0,0)\rightarrow \left(\frac{2}{5},
\widetilde{a}\right)\rightarrow \left(\frac{1}{2},
\widetilde{b}\right)\rightarrow \left(\frac{4}{5},
\widetilde{d}\right)\rightarrow (1, 1).
$$ Since $g$ is linear on $(\widetilde{c}, \widetilde{d})$, then
$h$ is linear on $(h^{-1}(\widetilde{a}), h^{-1}(\widetilde{d}))$,
whence $$ h:\, (0,0)\rightarrow \left(\frac{2}{5},
\widetilde{a}\right)\rightarrow \left(\frac{4}{5},
\widetilde{d}\right)\rightarrow (1, 1).
$$
Now $$ h\left(\frac{1}{2}\right) =\widetilde{a}
+\frac{\left(\frac{1}{2}-\frac{2}{5}\right)\cdot
(\widetilde{d}-\widetilde{a})}{\frac{4}{5}-\frac{2}{5}}
=\frac{3\widetilde{a}+\widetilde{d}}{4}
$$ and $$
h\left(\frac{3}{5}\right) =\widetilde{a}
+\frac{\left(\frac{3}{5}-\frac{2}{5}\right)\cdot
(\widetilde{d}-\widetilde{a})}{\frac{4}{5}-\frac{2}{5}}
=\frac{\widetilde{a}+\widetilde{d}}{2}.
$$ Finally, $$
g:\, (0,0)\rightarrow A_1\left(\frac{\widetilde{a}}{2},
\widetilde{a}\right) \rightarrow A_2\left(\widetilde{a},
\widetilde{d}\right)\rightarrow
A_3\left(\frac{3\widetilde{a}+\widetilde{d}}{4},
1\right)\rightarrow $$$$ \rightarrow
A_4\left(\frac{\widetilde{a}+\widetilde{d}}{2},
\widetilde{d}\right)\rightarrow A_5\left(\widetilde{d},
\widetilde{a}\right)\rightarrow (1, 0).
$$ is the general form of the map, which
is topologically conjugated to the tent map, and is restricted by
the lattice from Figure~\ref{fig:11}.

\begin{proposition}
Let $g$ be a piecewise linear map. It is restricted by the lattice
from Figure~\ref{fig:11} if and only if there exist $a$ and $b$
such that
$$
g:\, (0,0)\rightarrow A_1\left(\frac{a}{2}, a\right) \rightarrow
A_2\left(a, b\right)\rightarrow A_3\left(\frac{3a+b}{4},
1\right)\rightarrow $$$$ \rightarrow A_4\left(\frac{a+b}{2},
b\right)\rightarrow A_5\left(b, a\right)\rightarrow (1, 0).
$$ Moreover, in this case the map
$$ h:\, (0,0)\rightarrow \left(\frac{2}{5},
a\right)\rightarrow \left(\frac{4}{5}, b\right)\rightarrow (1, 1)
$$ provides the conjugation from the tent map to $g$.
\end{proposition}

\begin{example}
Describe the maps $g$, which are topologically conjugated to the
tent map, and are restricted by the lattice from
Figure~\ref{fig:18}.
\end{example}

Denote $a$ such that $h(a)$ is the first kink of $g$, and denote
$b$ such that $h(b)$ is the last kink of $g$. Thus, the consequent
kinks of $g$ are $\overrightarrow{OA_4} =(h(a), h(2a))$,
$\overrightarrow{OA_6} =(h(1-a), h(2a))$ and
$\overrightarrow{OA_7} =(h(b), h(2-2b))$. The radius vectors of
kinks of $\psi$, which we need for our further computations are
$\overrightarrow{OB_7} =\left(h\left(\frac{b}{3}\right),
h(b)\right)$, $\overrightarrow{OB_{13}}
=\left(h\left(\frac{2-b}{3}\right), h(b)\right)$ and
$\overrightarrow{OB_{28}} =\left(h\left(\frac{2+b}{3}\right),
h(b)\right)$. We can not determine $\overrightarrow{OB_{22}}$ from
the top part of Figure~\ref{fig:18} only. Thus, denote $c\in (0,
1)$ such that $\overrightarrow{OB_{22}} = (h(c), h(3c-2))$.

Notice that $g$ has kinks at $D_{11}$, $D_{19}$ and $D_{22}$.
Since $B_{22}D_{22}$ is a horizontal line, then
$\overrightarrow{OD_{22}} = (h(c), h(2-2c))$. Since we can not
determine neither $D_{11}$ not $D_{19}$ more, denote $d\in (0, 1)$
such that $\overrightarrow{OD_{11}} = (h(d), h(2d))$. Now
$\overrightarrow{OD_{19}} = (h(1-d), h(2d))$.

Now calculate $\overrightarrow{OC_4} =
\left(h\left(\frac{2a}{3}\right), h(2a)\right)$,
$\overrightarrow{OC_6} = \left(h\left(\frac{2-2a}{3}\right),
h(2a)\right)$, $\overrightarrow{OC_{11}} = (h(2d), h(6d-2))$ and
$\overrightarrow{OC_{14}} = \left(h\left(\frac{2+2a}{3}\right),
h(2a)\right)$.

Since $\overrightarrow{OD_{11}} = \overrightarrow{OA_4}$,
$\overrightarrow{OD_{19}} = \overrightarrow{OA_6}$,
$\overrightarrow{OD_{22}} = \overrightarrow{OA_7}$,
$\overrightarrow{OB_7} = \overrightarrow{OC_4}$,
$\overrightarrow{OB_{13}} = \overrightarrow{OC_6}$,
$\overrightarrow{OB_{22}} = \overrightarrow{OC_{11}}$ and
$\overrightarrow{OB_{28}} = \overrightarrow{OC_{14}}$. Thus,
$$\overrightarrow{OD_{11}} =
\overrightarrow{OA_4}\Longrightarrow (h(d), h(2d))= (h(a),
h(2a)),$$ $$\overrightarrow{OD_{19}} = \overrightarrow{OA_6}
\Longrightarrow (h(1-d), h(2d)) = (h(1-a), h(2a)),$$
$$\overrightarrow{OD_{22}} = \overrightarrow{OA_7}
\Longrightarrow (h(c), h(2-2c)) = (h(b), h(2-2b)),$$
$$\overrightarrow{OB_7} =
\overrightarrow{OC_4} \Longrightarrow
\left(h\left(\frac{b}{3}\right), h(b)\right) =
\left(h\left(\frac{2a}{3}\right), h(2a)\right),$$
$$\overrightarrow{OB_{13}} =
\overrightarrow{OC_6} \Longrightarrow
\left(h\left(\frac{2-b}{3}\right), h(b)\right) =
\left(h\left(\frac{2-2a}{3}\right), h(2a)\right),$$
$$\overrightarrow{OB_{22}} =
\overrightarrow{OC_{11}} \Longrightarrow (h(c), h(3c-2)) =(h(2d),
h(6d-2))$$ and
$$\overrightarrow{OB_{28}} =
\overrightarrow{OC_{14}} \Longrightarrow
\left(h\left(\frac{2+b}{3}\right), h(b)\right) =
\left(h\left(\frac{2+2a}{3}\right), h(2a)\right).$$

Since $h$ increase then $$\left\{ \begin{array}{l}a =d,\\
b =2d,\\
c = 2d.
\end{array}\right.$$
Thus, $$ g:\, (0, 0)\rightarrow A_4(h(a), h(2a))\rightarrow
A_5\left(h\left(\frac{1}{2}\right), 1\right)\rightarrow
$$$$\rightarrow A_6(h(1-a), h(2a)) \rightarrow A_7(h(2a),
h(2-4a))\rightarrow (1,0).
$$

By Remark~\ref{rem:3.27}, the first kink of $h$ is $(2a, h(2a))$.
Since $2a\geq \frac{1}{2}$ then, by Lemma~\ref{lema:3.24},
$$ h:\,
(0,0)\rightarrow %\left(\frac{1}{2},
%h\left(\frac{1}{2}\right)\right) \rightarrow
\left(2a,
h(2a)\right)\rightarrow (1,1).
$$
Denote $\widetilde{a} = h(a)$. Thus, $$ h\left(\frac{1}{2}\right)
=\frac{1}{2}\cdot \frac{\widetilde{a}}{a}
$$

Since $1-a<2a$, then $$ h\left(1-a\right)=(1-a)\cdot
\frac{\widetilde{a}}{a}
$$

Since $2-4a<2a$, then $$ h\left(2-4a\right)=(2-4a)\cdot
\frac{\widetilde{a}}{a}.
$$

Thus, $$ g:\, (0, 0)\rightarrow A_4(\widetilde{a},
2\widetilde{a})\rightarrow A_5\left(\frac{1}{2}\cdot
\frac{\widetilde{a}}{a}, 1\right)\rightarrow
$$$$\rightarrow A_6\left((1-a)\cdot
\frac{\widetilde{a}}{a}, 2\widetilde{a}\right) \rightarrow
A_7\left(2\widetilde{a}, (2-4a)\cdot
\frac{\widetilde{a}}{a}\right)\rightarrow (1,0).
$$
Denote $\widetilde{b} =\frac{1}{2}\cdot \frac{\widetilde{a}}{a}$,
whence $$ g:\, (0, 0)\rightarrow A_4(\widetilde{a},
2\widetilde{a})\rightarrow A_5\left(\widetilde{b},
1\right)\rightarrow
$$$$
\rightarrow A_6\left(2\widetilde{b} -\widetilde{a},
2\widetilde{a}\right) \rightarrow A_7\left(2\widetilde{a}, 4\cdot
(\widetilde{b}-\widetilde{a})\right)\rightarrow (1,0).
$$

\begin{proposition}
Let $g$ be a piecewise linear map. It is restricted by the lattice
from Figure~\ref{fig:18} if and only if there exist $a$ and $b$
such that
$$ g:\, (0, 0)\rightarrow A_4(a,
2a)\rightarrow A_5\left(b, 1\right)\rightarrow
$$$$
\rightarrow A_6\left(2b -a, 2a\right) \rightarrow A_7\left(2a,
4\cdot (b-a)\right)\rightarrow (1,0).
$$ Moreover, in this case the map
$$ h:\,
(0,0)\rightarrow \left(\frac{a}{b}, 2a\right)\rightarrow (1,1).
$$ provides the conjugation from the tent map to $g$.
\end{proposition}

\section{The main hypothesis}

We can generalize the descriptions of the maps $g$, restricted by
lattices from Figures~\ref{fig:09}, \ref{fig:11} and~\ref{fig:18}
as follows.

\begin{theorem}\label{th:3}
Suppose that $g$ is the general solution of some lattice, which
describes commutativity of a piecewise linear unimodal map with
piecewise linear surjective map $\psi$. Let the number of maximal
parts of monotonicity of of $\psi$ be not a power of $2$, i.e. let
$\psi$ be not an iteration of $g$. Then:

1. The increasing part of $g$ can be arbitrary such that the
derivative of $g$ at $0$ equals to $2$;

2. The decreasing part of $g$ is completely determined by the
increasing part of $g$.
\end{theorem}

Notice, that Theorem~\ref{th:3} follows
from~\cite[Th.~1]{1808.03622.2}. But we believe that it should be
more elegant proof of Theorem~\ref{th:3}, using the techniques,
which was introduced in this work.

 \pagestyle{empty}
\bibliography{Ds-Bib}{}
\bibliographystyle{makar}

\newpage
\tableofcontents

\end{document}